\documentclass{amsart}

\usepackage{ae,aecompl}
\usepackage{esint}
\usepackage{graphicx}
\usepackage[all,cmtip]{xy}
\usepackage{amsmath, amscd}
\usepackage{booktabs}
\usepackage{amsthm}
\usepackage{amssymb}
\usepackage{amsfonts}
\usepackage{qsymbols}
\usepackage{latexsym}
\usepackage{mathrsfs}
\usepackage{cite}
\usepackage{color}
\usepackage{url}
\usepackage{enumerate}
\usepackage{undertilde}

\usepackage{verbatim}
\usepackage[draft=false, colorlinks=true]{hyperref}

\allowdisplaybreaks

\newcommand {\A}{{\mathcal{A}}}

\newcommand {\B}{{\mathcal{B}}}

\newcommand {\bmo}{\mathrm{bmo}}

\newcommand {\C}{{\mathbb C}}
\newcommand {\Ca}{{\mathcal{C}}}

\newcommand {\Da}{\mathcal{D}}
\newcommand {\dec}{\mathrm{dec}}

\newcommand {\ud}{\mathrm{d}}

\newcommand {\veps}{\varepsilon}

\newcommand {\thetat}{\tilde{\theta}}
\newcommand {\F}{\mathcal{F}}

\newcommand {\HT}{\mathcal{H}}
\newcommand {\Hp}{\mathcal{H}^{p}_{FIO}(\Rn)}
\newcommand {\HpM}{\mathcal{H}^{p}_{FIO}(M)}
\newcommand {\Hps}{\mathcal{H}^{s,p}_{FIO}(\Rn)}
\newcommand {\HpsM}{\mathcal{H}^{s,p}_{FIO}(M)}

\newcommand {\inj}{\mathrm{inj}}

\newcommand {\ka}{\kappa}

\newcommand {\la}{\lambda}
\newcommand {\rb}{\rangle}
\newcommand {\lb}{{\langle}}
\newcommand {\La}{{\mathcal{L}}}
\newcommand {\loc}{{\mathrm{loc}}}

\newcommand {\N}{{{\mathbb N}}}

\newcommand {\ph}{\varphi}
\newcommand {\psit}{\tilde{\psi}}

\newcommand {\R}{\mathbb R}

\newcommand {\Rn}{\mathbb{R}^{n}}

\newcommand {\rank}{\mathrm{rank}}

\newcommand {\supp}{\mathrm{supp}}

\newcommand {\Sp}{S^{*}\Rn}

\newcommand {\Sw}{\mathcal{S}}

\newcommand {\Util}{\tilde{U}}

\newcommand {\w}{\omega}

\newcommand\WF{\operatorname{WF}}

\newcommand {\Z}{\mathbb Z}
\newcommand {\Zn}{\mathbb{Z}^{n}}

\newcommand {\vanish}[1]{\relax}

\newcommand{\wh}{\widehat}
\newcommand{\wt}{\widetilde}

\DeclareMathOperator{\Real}{Re}

\newtheorem{theorem}{Theorem}[section]
\newtheorem{lemma}[theorem]{Lemma}
\newtheorem{proposition}[theorem]{Proposition}
\newtheorem{corollary}[theorem]{Corollary}

\theoremstyle{definition}
\newtheorem{definition}[theorem]{Definition}
\newtheorem{remark}[theorem]{Remark}

\numberwithin{equation}{section}
\protected\def\ignorethis#1\endignorethis{}
\let\endignorethis\relax

\title[Local smoothing and Hardy spaces for FIOs on manifolds]
	{Local smoothing and Hardy spaces for Fourier integral operators on manifolds}

\author{Naijia Liu}
\address{Department of Mathematics,
    Sun Yat-sen University,
    Guangzhou, 510275,
    P.R.~China}
\email{liunj@mail2.sysu.edu.cn}
	
\author{Jan Rozendaal}
\address{Institute of Mathematics, Polish Academy of Sciences\\
Ul.~\'{S}niadeckich 8\\
00-656 Warsaw\\
Poland}
\email{jrozendaal@impan.pl}

\author{Liang Song}
\address{Department of Mathematics,
    Sun Yat-sen University,
    Guangzhou, 510275,
    P.R.~China}
\email{songl@mail.sysu.edu.cn}

\author{Lixin Yan}
\address{Department of Mathematics,
    Sun Yat-sen University,
    Guangzhou, 510275,
    P.R.~China}
\email{mcsylx@mail.sysu.edu.cn}



\thanks{The research leading to these results has received funding from the Norwegian Financial Mechanism 2014-2021, grant 2020/37/K/ST1/02765. This research was funded in part by the National Science Center, Poland, grant 2021/43/D/ST1/00667. N.J.~Liu is supported by  China Postdoctoral Science Foundation (No.~2022M723673). L.~Song is supported by  NNSF of China (No.~12071490).  L.~Song and L.~Yan are supported  by National Key R$\&$D Program of China 2022YFA1005700.}

\begin{document}

\begin{abstract}
We introduce the Hardy spaces for Fourier integral operators on Riemannian manifolds with bounded geometry. We then use these spaces to obtain improved local smoothing estimates for Fourier integral operators satisfying the cinematic curvature condition, and for wave equations on compact manifolds. The estimates are essentially sharp, for all $2<p<\infty$ and on each compact manifold. We also apply our local smoothing estimates to nonlinear wave equations with initial data outside of $L^{2}$-based Sobolev spaces.
\end{abstract}
	
\maketitle	

\tableofcontents

\section{Introduction}\label{sec:intro}

This article concerns fixed-time and local smoothing estimates for wave equations and Fourier integral operators on manifolds, and associated function spaces.

\subsection{Setting}\label{subsec:setting}

Let $(M,g)$ be a smooth compact Riemannian manifold of dimension $n\geq2$ with Laplace--Beltrami operator $\Delta_{g}$, and consider the Cauchy problem
\begin{equation}\label{eq:Cauchyintro}
\begin{cases}(\partial_{t}^{2}-\Delta_{g})u(x,t)=0,
\\u(x,0)=u_{0}(x), \ \partial_{t}u(x,0)=u_{1}(x),
\end{cases}
\end{equation}
on $M\times\R$, for suitable initial data $u_{0}$ and $u_{1}$ on $M$. It is a natural problem to determine, for a given $1\leq p\leq\infty$, conditions on $u_{0}$ and $u_{1}$ which guarantee that the solution $u(t):=u(\cdot,t)$ to \eqref{eq:Cauchyintro} at a fixed time $t\in\R$ is an element of $L^{p}(M)$.

It was shown by Seeger, Sogge and Stein in \cite{SeSoSt91} that $u(t)\in L^{p}(M)$ if $1<p<\infty$, $u_{0}\in W^{2s(p),p}(M)$ and $u_{1}\in W^{2s(p)-1,p}(M)$, and these exponents are sharp for all but a discrete set of times, on each compact manifold. Here and throughout, we write
\[
s(p):=\frac{n-1}{2}\Big|\frac{1}{p}-\frac{1}{2}\Big|
\]
for $p\in[1,\infty]$. This is a variable-coefficient version of the corresponding results from \cite{Peral80,Miyachi80a} for the Euclidean Laplacian $\Delta_{g}=\Delta$ on $M=\Rn$.

Although \cite{SeSoSt91} determined the optimal fixed-time $L^{p}$ regularity for wave equations with smooth coefficients, it was observed by Sogge in \cite{Sogge91} that, for $2<p<\infty$ and $\Delta_{g}=\Delta$ on $M=\Rn$, less regularity is required of the initial data to guarantee that $u\in L^{p}_{\loc}(\R;L^{p}(\Rn))$. In fact, the \emph{local smoothing conjecture} asserts that $u\in L^{p}_{\loc}(\R;L^{p}(\Rn))$ if $u_{0}\in W^{\sigma(p)+\veps,p}(\Rn)$ and $u_{1}\in W^{\sigma(p)-1+\veps,p}(\Rn)$ for some $\veps>0$, where $\sigma(p):=0$ for $2<p\leq 2n/(n-1)$, and $\sigma(p):=2s(p)-1/p$ for $2n/(n-1)<p<\infty$. Again, the exponent $\sigma(p)$ cannot be improved. Local smoothing estimates on compact manifolds were subsequently obtained in \cite{MoSeSo93}. However, the local smoothing conjecture for the Euclidean wave equation is still open in dimensions $n\geq3$.

The most relevant work for the present article is the recent contribution \cite{BeHiSo20} by Beltran, Hickman and Sogge. They showed that $u\in L^{p}_{\loc}(\R;L^{p}(M))$ if $u_{0}\in W^{d(p)+\veps,p}(M)$ and $u_{1}\in W^{d(p)-1+\veps,p}(M)$ for some $\veps>0$, on each compact manifold $M$. Here
\begin{equation}\label{eq:dp}
d(p):=\begin{cases}
s(p)&\text{if }2\leq p<\frac{2(n+1)}{n-1},\\
2s(p)-\frac{1}{p}&\text{if }\frac{2(n+1)}{n-1}\leq p<\infty.
\end{cases}
\end{equation}
This is a variable-coefficient version of the corresponding result due to Bourgain and Demeter \cite{Bourgain-Demeter15} for the Euclidean wave equation. For $p\geq 2(n+1)/(n-1)$ one has $d(p)=\sigma(p)$, which means that the exponent $d(p)$ is optimal for such $p$.

\subsection{Local smoothing for wave equations}\label{subsec:waveintro}

Even though the exponent $d(p)$ is sharp for $p\geq 2(n+1)/(n-1)$, it was observed in \cite{Rozendaal22b} that the local smoothing estimates from \cite{Bourgain-Demeter15}, for the Euclidean wave equation on $\Rn$, can be improved by replacing the spaces of initial data $W^{d(p)+\veps,p}(\Rn)$ and $W^{d(p)-1+\veps,p}(\Rn)$ by the Hardy spaces for Fourier integral operators $\HT^{d(p)-s(p)+\veps,p}_{FIO}(\Rn)$ and $\HT^{d(p)-s(p)-1+\veps,p}_{FIO}(\Rn)$, respectively. The main result of the present article is the following version of \cite{Rozendaal22b} on a compact manifold $M$.

\begin{theorem}\label{thm:waveintro}
Let $p\in(2,\infty)$ and $\veps,t_{0}>0$. Then there exists a $C\geq0$ such that, for all $u_{0}\in \HT^{d(p)-s(p)+\veps,p}_{FIO}(M)$ and $u_{1}\in \HT^{d(p)-s(p)-1+\veps,p}_{FIO}(M)$, the solution $u$ to \eqref{eq:Cauchyintro} satisfies
\[
\Big(\int_{-t_{0}}^{t_{0}}\|u(t)\|_{L^{p}(M)}^{p}\ud t\Big)^{1/p}\leq C\big(\|u_{0}\|_{\HT^{d(p)-s(p)+\veps,p}_{FIO}(M)}+\|u_{1}\|_{\HT^{d(p)-s(p)-1+\veps,p}_{FIO}(M)}).
\]
\end{theorem}

Theorem \ref{thm:waveintro} is a special case of Theorem \ref{thm:wave}, which considers more general Sobolev regularity of the solution. Moreover, the index $d(p)-s(p)$ is sharp, for all $2<p<\infty$ and on each compact manifold $M$, by Theorem \ref{thm:sharpgen}.

The Hardy spaces for Fourier integral operators $\HT^{s,p}_{FIO}(M)$ are introduced in Definition \ref{def:HpFIOman}, for $p\in[1,\infty]$ and $s\in\R$. For $M=\Rn$, these spaces coincide with the ones from \cite{HaPoRo20}, which in turn are an extension of a construction due to Smith \cite{Smith98a} for $p=1$. They are invariant under compactly supported Fourier integral operators (FIOs) of order zero associated with a local canonical graph, cf.~Theorem \ref{thm:FIObdd}. Moreover, for $1<p<\infty$ the Sobolev embeddings
\begin{equation}\label{eq:Sobolevintro}
W^{s+s(p),p}(M)\subseteq \HT^{s,p}_{FIO}(M)\subseteq W^{s-s(p),p}(M)
\end{equation}
hold, with the natural modifications involving the local Hardy space $\HT^{1}(M)$ for $p=1$, and $\bmo(M)$ for $p=\infty$, cf.~Theorem \ref{thm:SobolevM}. In particular, Theorem \ref{thm:waveintro} recovers the main result of \cite{BeHiSo20}. However, since the exponents in \eqref{eq:Sobolevintro} are sharp, Theorem \ref{thm:waveintro} is in fact a strict improvement of the main result of \cite{BeHiSo20}. And since $d(p)=\sigma(p)$ for $p\geq 2(n+1)/(n-1)$, Theorem \ref{thm:waveintro} improves upon the bounds in the local smoothing conjecture for such $p$.

For $n=2$, the local smoothing conjecture for the Euclidean wave equation was proved by Guth, Wang and Zhang in \cite{GuWaZh20}, and a variable-coefficient version of their result, on compact surfaces, was obtained in \cite{GaLiMiXi23}. Note that
\[
d(p)-s(p)=0=\sigma(p) \quad\text{for}\quad2<p\leq 2\frac{n}{n-1},
\]
and
\[
d(p)-s(p)=0<\sigma(p)<d(p)\quad \text{for}\quad 2\frac{n}{n-1}<p<2\frac{n+1}{n-1}.
\]
Hence, for $n=2$ and $2<p<6$, due to the sharpness of the embeddings  in \eqref{eq:Sobolevintro}, Theorem \ref{thm:waveintro} neither follows from the bounds in \cite{GaLiMiXi23}, nor does Theorem \ref{thm:waveintro} imply those bounds. In higher dimensions, certain Laplace--Beltrami operators on compact manifolds satisfy weaker local smoothing estimates than the Euclidean wave equation does, cf.~\cite{Minicozzi-Sogge97}. We refer to \cite[Conjecture 23]{BeHiSo21} for a local smoothing conjecture for wave equations on general compact manifolds. For $p<2(n+1)/(n-1)$, Theorem \ref{thm:waveintro} neither follows from that conjecture, nor does the converse hold.

In \cite{Lee-Seeger13}, local smoothing estimates without $\veps$ loss were obtained in the $L^{p}$ scale for $n\geq 4$ and $p>2\frac{n-1}{n-3}>2\frac{n+1}{n-1}$. This is a variable-coefficient version of a similar result for the Euclidean wave equation from \cite{HeNaSe11}. It is not clear whether one may also let $\veps=0$ in Theorem \ref{thm:waveintro} for certain $p$. For $p>2\frac{n-1}{n-3}$, by the sharpness of the exponents in \eqref{eq:Sobolevintro}, the conclusion of Theorem \ref{thm:waveintro} neither follows from the local smoothing result in \cite{Lee-Seeger13}, nor does the converse hold.

We also obtain new fixed-time estimates for \eqref{eq:Cauchyintro}. Namely, in Section \ref{sec:wave} we show that, in the setting of Theorem \ref{thm:waveintro},
\[
u\in C(\R;\HT^{d(p)-s(p)+\veps,p}_{FIO}(M))\cap L^{p}_{\loc}(\R;L^{p}(M)).
\]
Due to \eqref{eq:Sobolevintro}, the first inclusion improves upon the fixed-time bounds in \cite{SeSoSt91}, and in particular \eqref{eq:Cauchyintro} is well-posed on the Hardy spaces for Fourier integral operators.

\subsection{Nonlinear wave equations}\label{subsec:nonlinear}

Local smoothing estimates from \cite{Rozendaal22b} were used in \cite{Rozendaal-Schippa23} to prove new well-posedness results for nonlinear wave equations with slowly decaying initial data on $\Rn$, building on similar work for the Schr\"{o}dinger equation from \cite{Schippa22}. In turn, the present article yields variable-coefficient versions of \cite{Rozendaal-Schippa23}, for nonlinear wave equations on compact manifolds with initial data outside of $L^{2}$-based Sobolev spaces. 

Consider the cubic nonlinear wave equation on $M\times\R$:
\begin{equation}\label{eq:nonlinear}
\begin{cases}(\partial_{t}^{2}-\Delta_{g})u(x,t)=\pm |u(x,t)|^{2}u(x,t),
\\u(x,0)=u_{0}(x), \ \partial_{t}u(x,0)=u_{1}(x).
\end{cases}
\end{equation}
For the Euclidean wave equation on $\R^{2}$, the following result is contained in \cite[Theorem 1.2]{Rozendaal-Schippa23}. Instead, here we consider compact Riemannian surfaces.

\begin{theorem}\label{thm:nonlinear}
Suppose that $n=2$, and let $\veps,t_{0}>0$. Then \eqref{eq:nonlinear} is quantitatively well posed with initial data space
\begin{equation}\label{eq:initial}
X:=(\HT^{\veps,6}_{FIO}(M)+W^{1/2,2}(M))\times(\HT^{\veps-1,6}_{FIO}(M)+ W^{-1/2,2}(M))
\end{equation}
and solution space
\begin{equation}\label{eq:solutionspace}
S_{t_{0}}:=L^{4}\big([0,t_{0}];L^{6}(M)\big) \cap C\big([0,t_{0}];\HT^{\varepsilon,6}_{FIO}(M) + W^{1/2,2}(M)\big).
\end{equation}
\end{theorem}

We refer to Section \ref{subsec:wavenonlinear} for the definition of quantitative well-posedness. Loosely speaking, it implies that there exists a $\delta=\delta(t_{0})>0$ such that, if $\|(u_{0},u_{1})\|_{X}<\delta$, then \eqref{eq:nonlinear} has a unique solution $u\in S_{t_{0}}$, and this solution depends analytically on the initial data. Moreover, in Theorem \ref{thm:nonlinear}, for all $(u_{0},u_{1})\in X$ there exists a $t_{0}>0$ such that there is a unique solution $u\in S_{t_{0}}$ to \eqref{eq:nonlinear}.

Theorem \ref{thm:nonlinear} is contained in Theorem \ref{thm:nonlinearmain}, with a proof which is analogous to that for the Euclidean wave equation. More precisely, in Duhamel's formula for the solution to \eqref{eq:nonlinear}, one can use the invariance of $\HpsM$ under the solution operators to the linear equation to deal with the linear term, whereas Strichartz estimates suffice for the term involving the nonlinearity. The latter then leads to the Sobolev spaces $W^{s,2}(M)$ appearing in \eqref{eq:initial} and \eqref{eq:solutionspace}. In this manner one obtains local existence, and global existence holds for the defocusing equation
if $\veps>1/2$, cf.~Proposition \ref{prop:global}.

The terminology ``slowly decaying initial data" from \cite{Rozendaal-Schippa23,Schippa22}, which refers to the fact that initial data in $L^{p}(\Rn)$ may decay slower at infinity than an $L^{2}(\Rn)$ function does if $p>2$, does not seem appropriate on compact manifolds. Moreover, by H\"{o}lder's inequality one has $L^{p}(M)\subseteq L^{2}(M)$, and one can combine this embedding with Strichartz estimates to solve \eqref{eq:nonlinear} in $L^{4}([0,t_{0}];L^{6}(M))$. However, this procedure works for
\[
(u_{0},u_{1})\in W^{1/2,6}(M)\times W^{-1/2,6}(M)\subsetneq W^{1/2,2}(M)\times W^{-1/2,2}(M),
\]
whereas Theorem \ref{thm:nonlinear} applies when $(u_{0},u_{1})\in W^{1/6+\veps,6}(M)\times W^{-5/6+\veps,6}(M)$, due to \eqref{eq:Sobolevintro}. In particular, the initial data space in \eqref{eq:initial} is larger than what one obtains by using Strichartz estimates in this manner. In fact, due to the sharpness of the embeddings in \eqref{eq:Sobolevintro}, Theorem \ref{thm:nonlinear} applies to even rougher initial data, including certain data in Sobolev spaces of negative order. 

Moreover, one of the advantages of working with the Hardy spaces for FIOs is that they are invariant under the solution operators to the linear flow, cf.~Theorem \ref{thm:wave}. This in turn leads to the fixed-time invariance of the nonlinear flow which is captured by the space of continuous functions in \eqref{eq:solutionspace}.

\subsection{Local smoothing for Fourier integral operators}\label{subsec:localsmoothFIO}

As in \cite{MoSeSo93,Lee-Seeger13,BeHiSo20,GaLiMiXi23}, we derive local smoothing for wave equations from bounds for FIOs whose canonical relation satisfies the so-called \emph{cinematic curvature condition} from \cite{Sogge91}. This condition is a homogeneous counterpart of the Carleson--Sj\"{o}lin condition for (nonhomogeneous) oscillatory integrals (see \cite{Carleson-Sjolin72,Hormander73}), and it captures, on the level of FIOs, some of the underlying curvature and dispersion effects which govern the phenomenon of local smoothing.

Let $N$ be another compact Riemannian manifold, of dimension $n+1$. For $m\in\R$, we denote by $I^{m-1/4}(M,N;\Ca)$ the class of FIOs of order $m-1/4$ associated with a canonical relation $\Ca$ from $T^{*}M$ to $T^{*}N$. We refer to Section \ref{subsec:FIOs} for more this class of operators and on the cinematic curvature condition. A special case of our main result on local smoothing for Fourier integral operators, Theorem \ref{thm:localsmoothFIO}, is as follows.

\begin{theorem}\label{thm:localsmoothFIOintro}
Let $p\in(2,\infty)$, $\veps>0$ and $m\in\R$. Let $T\in I^{m-1/4}(M,N;\Ca)$, for $\Ca$ a canonical relation satisfying the cinematic curvature condition. Then
\[
T:\HT^{d(p)-s(p)+\veps+m,p}_{FIO}(M)\to L^{p}(N)
\]
is bounded.
\end{theorem}

The connection between Theorems \ref{thm:localsmoothFIOintro} and \ref{thm:waveintro} comes from the fact that the solution $u$ to \eqref{eq:Cauchyintro} is given by $u=T_{0}u_{0}+T_{1}u_{1}$, where $T_{j}\in I^{-j-1/4}(M,M\times\R;\Ca)$ for a $\Ca$ satisfying the cinematic curvature condition, $j\in\{0,1\}$.

Again, the index $d(p)-s(p)$ is sharp, for all $2<p<\infty$ and for each $T$ which is non-characteristic somewhere, by Theorem \ref{thm:sharpgen}.

By the sharpness of the embeddings in \eqref{eq:Sobolevintro}, Theorem \ref{thm:localsmoothFIOintro} is a strict improvement of the local smoothing result for FIOs in \cite{BeHiSo20}. In particular, if $n$ is odd, then Theorem \ref{thm:localsmoothFIOintro} improves upon the local smoothing conjecture for FIOs from \cite[Conjecture 24]{BeHiSo21}. For $n$ even and $2<p<2(n+1)/(n-1)$, Theorem \ref{thm:localsmoothFIOintro} neither follows from this conjecture nor implies it. A special case concerns the case where $n=2$, in \cite{GaLiMiXi23}. Similarly, for $n\geq 4$ and $p>2(n-1)/(n-3)$, Theorem \ref{thm:localsmoothFIOintro} neither follows from the local smoothing result for FIOs without $\veps$ loss in \cite{Lee-Seeger13}, nor does the converse hold.

\subsection{Wolff-type inequalities}\label{subsec:Wolffintro}

The local smoothing estimates for the Euclidean wave equation in \cite{Rozendaal22b} were proved by connecting the Hardy spaces for FIOs to the $\ell^{p}$ decoupling inequality for the light cone from \cite{Bourgain-Demeter15}. In fact, local smoothing estimates were obtained for a more general class of constant-coefficient operators using a more general class of $\ell^{p}$ decoupling inequalities, from \cite{Bourgain-Demeter17}. Similarly, the local smoothing estimates in this article are proved using a variable-coefficient analog of the $\ell^{p}$ decoupling inequality, the variable-coefficient Wolff-type inequalities from \cite{BeHiSo20} (see Theorem \ref{thm:BeHiSo}).

More precisely, instead of bounding the \emph{decoupling norm} that appears in those inequalities by the $W^{s,p}(\Rn)$ norm of the initial data for a suitable $s\in\R$, which is the approach that is used in \cite{BeHiSo20}, we bound the decoupling norm in terms of the $\HT^{s-s(p),p}_{FIO}(\Rn)$ norm of the initial data. By \eqref{eq:Sobolevintro}, this yields a stronger estimate. Moreover, as in the case of the Euclidean wave equation in \cite{Rozendaal22b}, the resulting inequalities are, in a quantified sense, equivalent to the Wolff-type inequalities. Hence, at least when restricted to dyadic frequency shells, the Hardy spaces for FIOs form the \emph{largest} space of initial data for which one can obtain local smoothing estimates using Wolff-type inequalities as in \cite{BeHiSo20}.

To bound the decoupling norm by the $\HT^{s-s(p),p}_{FIO}(\Rn)$ norm, one proceeds in a similar way as in \cite{Rozendaal22b}, using that the discrete and continuous dyadic-parabolic decompositions are equivalent on dyadic frequency shells (see Proposition \ref{prop:decouple}). On the other hand, the reverse inequality is more involved than that in \cite{Rozendaal22b}, due to the fact that the operators which appear in the variable-coefficient Wolff-type inequalities are neither invertible nor elliptic, unlike the Euclidean half-wave propagators. Instead, loosely speaking, we bound the $\HT^{s-s(p),p}_{FIO}(\Rn)$ norm of the initial data in terms of the right-hand side of the Wolff-type inequality wherever the FIO is microlocally invertible. We refer to Proposition \ref{prop:conversedecouple} (see also Remark \ref{rem:regularityf}) for the precise statement.

Hence Theorem \ref{thm:localsmoothFIOintro}, and its version in local coordinates in Corollary \ref{cor:decouple}, could be viewed as replacements of the Wolff-type inequalities. Moreover, whereas the discrete decoupling norm is different on each dyadic frequency shell and involves the FIO, the $\HT^{s-s(p),p}_{FIO}(\Rn)$ norm is defined using a continuous dyadic-parabolic decomposition that is the same on each dyadic frequency shell and is independent of the operator. Here it is worth mentioning that a microlocal version of the $\ell^{2}$ decoupling inequality, of a different nature, was obtained in \cite{IoLiXi22}. It would be of interest to determine the connection between the $\ell^{p}$ version of that inequality and the Hardy spaces for FIOs, or alternatively between that inequality and function spaces adapted to the $\ell^{2}$ decoupling inequality as in \cite{Rozendaal-Schippa23}.

\subsection{Hardy spaces for Fourier integral operators}\label{subsec:HpFIOintro}

The results discussed so far, and the proofs of Theorems \ref{thm:waveintro} and \ref{thm:localsmoothFIOintro}, rely mostly on estimates in local coordinates, FIO calculus and properties of the Hardy spaces for FIOs on $\Rn$. This material can be found in Sections \ref{sec:spacesRn}, \ref{sec:localsmooth}, \ref{sec:wave} and \ref{sec:sharpness}, and  Appendix \ref{sec:Wolff}. In particular, loosely speaking, for a compact manifold $M$, one has $u\in\HpsM$ if and only if $u\in\Hps$ in each coordinate chart, and geometric considerations do not play a significant role. Hence readers who are satisfied with such a definition and who are mainly interested in Theorems \ref{thm:waveintro} and \ref{thm:localsmoothFIOintro} may skip ahead to Section \ref{sec:localsmooth}, looking up results from other sections where needed.

However, apart from proving new local smoothing estimates for wave equations and Fourier integral operators, one of the aims of this article is to develop a proper theory of the Hardy spaces for Fourier integral operators on manifolds. In particular, both with future applications in mind and to unify the theory on $\Rn$ with that on compact manifolds, we introduce these spaces on Riemannian manifolds with bounded geometry.  More precisely, we consider $M$ with positive injectivity radius, and with the property that all covariant derivatives of the curvature tensor are bounded. This geometric setting
incorporates all compact manifolds and homogeneous spaces, as well as asymptotically Euclidean, asymptotically conic and asymptotically hyperbolic manifolds.
It has been used for microlocal analysis on non-compact manifolds before (see e.g.~\cite{BaGuIsMa22,Shubin92,Taylor09}).  On the other hand, connected components of $M$ do not necessarily have the doubling property when endowed with the Riemannian distance and measure (although they do have the local doubling property). This places our geometric setting outside of the scope of various results in harmonic analysis which hold on doubling metric measure spaces.

Our definition of $\HpsM$ is inspired by Triebel, who introduced the classical local Hardy spaces, and more general Triebel--Lizorkin spaces, on manifolds with bounded geometry, in \cite{Triebel86,Triebel87} (see also Taylor \cite{Taylor09}). We follow Triebel's template, using that the assumption of bounded geometry implies the existence of a suitable uniformly locally finite cover by geodesic coordinate charts (see Lemma \ref{lem:cover}). Then, loosely speaking, $u\in \HpsM$ if and only if $u\in\Hps$ in each such coordinate chart, and if the $\ell^{p}$ sum of the $\Hps$ norms of $u$ over these charts is finite (see Definition \ref{def:HpFIOman}). To show that one recovers the definition of $\Hps$ from \cite{Smith98a,HaPoRo20} when $M=\Rn$, we prove in Theorem \ref{thm:localization} a localization principle for $\Hps$.

This construction yields a space $\HpsM$ that is as one would expect on both compact manifolds and on $\Rn$. However, to derive the basic properties of these spaces, such as complex interpolaton, duality, Sobolev embeddings, the invariance under FIOs, and an atomic decomposition for $p=1$, it is convenient to work with $\ell^{p}$ spaces of $\Hps$-valued sequences. This operator-theoretic framework is introduced in Section \ref{sec:sequences}, and in Section \ref{sec:spacesmanifold} it allows us to directly derive properties of $\HpsM$ from those of $\Hps$. Our approach can be compared with the use of tent spaces for the analysis of $\Hp$ in \cite{HaPoRo20}, proving concrete estimates in local coordinates only where necessary and otherwise relying on existing theory.

We define the Hardy spaces for FIOs to consist of half densities. Doing so leads to a more convenient duality and $L^{2}$ theory, and it is typical for the theory of FIOs on manifolds. For readers unfamiliar with half densities, we note that elements of $\HpsM$ coincide with generalized functions on $M$, or with functionals on $C^{\infty}_{c}(M)$, after multiplication or division by the square root of the Riemannian density.

\subsection{Organization}

This article is organized as follows. In Section \ref{sec:spacesRn} we define the Hardy spaces for FIOs on $\Rn$, and we collect some of their basic properties. In Section \ref{sec:sequences} we develop an operator-theoretic framework using $\ell^{p}$ spaces of $\Hps$-valued sequences, which is then applied in Section \ref{sec:spacesmanifold} to introduce and study the Hardy spaces for FIOs on manifolds. Section \ref{sec:localsmooth} contains Theorem \ref{thm:localsmoothFIOintro}, as well as our version of the variable-coefficient Wolff-type inequalities, and in Section \ref{sec:wave} we prove Theorems \ref{thm:waveintro} and \ref{thm:nonlinear}. In Section \ref{sec:sharpness} we show that our local smoothing estimates are essentially sharp. In Appendix \ref{sec:background} we collect some background material on manifolds with bounded geometry, distributional densities and Fourier integral operators, for readers unfamiliar with this theory. Finally, Appendix \ref{sec:atomic} contains information on the atomic decomposition of $\HT^{s,1}_{FIO}(\Rn)$, and the main result of Appendix \ref{sec:Wolff} is an inequality involving the $\Hps$ norm and the decoupling norm.

\subsection{Notation and terminology}

Throughout, all manifolds are assumed to be infinitely smooth and without boundary. The cotangent bundle of a manifold $M$ is $T^{*}M$, and $o$ is the zero section of $T^{*}M$. The cosphere bundle of $\Rn$ is $\Sp:=\Rn\times S^{n-1}$.

The natural numbers are $\N=\{1,2,\ldots\}$, and $\Z_{+}:=\N\cup\{0\}$.

For $\xi\in\Rn$ we write $\lb\xi\rb=(1+|\xi|^{2})^{1/2}$, and $\hat{\xi}=\xi/|\xi|$ if $\xi\neq0$. We use multi-index notation, where $\partial_{\xi}=(\partial_{\xi_{1}},\ldots,\partial_{\xi_{n}})$ and $\partial^{\alpha}_{\xi}=\partial^{\alpha_{1}}_{\xi_{1}}\ldots\partial^{\alpha_{n}}_{\xi_{n}}$
for $\xi=(\xi_{1},\ldots,\xi_{n})\in\Rn$ and $\alpha=(\alpha_{1},\ldots,\alpha_{n})\in\Z_{+}^{n}$. We also write $\partial\chi$ for the Jacobian of a smooth map $\chi$ between open subsets of $\Rn$, and $\partial_{x\eta}^{2}\Phi$ is the mixed Hessian of a function $\Phi$ of the variables $x$ and $\eta$.

For $V\subseteq\Rn$ open we write $\Da(V)=C^{\infty}_{c}(V)$, and $\Da'(V)$ is the dual space of distributions on $V$. We write $\lb f,g\rb_{\Rn}$ for the standard distributional pairing between $f\in\Da'(V)$ and $g\in\Da(V)$, or simply $\lb f,g\rb$ when it is clear that we are working on $\Rn$. The Fourier transform of a tempered distribution $f\in\Sw'(\Rn)$ is denoted by $\F f$ or $\widehat{f}$, and the Fourier multiplier with symbol $\ph\in\Sw'(\Rn)$ is $\ph(D)$.

The space of bounded linear operators between Banach spaces $X$ and $Y$ is $\La(X,Y)$, and $\La(X):=\La(X,X)$. We let $|E|$ denote either the cardinality of a finite set or the measure of a subset of a measure space, depending on the context.

We write $f(s)\lesssim g(s)$ to indicate that $f(s)\leq Cg(s)$ for all $s$ and a constant $C>0$ independent of $s$, and similarly for $f(s)\gtrsim g(s)$ and $g(s)\eqsim f(s)$.

\section{Hardy spaces for Fourier integral operators on $\Rn$}\label{sec:spacesRn}

In this section we first recall the definition of the Hardy spaces for Fourier integral operators on $\Rn$. We then collect some of their basic properties, and most notably their invariance under suitable Fourier integral operators.

\subsection{Definitions}\label{subsec:defRn}

Throughout, for simplicity of notation, we set $\HT^{p}(\Rn):=L^{p}(\Rn)$ for $n\in\N$ and $p\in(1,\infty)$. Also, $\HT^{1}(\Rn)$ is the local real Hardy space, with norm
\[
\|f\|_{\HT^{1}(\Rn)}:=\|q(D)f\|_{L^{1}(\Rn)}+\|(1-q)(D)f\|_{H^{1}(\Rn)}
\]
for $f\in\HT^{1}(\Rn)$. Here and throughout, $q\in C^{\infty}_{c}(\Rn)$ is a fixed cut-off such that $q(\xi)=1$ for $|\xi|\leq 2$, and $H^{1}(\Rn)$ is the classical real Hardy space. Moreover, $\HT^{\infty}(\Rn):=\bmo(\Rn)$ is the dual of $\HT^{1}(\Rn)$, and we set $\HT^{s,p}(\Rn):=\lb D\rb^{-s}\HT^{p}(\Rn)$ for all $p\in[1,\infty]$ and $s\in\R$.

First consider the case $n=1$, where one simply sets $\HT^{s,p}_{FIO}(\R):=\HT^{s,p}(\R)$.

For the definition of $\Hps$ for $n\geq 2$, let $\ph\in C^{\infty}_{c}(\Rn)$ be non-negative and radial, such that $\ph(\xi)=0$ for $|\xi|>1$, and such that $\ph\equiv1$ in a neighborhood of zero. For $\w\in S^{n-1}$ a unit vector, $\sigma>0$ and $\xi\in\Rn\setminus\{0\}$, set $\ph_{\w,\sigma}(\xi):=c_{\sigma}\ph(\tfrac{\hat{\xi}-\w}{\sqrt{\sigma}})$, where $c_{\sigma}:=(\int_{S^{n-1}}\ph(\frac{e_{1}-\nu}{\sqrt{\sigma}})^{2}\ud\nu)^{-1/2}$ for $e_{1}=(1,0,\ldots,0)$ the first basis vector of $\Rn$.
Also set $\ph_{\w,\sigma}(0):=0$. Let $\Psi\in C^{\infty}_{c}(\Rn)$ be non-negative and radial, such that $\Psi(\xi)=0$ if $|\xi|\notin[1/2,2]$, and such that $\int_{0}^{\infty}\Psi(\sigma\xi)^{2}\frac{\ud\sigma}{\sigma}=1$ for all $\xi\neq 0$. Now set
\[
\ph_{\w}(\xi):=\int_{0}^{4}\Psi(\sigma\xi)\ph_{\w,\sigma}(\xi)\frac{\ud\sigma}{\sigma}
\]
for $\xi\in\Rn$. Some properties of these functions are as follows (see \cite[Remark 3.3]{Rozendaal21}):
\begin{enumerate}
\item\label{it:phiproperties1} For all $\w\in S^{n-1}$ and $\xi\neq0$ one has $\ph_{\w}(\xi)=0$ if $|\xi|<1/8$ or $|\hat{\xi}-\w|>2|\xi|^{-1/2}$.
\item\label{it:phiproperties2} For all $\alpha\in\Z_{+}^{n}$ and $\beta\in\Z_{+}$ there exists a $C_{\alpha,\beta}\geq0$ such that
\[
|(\w\cdot \partial_{\xi})^{\beta}\partial^{\alpha}_{\xi}\ph_{\w}(\xi)|\leq C_{\alpha,\beta}|\xi|^{\frac{n-1}{4}-\frac{|\alpha|}{2}-\beta}
\]
for all $\w\in S^{n-1}$ and $\xi\neq0$.
\item\label{it:phiproperties3}
There exists a radial $m\in S^{(n-1)/4}(\Rn)$ such that if $f\in\Sw'(\Rn)$ satisfies $\supp(\wh{f}\,)\subseteq \{\xi\in\Rn\mid |\xi|\geq1/2\}$, then $f=\int_{S^{n-1}}m(D)\ph_{\nu}(D)f\ud\nu$.
\end{enumerate}

We can now define the Hardy spaces for Fourier integral operators on $\Rn$.

\begin{definition}\label{def:HpFIO}
For $p\in[1,\infty)$, $\Hp$ consists of those $f\in\Sw'(\Rn)$ such that $q(D)f\in L^{p}(\Rn)$, $\ph_{\w}(D)f\in \HT^{p}(\Rn)$ for almost all $\w\in S^{n-1}$, and
\[
\Big(\int_{S^{n-1}}\|\ph_{\w}(D)f\|_{\HT^{p}(\Rn)}^{p}\ud\w\Big)^{1/p}<\infty,
\]
endowed with the norm
\begin{equation}\label{eq:HpFIOnorm}
\|f\|_{\HT^{p}_{FIO}(\Rn)}:=\|q(D)f\|_{L^{p}(\Rn)}+\Big(\int_{S^{n-1}}\|\ph_{\w}(D)f\|_{\HT^{p}(\Rn)}^{p}\ud\w\Big)^{1/p}.
\end{equation}
Moreover, $\HT^{\infty}_{FIO}(\Rn):=(\HT^{1}_{FIO}(\Rn))^{*}$. For $p\in[1,\infty]$ and $s\in\R$, $\HT^{s,p}_{FIO}(\Rn)$ consists of all $f\in\Sw'(\Rn)$ such that $\lb D\rb^{s}f\in\HT^{p}_{FIO}(\Rn)$, endowed with the norm
\[
\|f\|_{\HT^{s,p}_{FIO}(\Rn)}:=\|\lb D\rb^{s}f\|_{\Hp}.
\]
\end{definition}

In fact, this is not how the Hardy spaces for FIOs were originally defined in \cite{Smith98a} and \cite{HaPoRo20}, namely in terms of conical square functions and tent spaces over the cosphere bundle $\Sp$. That these definitions coincide, up to norm equivalence, was shown in \cite{FaLiRoSo23,Rozendaal21}. We will often rely on results from \cite{HaPoRo20}; this typically leads to harmless additional constants, arising from the equivalence of norms. We also note that, in \cite{Frey-Portal20}, function spaces related to the Hardy spaces for FIOs were defined by modifying the framework in Definition \ref{def:HpFIO}.

To explain the nomenclature ``Hardy spaces for FIOs," we emphasize the connection of these spaces to the classical local Hardy spaces, cf.~\eqref{eq:HpFIOnorm}, their invariance under FIOs, cf.~Proposition \ref{prop:FIORn}, and the fact that $\HT^{1}_{FIO}(\Rn)$ has an atomic decomposition, cf.~Theorem \ref{prop:atom}.

\subsection{Basic properties}\label{subsec:propRn}

Here we collect some of the basic properties of these spaces.

We first recall the Sobolev embeddings
\begin{equation}\label{eq:Sobolev}
\HT^{s+s(p),p}(\Rn)\subseteq \HT^{s,p}_{FIO}(\Rn)\subseteq \HT^{s-s(p),p}(\Rn)
\end{equation}
for $p\in[1,\infty]$ and $s\in\R$, which follow from \cite[Theorem 7.4]{HaPoRo20}.

Next, the Hardy spaces for FIOs form a complex interpolation scale. That is, let $p_{0},p_{1},p\in[1,\infty]$, $s_{0},s_{1},s\in\R$ and $\theta\in[0,1]$ be such that $\frac{1}{p}=\frac{1-\theta}{p_{0}}+\frac{\theta}{p_{1}}$ and $s=(1-\theta)s_{0}+\theta s_{1}$. Then, as follows from \cite[Proposition 6.7]{HaPoRo20},
\begin{equation}\label{eq:interpRn}
[\HT^{s_{0},p_{0}}_{FIO}(\Rn),\HT^{s_{1},p_{1}}_{FIO}(\Rn)]_{\theta}=\Hps
\end{equation}
with equivalent norms.

Moreover, by \cite[Proposition 6.8]{HaPoRo20}, for all $p\in[1,\infty)$ and $s\in\R$ one has
\begin{equation}\label{eq:dualRn}
(\Hps)^{*}=\HT^{-s,p'}_{FIO}(\Rn)
\end{equation}
with equivalent norms. Here the duality pairing is the standard distributional pairing $\lb f,g\rb_{\Rn}$ for $f\in\HT^{-s,p'}_{FIO}(\Rn)\subseteq\Sw'(\Rn)$ and $g\in\Sw(\Rn)\subseteq \Hps$.

By \cite[Proposition 6.6]{HaPoRo20}, the Schwartz functions are contained in $\Hps$ for all $p\in[1,\infty]$, and they lie dense for $p<\infty$. In fact, the following lemma improves the latter statement in a manner which will be useful in the next section.

\begin{lemma}\label{lem:Ccdense}
Let $p\in[1,\infty)$ and $s\in\R$. Then $C_{c}^{\infty}(\Rn)$ is dense in $\Hps$.
\end{lemma}
\begin{proof}
Since $\Sw(\Rn)\subseteq \Hps$ is dense, it suffices to approximate Schwartz functions by compactly supported functions in the $\Hps$ norm. However, by \eqref{eq:Sobolev}, $\|f\|_{\Hps}\lesssim \|f\|_{W^{N,p}(\Rn)}$ for all $f\in\Sw(\Rn)$ and $N\in\N$ with $N>s+s(p)$. Hence we need only approximate $f\in\Sw(\Rn)$ by $C^{\infty}_{c}(\Rn)$ functions in the $W^{N,p}(\Rn)$ norm. In turn, to this end one may choose the approximants $f_{\veps}(x):=\chi(\veps x)f(x)$ for $\veps>0$ and $x\in\Rn$, where $\chi\in C^{\infty}_{c}(\Rn)$ is such that $\chi(0)=1$.
\end{proof}

Finally, we note that $\HT^{s,1}_{FIO}(\Rn)$ has an atomic decomposition, in terms of atoms that are supported on anisotropic balls. Appendix \ref{sec:atomic} contains more information about this decomposition.

\subsection{Boundedness of Fourier integral operators}\label{subsec:FIORn}

In this subsection we collect some results on the boundedness of Fourier integral operators on $\Hps$. 

The main result of this subsection concerns compactly supported FIOs. For more on the class $I^{m}(\Rn,\Rn;\Ca)$ of Fourier integral operators of order $m$ associated with a canonical relation $\Ca$, we refer to Section \ref{subsec:FIOs}. 

\begin{proposition}\label{prop:FIORn}
Let $T\in I^{m}(\Rn,\Rn;\Ca)$, for $m\in\R$ and $\Ca$ a local canonical graph, and suppose that $T$ is compactly supported. Then
\[
T:\HT^{s+m,p}_{FIO}(\Rn)\to\Hps
\]
is bounded for all $p\in[1,\infty]$ and $s\in\R$.
\end{proposition}
\begin{proof}
For $m=s=0$, the result is contained in \cite[Theorem 6.10]{HaPoRo20}. The same techniques can be applied for general $m,s\in\R$, at least when combined with the theory of weighted tent spaces as used in \cite{Hassell-Rozendaal23}. Instead, here we use the composition theorem for FIOs to easily reduce to the case where $m=s=0$.

We have to show that
\[
\lb D\rb^{s}T\lb D\rb^{-s-m}:\Hp\to\Hp
\]
is bounded. To this end, we first use the compact support assumption on $T$ to write $T=\rho_{1}T\rho_{2}$, for suitable $\rho_{1},\rho_{2}\in C^{\infty}_{c}(\Rn)$. By Lemma \ref{lem:propersupp}, we have
\[
\lb D\rb^{s}\rho_{1}=S_{1}+R_{1}\quad\text{and}\quad
\rho_{2}\lb D\rb^{-s-m}=S_{2}+R_{2},
\]
for compactly supported pseudodifferential operators $S_{1}$ and $S_{2}$ of orders $s$ and $-s-m$, respectively, and smoothing operators $R_{1},R_{2}:\Sw'(\Rn)\to\Sw(\Rn)$. Hence
\[
\lb D\rb^{s}T\lb D\rb^{-s-m}=\lb D\rb^{s}\rho_{1}T\rho_{2}\lb D\rb^{-s-m}=S_{1}TS_{2}+R,
\]
where $R$ is smoothing. Moreover, the composition theorem for FIOs implies that $S_{1}TS_{2}$ is a compactly supported FIO of order zero. Now \cite[Theorem 6.10]{HaPoRo20} concludes the proof.
\end{proof}

As a corollary, we obtain a corresponding statement for FIOs that have a convenient oscillatory integral representation. For the standard Kohn--Nirenberg symbol class $S^{m}(\Rn\times\Rn)$, see \eqref{eq:Kohn-Nirenberg}. 

\begin{corollary}\label{cor:stanform}
Let $m\in\R$ and $a\in S^{m}(\R^{n}\times \R^{n})$, and let $V\subseteq \R^{n}\times(\R^{n}\setminus\{0\})$ be open, conic and such that $(x,\eta)\in V$ whenever $(x,\eta)\in\supp(a)$ and $\eta\neq0$. Let $\Phi\in C^{\infty}(V)$ be real-valued and positively homogeneous of degree $1$ in the fiber variable, and such that $\det\partial_{x\eta}^{2}\Phi(x,\eta)\neq 0$ for all $(x,\eta)\in V$. Set
\[
Tf(x):=\int_{\R^{n}}e^{i\Phi(x,\eta)}a(x,\eta)\wh{f}(\eta)\ud \eta
\]
for $f\in\Sw(\R^{n})$ and $x\in\R^{n}$. Suppose that there exists a compact $K\subseteq\Rn$ such that $a(x,\eta)=0$ for all $(x,\eta)\in\Rn\times\Rn$ with $x\notin K$. Then $T:\HT^{s+m,p}_{FIO}(\Rn)\to \Hps$ is bounded, for all $p\in[1,\infty]$ and $s\in\R$.
\end{corollary}
\begin{proof}
Note that the Schwartz kernel of $T$ need not even be properly supported. Nonetheless, we can reduce to the setting of Proposition \ref{prop:FIORn}.

Set
\begin{equation}\label{eq:cnonsingular}
c:=\sup\{|\partial_{\eta}\Phi(x,\eta)|\mid (x,\eta)\in \supp(a), \eta\neq 0\}<\infty,
\end{equation}
and let $\chi\in C^{\infty}_{c}(\Rn)$ be such that $\chi(y)=1$ if $|y|\leq 2c$. Then, by design, the phase function in the Schwartz kernel of $T(1-\chi)$ is non-singular, so one can integrate by parts to conclude that $T(1-\chi)$ is in fact a smoothing operator. Here one uses that $a$ has compact spatial support to obtain the required decay of the kernel at infinity. Hence it suffices to show that $T\chi:\HT^{s+m,p}_{FIO}(\Rn)\to\Hps$ boundedly.

Next, an inspection of the kernel of $T\chi$ shows that $T\chi\in I^{m}(\Rn,\Rn;\Ca)$ for
\[
\Ca:=\{(x,\partial_{x}\Phi(x,\eta),\partial_{\eta}\Phi(x,\eta),\eta)\mid (x,\eta)\in V\}.
\]
Moreover, $(\partial_{\eta}\Phi(x,\eta),\eta)\mapsto (x,\partial_{x}\Phi(x,\eta))$ is locally a well-defined diffeomorphism. Indeed, since this map is the composition of $(\partial_{\eta}\Phi(x,\eta),\eta)\mapsto (x,\eta)$ and $(x,\eta)\mapsto (x,\partial_{x}\Phi(x,\eta))$, the statement follows from the inverse function theorem, combined with the assumption on $\partial_{x\eta}^{2}\Phi$. We thus see that $\Ca$ is a local canonical graph, and Proposition \ref{prop:FIORn} concludes the proof.
\end{proof}

In particular, we can now conclude that $\Hps$ is invariant under coordinate changes. Recall that, for $\gamma\in\R$ and for $\chi:U\to U'$ a diffeomorphism between open subsets $U,U'\subseteq\Rn$, the pullback $\chi^{*}_{\gamma}f\in \Da'(U)$ of $f\in\Da'(U')$, viewed as a distributional density of order $\gamma$, is given by
\[
\lb \chi^{*}_{\gamma}f,\phi\rb_{\Rn}=\lb f,|\det\partial(\chi^{-1})|^{1-\gamma}\phi\circ\chi^{-1}\rb_{\Rn}
\]
for $\phi\in\Da(U)$. 

\begin{corollary}\label{cor:coordchange}
Let $p\in[1,\infty]$ and $s,\gamma\in\R$. Let $\chi:U\to U'$ be a diffeomorphism between open subsets $U,U'\subseteq \Rn$, and let $\psi\in C^{\infty}_{c}(U')$. Then there exists a constant $C\geq0$ such that
\[
\|\chi^{*}_{\gamma}(\psi f)\|_{\Hps}\leq C\|f\|_{\Hps}
\]
for all $f\in\Hps$.
\end{corollary}

We will mostly use this corollary for $\gamma=1/2$. In this case the pullback $\chi^{*}_{1/2}:L^{2}(U')\to L^{2}(U)$ is a unitary operator.

\begin{proof}
Note that, for $f\in\Sw(\Rn)$ and $x\in U$, one has
\begin{align*}
\chi^{*}_{\gamma}(\psi f)(x)&=|\!\det(\partial_{x}\chi(x))|^{\gamma}\psi(\chi(x))f(\chi(x))\\
&=\frac{1}{(2\pi)^{n}}\int_{\Rn}e^{i\chi(x)\cdot\eta}|\!\det(\partial_{x}\chi(x))|^{\gamma}\psi(\chi(x))\wh{f}(\eta)\ud\eta.
\end{align*}
Now apply Corollary \ref{cor:stanform}.
\end{proof}

\begin{remark}\label{rem:constantsFIO}
By keeping track of the implicit constants in the proof of Corollary \ref{cor:stanform}, one can check that the following holds. For each $i$ in an index set $I$, let $T_{i}$ be an operator as in Corollary \ref{cor:stanform}, with symbol $a_{i}\in S^{m}(\Rn\times\Rn)$ and phase function $\Phi_{i}\in C^{\infty}(V)$, for some fixed conic $V\subseteq \Rn\times (\Rn\setminus\{0\})$. Suppose that
\[
\sup\{\lb\eta\rb^{-m+|\beta|}|\partial_{x}^{\alpha}\partial_{\eta}^{\beta}a_{i}(x,\eta)|\mid x,\eta\in\Rn, i\in I\}<\infty
\]
for all $\alpha,\beta\in\Z_{+}^{n}$, and that there exists a fixed compact $K\subseteq \Rn$ such that $a_{i}(x,\eta)=0$ for all $i\in I$ and $(x,\eta)\in\R^{2n}$ with $x\notin K$. Moreover, suppose that
\[
\inf\{|\!\det\partial_{x\eta}^{2}\Phi_{i}(x,\eta)|\mid (x,\eta)\in V, i\in I\}>0,
\]
and that
\[
\sup\{|\partial_{x}^{\alpha}\partial_{\eta}^{\beta}\Phi_{i}(x,\eta)|\mid (x,\eta)\in V\cap \Sp, i\in I\}<\infty
\]
for all $\alpha,\beta\in\Z_{+}^{n}$ with $|\alpha|+|\beta|\geq 2$. Then
\begin{equation}\label{eq:constantschange}
\sup\{\|T_{i}\|_{\La(\HT^{s+m,p}_{FIO}(\Rn),\Hps)}\mid i\in I\}<\infty
\end{equation}
for all $p\in[1,\infty]$ and $s\in\R$. 

A specific case of \eqref{eq:constantschange} concerns Corollary \ref{cor:coordchange}. That is, let $U\subseteq \Rn$ be a fixed open set, and let $(\psi_{i})_{i\in I}\subseteq C^{\infty}_{c}(U)$ and $(\chi_{i})_{i\in I}$ be collections with $\chi_{i}:U\to U$ a diffeomorphism for each $i\in I$. Suppose that
\[
\sup\{|\partial_{x}^{\alpha}\psi_{i}(x)|+|\partial_{x}^{\alpha}\chi_{i}(x)|+|\partial_{x}^{\alpha}\chi_{i}^{-1}(x)|\mid x\in U,i\in I\}<\infty
\]
for each $\alpha\in\Z_{+}^{n}$, and let $\gamma\in\R$. Then there exists a $C\geq0$ such that
\[
\|(\chi_{i})^{*}_{\gamma}(\psi_{i}f)\|_{\Hps}\leq C\|f\|_{\Hps}
\]
for all $i\in I$ and $f\in\Hps$.
\end{remark}

\section{Sequence spaces}\label{sec:sequences}

To derive the basic properties of the Hardy spaces for Fourier integral operators on manifolds, we will rely on operators that map to and from suitable sequence spaces. In this section we prove the required results about these sequence spaces and the operators between them. Moreover, in the final subsection we prove a localization principle that allows one to apply this framework to $\Hps$.

\subsection{Definitions}\label{subsec:defseq}

Here we introduce the sequence spaces which we will work with, and the relevant operators. These operators can be viewed as a symmetrized version of the standard correspondence between distributional densities and sequences of distributions (see \eqref{eq:Qinv} and \eqref{eq:Pinv}).

For the remainder of this section, let $(M,g)$ be a complete Riemannian manifold of dimension $n\in\N$ with bounded geometry and with smooth structure $\A$. Let $\inj(M)>0$ be the injectivity radius of $M$, $\delta\in(0,\inj(M)/9]$, and let $(\psi_{\ka}^{2})_{\ka\in K}\subseteq C^{\infty}_{c}(M)$ be a uniformly locally finite partition of unity as in Lemma \ref{lem:cover}, associated with a countable collection $K\subseteq \A$ of geodesic coordinate charts $\ka:U_{\ka}\to\tilde{U}_{\ka}\subseteq\Rn$ such that the $U_{\ka}$, $\ka\in K$, cover $M$. Write $\psit_{\ka}:=\psi_{\ka}\circ\ka^{-1}\in C^{\infty}_{c}(\tilde{U}_{\ka})$.

One of the sequence spaces that we will use is the locally convex space of $\Da(\Rn)$-valued sequences with finitely many nonzero values, denoted $c_{00}(K;\Da(\Rn))$, endowed with the natural inductive limit topology. These sequences will play the role of ``test sequences" in our setup, in analogy with the role that the test functions $\Da(\Rn)$ play on $\Rn$. Moreover, $c_{00}(K;\Da(\Rn))^{*}$ is its dual, and we write $\lb F,G\rb_{K}$ for the duality between $F\in c_{00}(K;\Da(\Rn))^{*}$ and $G\in c_{00}(K;\Da(\Rn))$.

Then $c_{00}(K;\Da(\Rn))\subseteq c_{00}(K;\Da(\Rn))^{*}$ is dense, if one sets
\begin{equation}\label{eq:dualityK}
 \lb F,G\rb_{K}:=\sum_{\ka\in K}\lb F_{\ka},G_{\ka}\rb_{\Rn}
\end{equation}
for $F=(F_{\ka})_{\ka\in K},G=(G_{\ka})_{\ka\in K}\in c_{00}(K;\Da(\Rn))$. Indeed,
\[
c_{00}(K';\Da(\Rn))\subseteq c_{00}(K';\Da(\Rn))^{*}=c_{00}(K';\Da'(\Rn))
\]
is dense for each finite subset $K'\subseteq K$, as follows from the density of $\Da(\Rn)$ in $\Da'(\Rn)$, and this in turn implies that $c_{00}(K;\Da(\Rn))\subseteq c_{00}(K;\Da(\Rn))^{*}$ is dense. Also note that $c_{00}(K;\Da(\Rn))\subseteq \ell^{p}(K;\Hps)$ is dense for all $p\in[1,\infty)$ and $s\in\R$, by Lemma \ref{lem:Ccdense} and because the finite sequences $c_{00}(K)$ are dense in $\ell^{p}(K)$.

Let $\Da(M,\Omega_{1/2})$ be the space of smooth compactly supported half densities on $M$, and let $\Da'(M,\Omega_{1/2})$ be its dual, the space of distributional half densities. Denote by $\lb\cdot,\cdot\rb_{M}$ the associated duality pairing. Then $L^{2}(M,\Omega_{1/2})$ is the completion of $\Da(M,\Omega_{1/2})$ with respect to the norm
\begin{equation}\label{eq:L2norm}
\|u\|_{L^{2}(M,\Omega_{1/2})}:=\lb u,\overline{u}\rb_{M}^{1/2},
\end{equation}
for $u\in\Da(M,\Omega_{1/2})$.  Write $\ka_{*}=(\ka^{-1})^{*}$, where $\ka^{*}=(\ka^{*})_{1/2}$ is pull-back by $\ka$, acting on half densities. For $u\in \Da'(M,\Omega_{1/2})$ and $G=(g_{\ka})_{\ka\in K}\in c_{00}(K;\Da(\Rn))$, set
\begin{equation}\label{eq:Q}
\lb Qu,G\rb_{K}:=\sum_{\ka\in K}\lb\ka_{*}(\psi_{\ka} u),g_{\ka}\rb_{\Rn}.
\end{equation}
Moreover, for $F=(f_{\ka})_{\ka\in K}\in c_{00}(K;\Da(\Rn))$, set
\begin{equation}\label{eq:P}
PF:=\sum_{\ka\in K}\ka^{*}(\psit_{\ka}f_{\ka}).
\end{equation}

\subsection{Basic properties}\label{subsec:basicseq}

The following proposition collects some basic properties of the operators $Q$ and $P$.

\begin{proposition}\label{prop:Qproperties}
The following properties hold:
\begin{enumerate}
\item\label{it:Qproperties1} $Q:\Da'(M,\Omega_{1/2})\to c_{00}(K;\Da(\Rn))^{*}$ and $Q:\Da(M,\Omega_{1/2})\to c_{00}(K;\Da(\Rn))$ are continuous;
\item\label{it:Qproperties2} $P:c_{00}(K;\Da(\Rn))\to \Da(M,\Omega_{1/2})$ is continuous, and $P$ extends uniquely to a continuous linear operator $P:c_{00}(K;\Da(\Rn))^{*}\to \Da'(M,\Omega_{1/2})$;
\item\label{it:Qproperties3} $\lb PF,v\rb_{M}=\lb F,Qv\rb_{K}$  for all $F\in c_{00}(K;\Da(\Rn))^{*}$ and $v\in\Da(M,\Omega_{1/2})$, and $\lb Qu,G\rb_{K}=\lb u,PG\rb_{M}$ for all $u\in\Da'(M,\Omega_{1/2})$ and $G\in c_{00}(K;\Da(\Rn))$;
\item \label{it:Qproperties4} $PQu=u$ for all $u\in\Da'(M,\Omega_{1/2})$;
\item\label{it:Qproperties5} $Q:L^{2}(M,\Omega_{1/2})\to \ell^{2}(K;L^{2}(\Rn))$ is an isometry.
\end{enumerate}
\end{proposition}
\begin{proof}
\eqref{it:Qproperties1}: Firstly, $\ka_{*}(\psi_{\ka}u)=\psit_{\ka}\ka_{*}u$ is a well-defined and compactly supported distribution on $\Util_{\ka}$, for all $u\in\Da'(M,\Omega_{1/2})$ and $\ka\in K$. In particular, it extends to a distribution on all of $\Rn$ by setting it equal to zero outside of $\Util_{\ka}$. Moreover, the sum in \eqref{eq:Q} is finite for each $G\in c_{00}(K;\Da(\Rn))$, and appropriate bounds hold for the resulting functional because $(\psit_{\ka})_{\ka\in K}\subseteq \Da(\Rn)$.

Next, suppose that $u\in \Da(M,\Omega_{1/2})$, and note that $\ka_{*}(\psi_{\ka}u)\in \Da(\Util_{\ka})\subseteq \Da(\Rn)$ for each $\ka\in K$, because $\psi_{\ka}\in C^{\infty}_{c}(U_{\ka})$. Moreover, since the cover $(U_{\ka})_{\ka\in K}$ of $M$ is locally finite, a given compact set only intersects finitely many of the $U_{\ka}$. Hence $\psi_{\ka}u=0$ for all but finitely many $\ka\in K$, and
\[
Qu=(\ka_{*}(\psi_{\ka}u))_{\ka\in K}\in c_{00}(K;\Da(\Rn)),
\]
with the appropriate bounds for the associated seminorms.

\eqref{it:Qproperties2}: Let $F=(f_{\ka})_{\ka\in K}\in c_{00}(K;\Da(\Rn))$. Then  $\ka^{*}(\psit_{\ka}f_{\ka})\in \Da(M,\Omega_{1/2})$ for all $\ka\in K$, since $\psit_{\ka}f_{\ka}\in \Da(\Util_{\ka})$. Moreover, the sum in \eqref{eq:P} is finite, so $PF\in \Da(M,\Omega_{1/2})$, with the appropriate bounds for the associated seminorms. The second statement now follows from the fact that $c_{00}(K;\Da(\Rn))$ is dense in $c_{00}(K;\Da(\Rn))^{*}$.

\eqref{it:Qproperties3}: Let $F=(f_{\ka})_{\ka\in K}\in c_{00}(K;\Da(\Rn))$ and $v\in\Da(M,\Omega_{1/2})$. Then
\[
\lb PF,v\rb_{M}=\sum_{\ka\in K}\lb \psi_{\ka}\ka^{*}f_{\ka},v\rb_{M}=\sum_{\ka\in K}\lb f_{\ka}, \ka_{*}(\psi_{\ka}v)\rb_{\Rn}=\lb F,Qv\rb_{K}.
\]
By \eqref{it:Qproperties2} and because $c_{00}(K;\Da(\Rn))$ lies dense in $c_{00}(K;\Da(\Rn))^{*}$, this proves the first statement. Similarly, since $\Da(M,\Omega_{1/2})\subseteq \Da'(M,\Omega_{1/2})$ is dense, it suffices to prove the second statement for $u\in \Da(M,\Omega_{1/2})$. But in this case one can simply use the first statement.

\eqref{it:Qproperties4}: By \eqref{it:Qproperties1} and \eqref{it:Qproperties2}, and because $\Da(M,\Omega_{1/2})\subseteq \Da'(M,\Omega_{1/2})$ is dense, it suffices to consider $u\in\Da(M,\Omega_{1/2})$. In this case we have
\[
PQu=\sum_{\ka\in K}\psi_{\ka}\ka^{*}(Qu)_{\ka}=\sum_{\ka\in K}\psi_{\ka}\ka^{*}\ka_{*}(\psi_{\ka}u)=u,
\]
since $(\psi_{\ka}^{2})_{\ka\in K}$ is a partition of unity.

\eqref{it:Qproperties5}: Let $u\in\Da(M,\Omega_{1/2})$ and note that
\[
\overline{(Qu)_{\ka}(y)}=\overline{\ka_{*}(\psi_{\ka}u)(y)}=(Q\overline{u})_{\ka}(y)
\]
for all $\ka\in K$ and $y\in \Rn$, by definition of the pullback of a smooth density (see \eqref{eq:pullbacktest}) and because $\psi_{\ka}$ is real-valued. Now \eqref{it:Qproperties3} and \eqref{it:Qproperties4} yield
\[
\|Qu\|_{\ell^{2}(L^{2}(\Rn))}^{2}=\lb Qu,\overline{Qu}\rb_{K}=\lb PQu,\overline{u}\rb_{M}=\lb u,\overline{u}\rb_{M}=\|u\|_{L^{2}(M,\Omega_{1/2})}^{2},
\]
by definition of the $L^{2}(M,\Omega_{1/2})$ norm in \eqref{eq:L2norm}. This concludes the proof, given that $L^{2}(M,\Omega_{1/2})$ is the completion of $\Da(M,\Omega_{1/2})$ with respect to this norm.
\end{proof}

\subsection{Operators between sequence spaces}\label{subsec:mainseq}

In this subsection we will consider operators between $\Hps$-valued sequence spaces.

Throughout, let $(N,g')$ be another complete $n$-dimensional
Riemannian manifold with bounded geometry, with smooth structure $\B$,
and let $\veps\in(0,\inj(N)/9]$, $L\subseteq \B$
and $(\theta_{\la})_{\la\in L}\subseteq C^{\infty}_{c}(N)$ be data as in Lemma \ref{lem:cover}, with $\theta_{\la}:U_{\la}\to\tilde{U}_{\la}\subseteq\Rn$ for $\la\in L$. That is, $\veps$, $L$
and $(\theta_{\la})_{\la\in L}$ have similar properties as $\delta$, $K$ and $(\psi_{\ka})_{\ka\in K}$ from the previous subsection. Set $\thetat_{\la}:=\theta_{\la}\circ\la^{-1}\in C^{\infty}_{c}(\tilde{U}_{\la})$ for $\la\in L$.
Next, let $Q'$ and $P'$ be defined as $Q$ and $P$:
\begin{equation}\label{eq:Qprime}
\lb Q'u,G\rb_{L}:=\sum_{\la\in L}\lb\la_{*}(\theta_{\la} u),g_{\la}\rb_{\Rn}
\end{equation}
for $u\in\Da'(N,\Omega_{1/2})$ and $G=(g_{\la})_{\la\in L}\in c_{00}(L;\Da(\Rn))$, and
\begin{equation}\label{eq:Pprime}
P'F:=\sum_{\la\in L}\la^{*}(\thetat_{\la}f_{\la})
\end{equation}
for $F=(f_{\la})_{\la\in L}\in c_{00}(L;\Da(\Rn))$. Then $Q':\Da'(N,\Omega_{1/2})\to c_{00}(L;\Da(\Rn))^{*}$ and $P':c_{00}(L;\Da(\Rn))^{*}\to\Da'(N,\Omega_{1/2})$ have the same properties as $Q$ and $P$, from Proposition \ref{prop:Qproperties}.

Before we get to the main results of this section, we prove an abstract lemma about vector-valued sequence spaces. Recall that $|E|$ denotes the cardinality of a finite set $E$.

\begin{lemma}\label{lem:matrix}
Let $X$ and $Y$ be Banach spaces, let $I$ and $J$ be countable sets, and let $p\in[1,\infty]$. Let $A=(A_{ij})_{i\in I,j\in J}$ be an operator matrix with the following properties:
\begin{enumerate}
\item\label{it:matrix1} $A_{ij}\in \La(X,Y)$ for all $i\in I$ and $j\in J$, and
\[
\sup\{\|A_{ij}\|_{\La(X,Y)}\mid i\in I,j\in J\}<\infty;
\]
\item\label{it:matrix2} $\sup_{i\in I}|J(i)|<\infty$, where $J(i):=\{j\in J\mid A_{ij}\neq 0\}$ for $i\in I$, and $\sup_{j\in J}|I(j)|<\infty$, where $I(j):=\{i\in I\mid A_{ij}\neq0\}$ for $j\in J$.
\end{enumerate}
Then $A:\ell^{p}(J;X)\to \ell^{p}(I;Y)$ is bounded.
\end{lemma}
\begin{proof}
We consider $p<\infty$, with the proof for $p=\infty$ being identical up to an obvious change of notation.

Let $F=(f_{j})_{j\in J}\in \ell^{p}(J;X)$. Then H\"{o}lder's inequality and \eqref{it:matrix2} yield
\begin{align*}
\|(AF)_{i}\|_{Y}&=\Big\|\sum_{j\in J(i)}A_{ij}f_{j}\Big\|_{Y}\leq \sum_{j\in J(i)}\|A_{ij}f_{j}\|_{Y}\\
&\lesssim \Big(\sum_{j\in J(i)}\|A_{ij}f_{j}\|_{Y}^{p}\Big)^{1/p}
\end{align*}
for an implicit constant independent of $i\in I$. Now combine this with \eqref{it:matrix1} and \eqref{it:matrix2}:
\begin{align*}
\|AF\|_{\ell^{p}(I;Y)}&=\Big(\sum_{i\in I}\|(AF)_{i}\|_{Y}^{p}\Big)^{1/p}\lesssim \Big(\sum_{i\in I}\sum_{j\in J(i)}\|A_{ij}f_{j}\|_{Y}^{p}\Big)^{1/p}\\
&=\Big(\sum_{j\in J}\sum_{i\in I(j)}\|A_{ij}f_{j}\|_{Y}^{p}\Big)^{1/p}\lesssim \Big(\sum_{j\in J}\sum_{i\in I(j)}\|f_{j}\|_{X}^{p}\Big)^{1/p}\lesssim \|F\|_{\ell^{p}(J;X)},
\end{align*}
for implicit constants independent of $F$.
\end{proof}

We now consider a specific instance of this lemma, involving the types of operators that we will encounter in the next section.

\begin{theorem}\label{thm:FIOseq}
Let $p\in[1,\infty]$, $s,m\in\R$, and let $T\in I^{m}(M,N;\Ca)$, for $\Ca$ a local canonical graph. Suppose that the following conditions hold:
\begin{enumerate}
\item\label{it:FIOseq1} There exists a $C\geq 0$ such that $\la_{*}(\theta_{\la}T\ka^{*}(\psit_{\ka}f))\in\Hps$ and
\[
\|\la_{*}(\theta_{\la}T\ka^{*}(\psit_{\ka}f))\|_{\Hps}\leq C\|f\|_{\HT^{s+m,p}_{FIO}(\Rn)},
\]
for all $\la\in L$, $\ka\in K$ and $f\in\HT^{s+m,p}_{FIO}(\Rn)$;
\item\label{it:FIOseq2} $\sup_{\la\in L}|K(\la)|<\infty$, where $K(\la):=\{\ka\in K\mid \theta_{\la}T\psi_{\ka}\neq 0\}$ for $\la\in L$, and $\sup_{\ka\in K}|L(\ka)|<\infty$, where $L(\ka):=\{\la\in L\mid \theta_{\la}T\psi_{\ka}\neq0\}$ for $\ka\in K$.
\end{enumerate}
Then $Q'TP:\ell^{p}(K;\HT^{s+m,p}_{FIO}(\Rn))\to\ell^{p}(L;\Hps)$ is bounded.
\end{theorem}
Condition \eqref{it:FIOseq2} implies that $T$ is properly supported, and it in fact prescribes proper support in a certain uniform sense.
\begin{proof}
We will apply Lemma \ref{lem:matrix} with $A_{\la\ka}f:=\la_{*}(\theta_{\la}T\ka^{*}(\psit_{\ka}f))$ for $\la\in L$, $\ka\in K$ and $f\in\HT^{s+m,p}_{FIO}(\Rn)$. To this end, we only need to address a minor technicality, by checking that $Q'TP$ is given by matrix multiplication with $A=(A_{\la\ka})_{\la\in L,\ka\in K}$.

Indeed, $Q'TPF\in c_{00}(L,\Da(\Rn))^{*}$ for each $F\in\ell^{p}(K;\HT^{s+m,p}_{FIO}(\Rn))$, by Proposition \ref{prop:Qproperties}. However, a priori $Q'TPF$ is defined through adjoint action, as
\begin{equation}\label{eq:adjointdef}
\lb Q'TPF,G\rb_{L}=\sum_{\la\in L}\lb PF,T^{t}\la^{*}(\thetat_{\la}g_{\la})\rb_{M}
\end{equation}
for $G=(g_{\la})_{\la\in L}\in c_{00}(L;\Da(\Rn))$, where $P:c_{00}(K;\Da(\Rn))^{*}\to\Da'(M,\Omega_{1/2})$ is the unique continuous extension of \eqref{eq:P}. On the other hand, for $F\in c_{00}(K;\Da(\Rn))$ one can use \eqref{eq:P} to obtain $Q'TPF=AF$, and for a general $F\in \ell^{p}(K;\HT^{s+m,p}_{FIO}(\Rn))$ one can approximate by test sequences in \eqref{eq:adjointdef} to obtain the same conclusion.
\end{proof}


The following corollary of Theorem \ref{thm:FIOseq} deals with the special case where $N=M$ and $T$ is the identity operator, although we still allow the sequence $(\theta_{\la})_{\la\in L}\subseteq C^{\infty}_{c}(M)$ to be distinct from $(\psi_{\ka})_{\ka\in K}$. This corollary will frequently be used in the next section, and allowing for distinct sequences will enable us to prove that the definition of $\HpsM$ is independent of the choice of such a sequence.

\begin{corollary}\label{cor:boundedPQ}
Let $p\in[1,\infty]$ and $s\in\R$. Then $Q'P:\ell^{p}(K;\Hps)\to\ell^{p}(L;\Hps)$ is bounded.
\end{corollary}
\begin{proof}
We apply Theorem \ref{thm:FIOseq} with $T$ the identity operator, which is an FIO of order zero associated with the graph of the identity map on $T^{*}M\setminus o$. It thus suffices to verify conditions \eqref{it:FIOseq1} and \eqref{it:FIOseq2}.

For \eqref{it:FIOseq1}, we only have to consider $\lambda\in L$ and $\ka\in K$ with $U_{\la}\cap U_{\ka}\neq \emptyset$.
Then Corollary \ref{cor:coordchange}, Remark \ref{rem:constantsFIO}, \eqref{eq:uniformchange} and Lemma \ref{lem:cover} \eqref{it:cover4} yield
\begin{align*}
\|\la_{*}(\theta_{\la}\ka^{*}(\psit_{\ka}f))\|_{\Hps}&=\|\thetat_{\la}\mu_{\ka\la}^{*}(\psit_{\ka}f)\|_{\Hps}\lesssim \|\mu_{\ka\la}^{*}(\psit_{\ka}f)\|_{\Hps}\\
&\lesssim \|f\|_{\Hps},
\end{align*}
for implicit constants independent of $f\in\Hps$, $\la\in L$ and $\ka\in K$.

For \eqref{it:FIOseq2}, we suppose that $\delta\geq\veps$, with the proof for the other case being analogous. First fix $\la\in L$, and note that $K(\la)\subseteq \{\ka\in K\mid U_{\la}\cap U_{\ka}\neq\emptyset\}$. Hence, for a given $\ka\in K(\la)$, one has $U_{3\delta}(p_{\ka})\cap U_{3\delta}(p_{\ka'})\neq\emptyset$ for all $\ka'\in K(\la)$. Since $\delta\leq \inj(M)/9$, the conclusion for $|K(\lambda)|$ now follows from Lemma \ref{lem:cover} \eqref{it:cover2}. One can similarly apply Lemma \ref{lem:cover} \eqref{it:cover2} for fixed $\ka\in K$ and $\la\in L(\ka)$, since then $U_{3\delta}(p_{\la})\cap U_{3\delta}(p_{\la'})\neq\emptyset$ for all $\la'\in L(\ka)$.
\end{proof}

\begin{remark}\label{rem:seqclassical}
Theorem \ref{thm:FIOseq} and Corollary \ref{cor:boundedPQ} were formulated for $\Hps$-valued sequences, but both relied on Lemma \ref{lem:matrix}, and one could equally well apply that abstract result with $X=Y=\HT^{s,p}(\Rn)$. Then condition \eqref{it:FIOseq1} in Theorem \ref{thm:FIOseq} will typically not be satisfied for general Fourier integral operators, and pseudodifferential operators form a more natural class to consider. For us, it is only relevant to note that Corollary \ref{cor:boundedPQ} holds with $\Hps$ replaced by $\HT^{s,p}(\Rn)$, as follows from the fact that $\HT^{s,p}(\Rn)$ is also invariant under coordinate changes as in Corollary \ref{cor:coordchange} and Remark \ref{rem:constantsFIO} (see e.g.~\cite[Sections 4.2 and 4.3]{Triebel92}).
\end{remark}

\subsection{A localization principle}\label{subsec:localization}

In this subsection, we will consider the special case where $M=\Rn$ and prove a localization principle for the Hardy spaces for Fourier integral operators on $\Rn$.

Throughout this subsection, let $\psi\in C^{\infty}_{c}(\Rn)$ be such that $0\leq \psi\leq 1$ and
\begin{equation}\label{eq:reproduce}
\sum_{\ka\in\Z^{n}}\psi(x-\ka)^{2}=1\quad(x\in\Rn).
\end{equation}
Set $\psi_{\ka}(x):=\psi(x-\ka)$ for $\ka\in\Z^{n}$, and note that $(\psi_{\ka})_{\ka\in\Z^{n}}$ is a uniformly locally finite family as in Lemma \ref{lem:cover}. Here the associated family of geodesic charts consists of translations $(\tau_{\ka})_{\ka\in\Z^{n}}$ of a single ball $B$, centered at zero and containing $\supp(\psi)$.

More precisely, set $\tau_\ka(x):=x-\ka$ for $\ka\in\Z^{n}$ and $x\in\Rn$. Then $\tau_{\ka}:B+\ka\to B$ is a geodesic normal chart as in Lemma \ref{lem:cover} for each $\ka\in \Z^{n}$, and $(\psi_{\ka})_{\ka\in \Z^{n}}$ is the associated uniformly locally finite family. In this case, the operators $Q$ and $P$ from \eqref{eq:Q} and \eqref{eq:P} are given by
\begin{equation}\label{eq:QRn}
\lb Qf,G\rb_{\Z^{n}}=\sum_{\ka\in \Z^{n}}\lb \tau_{-\ka}^{*}(\psi_{\ka}f),g_{\ka}\rb_{\Rn}=\sum_{\ka\in \Z^{n}}\lb \psi\tau_{-\ka}^{*}f,g_{\ka}\rb_{\Rn},
\end{equation}
for $f\in \Da'(\Rn)$ and $G=(g_{\ka})_{\ka\in K}\in c_{00}(\Z^{n};\Da(\Rn))$, and
\begin{equation}\label{eq:PRn}
PF=\sum_{\ka\in \Z^{n}}\tau_{\ka}^{*}(\psi f_{\ka}),
\end{equation}
for $F=(f_{\ka})_{\ka\in \Z^{n}}\in c_{00}(\Z^{n};\Da(\Rn))$.

Now, for $p\in[1,\infty]$ and $s\in\R$, an $f\in\Sw'(\Rn)$ satisfies $f\in\HT^{s,p}(\Rn)$ if and only if $(\psi_{\ka}f)_{\ka\in\Zn}\in \ell^{p}(\Zn;\HT^{s,p}(\Rn))$, in which case
\begin{equation}\label{eq:localclassical}
\|f\|_{\HT^{s,p}(\Rn)}\eqsim \|(\psi_{\ka}f)_{\ka\in\Zn}\|_{\ell^{p}(\Zn;\HT^{s,p}(\Rn))}
=\Big(\sum_{\ka\in\Zn}\|\psi_{\ka}f\|^p_{\HT^{s,p}(\Rn)}\Big)^{1/p},
\end{equation}
with the obvious modification in the final identity for $p=\infty$. This statement is essentially contained in \cite[Theorem 2.4.7]{Triebel92}, although there only the case $p<\infty$ is considered, and $\psi_{\ka}$ is replaced by $\psi_{\ka}^{2}$. To obtain \eqref{eq:localclassical} one can rely on the arguments below, combined with Remark \ref{rem:seqclassical} and the atomic decomposition of $\HT^{s,1}(\Rn)$ in Lemma \ref{lem:atomHp}.

The following theorem is an analogue of \eqref{eq:localclassical} for $\Hps$.

\begin{theorem}\label{thm:localization}
Let $p\in[1,\infty]$ and $s\in\R$. Then there exists a $C>0$ such that an $f\in\Sw'(\Rn)$ satisfies $f\in\Hps$ if and only if $(\psi_{\ka}f)_{\ka\in\Zn}\in\ell^{p}(\Zn;\Hps)$, in which case
\[
\frac{1}{C}\|f\|_{\Hps}\leq \|(\psi_{\ka}f)_{\ka\in\Zn}\|_{\ell^{p}(\Zn;\Hps)}\leq C\|f\|_{\Hps}.
\]
\end{theorem}
\begin{proof}
Let $Q$ and $P$ be as in \eqref{eq:QRn} and \eqref{eq:PRn}. It suffices to show that
\[
Q:\Hps\to \ell^{p}(\Zn;\Hps)\text{ and }P:\ell^{p}(\Zn;\Hps)\to\Hps
\]
are bounded, given that the $\Hps$ norm is invariant under translations, and that $PQf=f$ for all $f\in\Da'(\Rn)$.

We claim that it in fact suffices to show that $Q:\HT^{s,1}_{FIO}(\Rn)\to \ell^{1}(\Z^{n};\HT^{s,1}_{FIO}(\Rn))$ is bounded, for all $s\in\R$. To prove this claim, first note that $Q:\HT^{s,\infty}_{FIO}(\Rn)\to \ell^{\infty}(\Zn;\HT^{s,\infty}(\Rn))$ is bounded, since
\[
\sup_{\ka\in\Zn}\|\psi \tau_{-\ka}^{*}f\|_{\HT^{s,\infty}_{FIO}(\Rn)}\lesssim \|f\|_{\HT^{s,\infty}_{FIO}(\Rn)}
\]
for all $f\in\HT^{s,\infty}_{FIO}(\Rn)$. Here we used that $\psi$ acts boundedly on $\HT^{s,\infty}_{FIO}(\Rn)$, by Corollary \ref{cor:coordchange}, and that the $\HT^{s,\infty}_{FIO}(\Rn)$ norm is invariant under translations. Moreover, by Corollary \ref{cor:boundedPQ}, $QP:\ell^{p}(\Z^{n};\Hps)\to\ell^{p}(\Z^{n};\Hps)$ is a bounded projection for all $p\in[1,\infty]$ and $s\in\R$. Hence, if we know that $Q:\HT^{s,1}_{FIO}(\Rn)\to \ell^{1}(\Z^{n};\HT^{s,1}_{FIO}(\Rn))$ is bounded, then interpolation of complemented subspaces (see \cite[Theorem 1.17.1.1]{Triebel78}) shows that $Q:\Hps\to\ell^{p}(\Zn;\Hps)$ is bounded.

On the other hand, the boundedness of $P:\ell^{1}(\Z^{n};\HT^{s,1}_{FIO}(\Rn))\to \HT^{s,1}_{FIO}(\Rn)$ is immediate. Indeed, for all $F=(f_{\ka})_{\ka\in\Zn}\in \ell^{1}(\Zn;\HT^{s,1}_{FIO}(\Rn))$ one has
\[
\|PF\|_{\HT^{s,1}_{FIO}(\Rn)}\leq \sum_{\ka\in\Zn}\|\tau_{\ka}^{*}(\psi f_{\ka})\|_{\HT^{s,1}_{FIO}(\Rn)}\lesssim \sum_{\ka\in\Zn}\|f_{\ka}\|_{\HT^{s,1}_{FIO}(\Rn)}.
\]
Moreover, if $Q:\HT^{-s,1}_{FIO}(\Rn)\to \ell^{1}(\Zn;\HT^{-s,1}_{FIO}(\Rn))$ is bounded, then for all $F\in \ell^{\infty}(\Zn;\HT^{s,\infty}_{FIO}(\Rn))$ and $g\in\Da(\Rn)$ one has, by Proposition \ref{prop:Qproperties} and Corollary \ref{cor:boundedPQ},
\begin{align*}
|\lb PF,g\rb_{\Rn}|&=|\lb PF,PQg\rb_{\Rn}|=|\lb QPF,Qg\rb_{\Z^{n}}|\\
&\lesssim\|QPF\|_{\ell^{\infty}(\Zn;\HT^{s,\infty}_{FIO}(\Rn))}\|Qg\|_{\ell^{1}(\Zn;\HT^{-s,1}_{FIO}(\Rn))}\\
&\lesssim \|F\|_{\ell^{\infty}(\Zn;\HT^{s,\infty}_{FIO}(\Rn))}\|g\|_{\HT^{-s,1}_{FIO}(\Rn)}.
\end{align*}
Here we also used that
\[
\ell^{\infty}(\Zn;\HT^{s,\infty}_{FIO}(\Rn))\subseteq (\ell^{1}(\Zn;\HT^{-s,1}_{FIO}(\Rn)))^{*},
\]
by \cite[Proposition 1.3.1]{HyNeVeWe16} and because $\HT^{s,\infty}_{FIO}(\Rn)=(\HT^{-s,1}_{FIO}(\Rn))^{*}$. By Lemma \ref{lem:Ccdense}, $\Da(\Rn)\subseteq\HT^{-s,1}_{FIO}(\Rn))$ is dense, so $P:\ell^{\infty}(\Zn;\HT^{s,\infty}_{FIO}(\Rn))\to\HT^{s,\infty}_{FIO}(\Rn))$ is bounded. Now interpolation of complemented subspaces shows that $P:\ell^{p}(\Zn;\Hps)\to\Hps$ is bounded for all $p\in[1,\infty]$ an $s\in\R$. This proves the claim.

It remains to show that $Q:\HT^{s,1}_{FIO}(\Rn)\to\ell^{1}(\Z^{n};\HT^{s,1}_{FIO}(\Rn))$ is bounded. To this end, we will rely on the atomic decomposition of $\HT^{s,1}_{FIO}(\Rn)$. More precisely, by Proposition \ref{prop:atom}, it suffices to prove that
\begin{equation}\label{eq:atomlow}
\sum_{\ka\in\Zn}\|\psi \tau_{-\ka}^{*}(r(D)f)\|_{\HT^{s,1}_{FIO}(\Rn)}\lesssim \|f\|_{\HT^{s,1}_{FIO}(\Rn)}
\end{equation}
for all $f\in\HT^{s,1}_{FIO}(\Rn)$, and that
\begin{equation}\label{eq:atomhigh}
\sum_{\ka\in\Zn}\|\psi \tau_{-\ka}^{*}f\|_{\HT^{s,1}_{FIO}(\Rn)}\lesssim \|f\|_{\HT^{s,1}_{FIO}(\Rn)}
\end{equation}
whenever $f$ is an $\HT^{s,1}_{FIO}(\Rn)$-atom associated with a ball of radius at most $1$. Here $r\in C^{\infty}_{c}(\Rn)$ is such that $r(\xi)=1$ for $|\xi|\leq 1$ and $r(\xi)=0$ for $|\xi|\geq 2$, although we will only use that $r\in C^{\infty}_{c}(\Rn)$.

For \eqref{eq:atomlow}, fix $N\in\N$ such that $N>s+s(1)$. Then \eqref{eq:Sobolev} yields
\begin{align*}
&\sum_{\ka\in\Zn}\|\psi \tau_{-\ka}^{*}(r(D)f)\|_{\HT^{s,1}_{FIO}(\Rn)}\lesssim \sum_{\ka\in\Zn}\|\psi \tau_{-\ka}^{*}(r(D)f)\|_{\HT^{s+s(1),1}(\Rn)}\\
&\lesssim \sum_{\ka\in\Zn}\|\psi \tau_{-\ka}^{*}(r(D)f)\|_{W^{N,1}(\Rn)}=\sum_{\ka\in\Zn}\|\psi_{\ka}r(D)f\|_{W^{N,1}(\Rn)}\\
&= \sum_{\alpha\in\Z^{n},|\alpha|\leq N}\sum_{\ka\in \Z^{n}}\|\partial_{x}^{\alpha}(\psi_{\ka}r(D)f)\|_{L^{1}(\Rn)},
\end{align*}
for implicit constants independent of $f\in\HT^{s,1}_{FIO}(\Rn)$. Moreover, since $(\psi_{\ka})_{\ka\in\Z^{n}}\subseteq C^{\infty}_{c}(\Rn)$ is a uniformly locally finite family, we can write $\Z^{n}=\sqcup_{j=1}^{M}A_{j}$ for some $M\in\N$, where each $A_{j}\subseteq \Z^{n}$ has the property that $\supp(\psi_{\ka})\cap \supp(\psi_{\la})=\emptyset$ whenever $\ka,\la\in A_{j}$. Hence, for each $\alpha\in\Z_{+}^{n}$ with $|\alpha|\leq N$ we can combine Leibniz' rule with the Sobolev embeddings in \eqref{eq:Sobolev}:
\begin{align*}
&\sum_{\ka\in \Z^{n}}\|\partial_{x}^{\alpha}(\psi_{\ka}r(D)f)\|_{L^{1}(\Rn)}=\sum_{j=1}^{M}\sum_{\ka\in A_{j}}\int_{\supp(\psi_{\ka})}|\partial_{x}^{\alpha}(\psi_{\ka}r(D)f)(x)|\ud x\\
&= \sum_{j=1}^{M} \int_{\Rn}\Big|\sum_{\ka\in A_{j}}\partial_{x}^{\alpha}(\psi_{\ka}r(D)f)(x)\Big|\ud x\lesssim \sum_{j=1}^{M} \|r(D)f\|_{W^{N,1}(\Rn)}\lesssim \|f\|_{\HT^{s-s(1),1}(\Rn)}\\
&\lesssim \|f\|_{\HT^{s,1}_{FIO}(\Rn)}.
\end{align*}
Here we also used that $r\in C^{\infty}_{c}(\Rn)$. This proves \eqref{eq:atomlow}.

On the other hand, let $f$ be an $\HT^{s,1}_{FIO}(\Rn)$-atom, associated with a ball in $\Sp$ of radius at most $1$. Then $\supp(f)$ is contained in a ball in $\Rn$ of radius at most $1$, by \eqref{eq:supportatom}. Hence the collection $K'$ of $\ka\in\Z^{n}$ with $\psi \tau_{\ka}^{*}f\neq 0$ has at most $N'\in\N$ elements, where $N'$ is independent of $f$. We thus obtain \eqref{eq:atomhigh}:
\begin{align*}
\sum_{\ka\in\Zn}\|\psi\tau_{-\ka}^{*}f\|_{\HT^{s,1}_{FIO}(\Rn)}&=\sum_{\ka\in K'}\|\psi\tau_{-\ka}^{*}f\|_{\HT^{s,1}_{FIO}(\Rn)}\lesssim \sum_{\ka\in K'}\|f\|_{\HT^{s,1}_{FIO}(\Rn)}\\
&\lesssim \|f\|_{\HT^{s,1}_{FIO}(\Rn)}.
\end{align*}
This concludes the proof.
\end{proof}

\section{Hardy spaces for Fourier integral operators on manifolds}\label{sec:spacesmanifold}

In this section we define the Hardy spaces for Fourier integral operators on manifolds, and we derive their fundamental properties.

Throughout, $(M,g)$ is a complete
Riemannian manifold of dimension $n\in\N$ with bounded geometry, with smooth structure $\A$ (see Section \ref{subsec:geom}).

\subsection{Definitions}\label{subsec:defman}

Throughout, fix $\delta\in (0,\inj(M)/9]$, $K\subseteq \A$ and a uniformly locally finite family $(\psi_{\ka})_{\ka\in K}\subseteq C^{\infty}_{c}(M)$ as in Lemma \ref{lem:cover}, with $\psit_{\ka}:=\psi_{\ka}\circ\ka^{-1}\in C^{\infty}_{c}(\Util_{\ka})$ for $\ka\in K$. We will interchangeably write $U_{\ka}$ and $U_{\delta}(p_{\ka})$ for the domain of $\ka$. Recall also the definition of the pullback
\[
(\ka^{-1})^{*}=(\ka^{-1})^{*}_{1/2}:\Da'(U_{\ka},\Omega_{1/2})\to \Da'(\Util_{\ka}),
\]
cf.~\eqref{eq:pullbackdensity}, and that $\ka_{*}=(\ka^{-1})^{*}$.

\begin{definition}\label{def:HpFIOman}
Let $p\in[1,\infty]$ and $s\in \mathbb{R}$. Then $\HpsM$ consists of those $u\in\Da'(M,\Omega_{1/2})$ such that $(\ka_{*}(\psi_{\ka}u))_{\ka\in K}\in \ell^{p}(K;\Hps)$, with the norm
\begin{equation}\label{eq:defHpFIOnorm}
\|u\|_{\HpsM}:=\|(\ka_{*}(\psi_{\ka}u))_{\ka\in K}\|_{\ell^{p}(K;\Hps)}.
\end{equation}
We write $\HpM:=\HT^{0,p}_{FIO}(M)$.
\end{definition}

Note that, after trivially identifying $\Da'(\Rn)$ with $\Da'(\Rn,\Omega_{1/2})$, Definition \ref{def:HpFIOman} extends the case where $M=\Rn$ from Definition \ref{def:HpFIO}, as follows from Theorem \ref{thm:localization}.

To relate the Hardy spaces for FIOs to classical function spaces on manifolds, we also need to define the latter. This is done in the same way as in Definition \ref{def:HpFIOman}.

\begin{definition}\label{def:Hpman}
Let $p\in[1,\infty]$ and $s\in \mathbb{R}$. Then $\HT^{s,p}(M)$ consists of those $u\in \Da'(M,\Omega_{1/2})$ such that $(\ka_{*}(\psi_{\ka}u))_{\ka\in K}\in \ell^{p}(K;\HT^{s,p}(M))$, with the norm
\[
\|u\|_{\HT^{s,p}(M)}:=\|(\ka_{*}(\psi_{\ka}u))_{\ka\in K}\|_{\ell^{p}(K;\HT^{s,p}(\Rn))}.
\]
We write $\HT^{p}(M):=\HT^{0,p}(M)$, and $\bmo(M):=\HT^{\infty}(M)$.
\end{definition}

\begin{remark}\label{rem:HpFIOman}
When formulating Definition \ref{def:HpFIOman} using the typical identification of a distributional density $u\in \Da'(M,\Omega_{1/2})$ with a sequence of distributions $(u_{\ka})_{\ka\in \A}$ on $\Rn$, via $u_{\ka}=\ka_{*}u$, one gets
\[
\|u\|_{\HpsM}=\Big(\sum_{\ka\in K}\|\psit_{\ka}u_{\ka}\|_{\Hps}^{p}\Big)^{1/p}
\]
for $p<\infty$, with the obvious modification for $p=\infty$. Note that if $M$ is compact and $K$ is finite, then $u\in \HpsM$ if and only if $\psit_{\ka}u_{\ka}\in\Hps$ for all $\ka\in K$. The same statements hold with $\HT^{s,p}_{FIO}(M)$ replaced by $\HT^{s,p}(M)$.
\end{remark}

In Section \ref{subsec:equivnorms} we will show that, for $1<p<\infty$, the spaces $\HT^{s,p}(M)$ coincide with classical Sobolev spaces on $M$ as introduced in \cite{Aubin82,Strichartz83} in a coordinate-free manner, up to a fractional power of the Riemannian density $\rho_{g}\in C^{\infty}(M,\Omega_{1})$. To this end, we first define Sobolev spaces on $M$ using fractional powers of the non-positive Laplace--Beltrami operator $\Delta_{g}$. Recall that the latter acts on smooth functions $f\in C^{\infty}(M)$ in local coordinates by
\begin{equation}\label{eq:Laplace}
\Delta_{g}f:=\frac{1}{\sqrt{\det g}}\sum_{i,j=1}^{n}\partial_{i}\big(\sqrt{\det g}\,g^{ij}\partial_{j}f\big).
\end{equation}

\begin{definition}\label{def:Sobolevman}
Let $p\in(1,\infty)$ and $s\in \mathbb{R}$. Then $W^{s,p}(M)$ is the completion of $\Da(M,\Omega_{1/2})$ with respect to the norm
\begin{equation}\label{eq:LpMnorm}
\|u\|_{W^{s,p}(M)}:=\Big(\int_{M} |(1-\Delta_{g})^{s/2}(\rho_{g}^{-1/2}u)(x)|^{p}\rho_{g}(x)\Big)^{1/p}
\end{equation}
for $u\in \Da(M,\Omega_{1/2})$. We write $L^{p}(M):=W^{0,p}(M)$.
\end{definition}

Note that $\rho_{g}^{-1/2}u\in \Da(M,\Omega_{0})=C^{\infty}_{c}(M)$ for each $u\in \Da(M,\Omega_{1/2})$, so that $(1-\Delta_{g})^{s/2}(\rho^{-1/2}_{g}u)\in\Da'(M,\Omega_{0})$ is a well defined function for $s\leq 0$, due to the spectral theorem. It follows from abstract methods (see e.g.~\cite{Strichartz83,Haase06a}) that \eqref{eq:LpMnorm} is finite for such $u$ and $s$. For general $s\in\R$ one sets
\[
(1-\Delta_{g})^{s/2}f:=(1-\Delta_{g})^{s/2-k}(1-\Delta_{g})^{k}f
\]
for $f\in C^{\infty}_{c}(M)$ and some $k\in\N$ with $k\geq s/2$, which is independent of the choice of $k$. Then $W^{s,p}(M)\subseteq \Da'(M,\Omega_{1/2})$.

\subsection{Basic properties}\label{subsec:propmanbasic}

In this subsection we derive the basic properties of $\HpsM$. To this end, and also in later subsections, we will use the operators $Q$ and $P$ from Section \ref{sec:sequences}. This is natural because a $u\in\Da'(M,\Omega_{1/2})$ satisfies $u\in\HpsM$ if and only if $Qu\in \ell^{p}(K;\Hps)$, in which case
\begin{equation}\label{eq:Qisom}
\|u\|_{\HpsM}=\|Qu\|_{\ell^{p}(K;\Hps)},
\end{equation}
as follows from the definition of $Q$ in \eqref{eq:Q}.

We first show that $\HpsM$ is a Banach space.

\begin{proposition}\label{prop:Banach}
Let $p\in[1,\infty]$ and $s\in\R$. Then $Q:\HpsM\to \ell^{p}(K;\Hps)$ is an isometry, and $P:\ell^{p}(K;\Hps)\to \HpsM$ is bounded. In particular, $Q\HpsM\subseteq \ell^{p}(K;\Hps)$ is a complemented subspace,
\begin{equation}\label{eq:Pisom}
P:\ell^{p}(K;\Hps)/\ker(P)\to \HpsM
\end{equation}
is an isomorphism, and $\HpsM$ is a Banach space.
\end{proposition}
\begin{proof}
The first statement is \eqref{eq:Qisom}. The second statement follows by observing that
\[
QP:\ell^{p}(K;\Hps)\to\ell^{p}(K;\Hps)
\]
is bounded, by Corollary \ref{cor:boundedPQ} in the special case where $Q'=Q$. Now Proposition \ref{prop:Qproperties} \eqref{it:Qproperties4} shows that $P:\ell^{p}(K;\Hps)\to \HpsM$ is surjective, which implies that \eqref{eq:Pisom} is an isomorphism. For the same reason, $Q\HpsM$ is the range of the bounded projection $QP$ on $\ell^{p}(K;\Hps)$. Since complemented subspaces are closed, or alternatively because $\ker(P)$ is closed, $\HpsM$ is a Banach space.
\end{proof}

Next, we show that the definition of $\HpsM$ is independent of the choice of data as in Lemma \ref{lem:cover}. To this end, let $\veps\in(0,\inj(M)/9]$, $L\subseteq \A$ and a uniformly locally finite family $(\theta_{\la})_{\la\in L}\subseteq C^{\infty}_{c}(M)$
as in Lemma \ref{lem:cover} be given, possibly different from the data that was used to define $\HpsM$.

\begin{proposition}\label{prop:independence}
Let $p\in[1,\infty]$ and $s\in\R$. Then there exists a $C>0$ with the following property. A $u\in \Da'(M,\Omega_{1/2})$ satisfies $u\in\HpsM$ if and only if $(\la_{*}(\theta_{\la}u))_{\la\in L}\in \ell^{p}(L;\Hps)$, in which case
\begin{equation}\label{eq:independence}
\frac{1}{C}\|u\|_{\HpsM}\leq \|(\la_{*}(\theta_{\la}u))_{\la\in L}\|_{\ell^{p}(L;\Hps)}\leq C\|u\|_{\HpsM}.
\end{equation}
\end{proposition}
\begin{proof}
We use the operators $Q'$ and $P'$ from \eqref{eq:Qprime} and \eqref{eq:Pprime}. More precisely, applying Proposition \ref{prop:Qproperties} \eqref{it:Qproperties4} with $Q$ and $P$ replaced by $Q'$ and $P'$, we have $P'Q'u=u$ for all $u\in\Da'(M,\Omega_{1/2})$. Hence Proposition \ref{prop:Banach} and Corollary \ref{cor:boundedPQ} yield
\begin{align*}
\|u\|_{\HpsM}&=\|Qu\|_{\ell^{p}(K;\Hps)}=\|QP'Q'u\|_{\ell^{p}(K;\Hps)}\\
&\lesssim \|Q'u\|_{\ell^{p}(L;\Hps)},
\end{align*}
for all $u\in\Da'(M,\Omega_{1/2})$ for which the right-hand side is finite. This proves the first inequality in \eqref{eq:independence}. The second inequality follows by symmetry.
\end{proof}

The following lemma deals with the simplest case of the theory, where $p=2$.

\begin{lemma}\label{lem:L2}
One has $\HT^{2}_{FIO}(M)=L^{2}(M,\Omega_{1/2})$, with equivalence of norms.
\end{lemma}
\begin{proof}
This follows directly from Propositions \ref{prop:Qproperties} \eqref{it:Qproperties5} and \ref{prop:Banach}, since $\HT^{2}_{FIO}(\Rn)=L^{2}(\Rn)$ with equivalence of norms.
\end{proof}

\begin{remark}\label{rem:L2}
The definition of $\Hp$ in Definition \ref{def:HpFIO} does not yield equality of the $\HT^{2}_{FIO}(\Rn)$ and $L^{2}(\Rn)$ norms. However, in \cite{HaPoRo20} the Hardy spaces for FIOs were defined in such a way that $\HT^{2}_{FIO}(\Rn)=L^{2}(\Rn)$ isometrically.
It thus follows from the proof of Lemma \ref{lem:L2} that, when starting off with the definition of $\Hp$ from \cite{HaPoRo20}, one has $\HT^{2}_{FIO}(M)=L^{2}(M,\Omega_{1/2})$ isometrically.
\end{remark}

To conclude this subsection, we consider the smooth compactly supported half densities as a subset of $\HpsM$.

\begin{proposition}\label{prop:densityM}
Let $p\in[1,\infty]$ and $s\in\R$. Then $\Da(M,\Omega_{1/2})\subseteq \HpsM$. Moreover, if $p<\infty$, then $\Da(M,\Omega_{1/2})$ is dense in $\HpsM$.
\end{proposition}
\begin{proof}
By Proposition \ref{prop:Qproperties} \eqref{it:Qproperties1}, $Qu\in c_{00}(K;\Da(\Rn))$ for all $u\in\Da(M,\Omega_{1/2})$. Since $\Da(\Rn)\subseteq\Sw(\Rn)\subseteq\Hps$, this implies that $Qu\in\ell^{p}(K;\Hps)$, which in turn means that $u\in \HpsM$, by Proposition \ref{prop:Banach}.

Now suppose that $p<\infty$, and let $u\in\HpsM$. Then $c_{00}(K;\Da(\Rn))\subseteq\ell^{p}(K;\Hps)$ lies dense, as remarked in Section \ref{subsec:defseq}, so there exists a sequence  $(F_{j})_{j=0}^{\infty}\subseteq c_{00}(K;\Da(\Rn))$ such that $F_{j}\to Qu$ in $\ell^{p}(K;\Hps)$, as $j\to\infty$. Then, by Propositions \ref{prop:Qproperties} and \ref{prop:Banach}, $(PF_{j})_{j=0}^{\infty}\subseteq \Da(M,\Omega_{1/2})$ and
\[
\|PF_{j}-u\|_{\HpsM}=\|PF_{j}-PQu\|_{\HpsM}\lesssim \|F_{j}-Qu\|_{\ell^{p}(K;\Hps)}\to0
\]
as $j\to\infty$.
\end{proof}

\begin{remark}\label{rem:densityM}
It was shown in \cite[Proposition 6.6]{HaPoRo20} (for $s=0$, which suffices) that $\Sw(\Rn)\subseteq \Hps$ for all $p\in[1,\infty]$ and $s\in\R$. The first part of Proposition \ref{prop:densityM} could be strengthened to also cover this statement, by considering $u\in\Da'(M,\Omega_{1/2})$ for which each of the Schwartz seminorms of $(Qu)_{\ka}=\ka_{*}(\psi_{\ka}u)\in\Sw(\Rn)$ is rapidly decreasing in $\ka\in K$.
\end{remark}

\subsection{Interpolation and duality}\label{subsec:interdual}

In this subsection we extend interpolation and duality properties of the Hardy spaces for FIOs from $\Hps$ to $\HpsM$.

First, we prove that the Hardy spaces for FIOs form a complex interpolation scale.

\begin{proposition}\label{prop:interpM}
Let $p_{0},p_{1},p\in[1,\infty]$, $s_{0},s_{1},s\in \R$ and $\theta\in[0,1]$ be such that $\frac{1}{p}=\frac{1-\theta}{p_{0}}+\frac{\theta}{p_{1}}$, $s=(1-\theta)s_{0}+\theta s_{1}$ and $(p_{0},p_{1})\neq (\infty,\infty)$. Then
\[
[\HT^{s_{0},p_{0}}_{FIO}(M),\HT^{s_{1},p_{1}}_{FIO}(M)]_{\theta}=\HpsM
\]
with equivalent norms.
\end{proposition}
\begin{proof}
By \eqref{eq:interpRn} and by a basic fact about interpolation of vector-valued function spaces (see \cite[Theorem 2.2.6]{HyNeVeWe16}), one has
\[
[\ell^{p_{0}}(K;\HT^{s_{0},p_{0}}_{FIO}(\Rn)),\ell^{p_{1}}(K;\HT^{s_{1},p_{1}}_{FIO}(\Rn))]_{\theta}=\ell^{p}(K;\HT^{s,p}_{FIO}(\Rn)),
\]
for $(p_{0},p_{1})\neq (\infty,\infty)$. Then one can use interpolation of complemented subspaces, cf.~\cite[Theorem 1.17.1.1]{Triebel78}, combined with Corollary \ref{cor:boundedPQ}, to conclude that
\[
[QP\ell^{p_{0}}(K;\HT^{s_{0},p_{0}}_{FIO}(\Rn)),QP\ell^{p_{1}}(K;\HT^{s_{1},p_{1}}_{FIO}(\Rn))]_{\theta}=QP\ell^{p}(K;\HT^{s,p}_{FIO}(\Rn))
\]
with equivalence of norms. Since $Q:\HT^{t,q}_{FIO}(M)\to QP\ell^{q}(K;\HT^{t,q}_{FIO}(\Rn))$ is an isomorphism for all $q\in[1,\infty]$ and $t\in\R$, this proves the required statement.
\end{proof}

\begin{remark}\label{rem:dualinfty}
The reason for the restriction $(p_{0},p_{1})\neq (\infty,\infty)$ in Proposition \ref{prop:interpM} is that we used a general fact about interpolation of Bochner spaces, \cite[Theorem 2.2.6]{HyNeVeWe16}, the proof of which relies on  approximation by simple functions. Since the finite sequences do not lie dense in $\ell^{\infty}$, this approach does not immediately extend to $p_{0}=p_{1}=\infty$. For the sake of brevity and because the latter case plays no real role in this article, we will not attempt to extend Proposition \ref{prop:interpM}. We do note that Proposition \ref{prop:dualM} below and \cite[Lemma 1.11.3]{Triebel78} immediately yield
\[
[\HT^{s_{0},\infty}_{FIO}(M),\HT^{s_{1},\infty}_{FIO}(M)]_{\theta}\subseteq \HT^{s,\infty}_{FIO}(M).
\]
Moreover, if $M$ is compact and $K$ is finite, then Proposition \ref{prop:interpM} also holds for $p_{0}=p_{1}=\infty$, because in this case $\ell^{\infty}(K,\HT^{s,\infty}_{FIO}(\Rn))=\ell^{p}(K;\HT^{s,\infty}_{FIO}(\Rn))$ for any $1\leq p<\infty$, so that one can apply \cite[Theorem 2.2.6]{HyNeVeWe16} after all.
\end{remark}

Next, we extend the duality statement in \eqref{eq:dualRn} from $\Hps$ to $\HpsM$. We write $\lb f,g\rb_{s,p}$ for the duality pairing between $f\in\HT^{-s,p'}_{FIO}(\Rn)$ and $g\in\Hps$, which in turn is equal to the standard duality $\lb f,g\rb_{\Rn}$ if $g\in\Sw(\Rn)$. Also recall the duality pairing $\lb \cdot,\cdot\rb_{K}$ for vector-valued sequence spaces over $K$ from \eqref{eq:dualityK}.

\begin{proposition}\label{prop:dualM}
Let $p\in[1,\infty)$ and $s\in\R$. Then
\[
(\HpsM)^{*}=\HT^{-s,p'}_{FIO}(M)
\]
with equivalent norms, where the duality pairing is given by
\begin{equation}\label{eq:dualitypair}
\lb Qu,Qv\rb_{K}=\sum_{\ka\in K}\lb \ka_{*}(\psi_{\ka}u),\ka_{*}(\psi_{\ka}v)\rb_{s,p},
\end{equation}
for $u \in\HT^{-s,p'}_{FIO}(M)$ and $v\in\HpsM$, and by $\lb u,v\rb_{M}$ if $v\in\Da(M,\Omega_{1/2})$.
\end{proposition}
\begin{proof}
For the most part, the statement follows from Proposition \ref{prop:Banach} and the characterization of the dual of a quotient space, but we argue slightly more explicitly.

First let $u\in \HT^{-s,p'}_{FIO}(M)$. Since $\HT^{-s,p'}_{FIO}(\Rn)=(\Hps)^{*}$, \eqref{eq:dualitypair} defines a functional on $\Hps$ such that
\begin{align*}
&\Big|\sum_{\ka\in K}\lb \ka_{*}(\psi_{\ka}u),\ka_{*}(\psi_{\ka}v)\rb_{s,p}\Big|\lesssim \sum_{\ka\in K}\|\ka_{*}(\psi_{\ka}u)\|_{\HT^{-s,p'}_{FIO}(\Rn)}\|\ka_{*}(\psi_{\ka}v)\|_{\Hps}\\
&\leq \Big(\sum_{\ka\in K}\|\ka_{*}(\psi_{\ka}u)\|_{\HT^{-s,p'}_{FIO}(\Rn)}^{p'}\Big)^{1/p'}\Big(\sum_{\ka\in K}\|\ka_{*}(\psi_{\ka}v)\|_{\Hps}^{p}\Big)^{1/p}\\
&=\|Qu\|_{\ell^{p'}(K;\HT^{-s,p'}_{FIO}(\Rn))}\|Qv\|_{\ell^{p}(K;\Hps)}=\|u\|_{\HT^{-s,p'}_{FIO}(M)}\|v\|_{\HpsM},
\end{align*}
for an implicit constant independent of $u$ and $v\in\Hps$, with the obvious notational modification for $p'=\infty$. Moreover, by Proposition \ref{prop:Qproperties}, if $v\in\Da(M,\Omega_{1/2})$ then $\lb Qu,Qv\rb_{K}=\lb u,v\rb_{M}$. In particular, if $\lb Qu,Qv\rb_{K}=0$ for all $v\in\HpsM$, then $u=0$. Hence $\HT^{-s,p'}_{FIO}(M)\subseteq (\HpsM)^{*}$.

Conversely, let $l\in (\HpsM)^{*}$. Then $l\circ P\in (\ell^{p}(K;\Hps)^{*}$, by Proposition \ref{prop:Banach}.  Moreover, since $(\Hps)^{*}=\HT^{-s,p'}_{FIO}(\Rn)$, \cite[Proposition 1.3.3]{HyNeVeWe16} yields
\begin{equation}\label{eq:dualseq}
(\ell^{p}(K;\Hps))^{*}=\ell^{p'}(K;\HT^{-s,p'}_{FIO}(\Rn)),
\end{equation}
with the duality pairing $\lb\cdot,\cdot\rb_{K}$. Hence there exists an $F\in \ell^{p'}(K;\HT^{-s,p'}_{FIO}(\Rn))$ such that $l(P(G))=\lb F,G\rb_{K}$ for all $G\in\ell^{p}(K;\Hps))$. Set $u:=PF$. Then $u\in\HT^{-s,p'}_{FIO}(M)$, by Proposition \ref{prop:Banach}. Now Proposition \ref{prop:Qproperties} implies that
\[
\lb u,v\rb_{M}=\lb PF,v\rb_{M}=\lb F,Qv\rb_{K}=l(PQv)=l(v)
\]
for all $v\in\Da(M,\Omega_{1/2})$. Finally, Propositions \ref{prop:Qproperties} and \ref{prop:Banach} yield
\begin{align*}
|\lb Qu,G\rb_{K}|&=|\lb u,PG\rb_{M}|=|\lb F,QPG\rb_{K}|=|l(PQPG)|=|l(PG)|\\
&\leq \|l\| \|PG\|_{\HpsM}\lesssim \|l\| \|G\|_{\ell^{p}(K;\Hps)}
\end{align*}
for all $G\in c_{00}(K;\Da(\Rn))$. This concludes the proof, since it follows from \eqref{eq:dualseq} that
\[
\|u\|_{\HT^{-s,p'}_{FIO}(M)}=\|Qu\|_{\ell^{p}(K;\HT^{-s,p'}_{FIO}(\Rn))}\eqsim \sup |\lb Qu,G\rb_{K}|,
\]
where the supremum is over all $G\in c_{00}(K;\Da(\Rn))$ with $\|G\|_{\ell^{p}(K;\Hps)}\leq 1$.
\end{proof}

\begin{remark}\label{rem:propclassical}
One can also apply the approach which we have been using for $\HpsM$ to $\HT^{s,p}(M)$. In particular, the natural analogues of Propositions \ref{prop:Banach}, \ref{prop:independence}, \ref{prop:densityM}, \ref{prop:interpM} and \ref{prop:dualM}, as well as Lemma \ref{lem:L2}, hold for $\HT^{s,p}(M)$. This is because the proofs of those results all relied on combining properties of the function spaces on $\Rn$ with Proposition \ref{prop:Qproperties} and Corollary \ref{cor:boundedPQ}, and we already noted in Remark \ref{rem:seqclassical} that a version of Corollary \ref{cor:boundedPQ} holds for $\HT^{s,p}(M)$. See \cite{Taylor09} for another approach to such results.
\end{remark}

\subsection{Sobolev embeddings and Fourier integral operators}\label{subsec:FIOembed}

In this subsection we extend the Sobolev embeddings for $\Hps$ to $\HpsM$, and we show that $\HpsM$ is invariant under suitable FIOs.

We first extend the Sobolev embeddings in \eqref{eq:Sobolev} from $\Hps$ to $\HpsM$.

\begin{theorem}\label{thm:SobolevM}
Let $p\in[1,\infty]$ and $s\in\R$. Then
\begin{equation}\label{eq:SobolevM}
\HT^{s+s(p),p}(M)\subseteq \HpsM\subseteq \HT^{s-s(p),p}(M).
\end{equation}
\end{theorem}
Combined with Corollary \ref{cor:Sobolevman} below, \eqref{eq:SobolevM} implies \eqref{eq:Sobolevintro}.
\begin{proof}
This follows directly by arguing in local coordinates. That is, by \eqref{eq:Sobolev},
\[
\|\ka_{*}(\psi_{\ka}u)\|_{\HT^{s-s(p),p}(\Rn)}\lesssim \|\ka_{*}(\psi_{\ka}u)\|_{\HT^{s,p}_{FIO}(\Rn)}\lesssim \|\ka_{*}(\psi_{\ka}u)\|_{\HT^{s+s(p),p}(\Rn)}
\]
for implicit constants independent of $u\in\Da'(M,\Omega_{1/2})$ and $\ka\in K$, whenever the relevant quantities are finite.
\end{proof}

We note that the exponents in \eqref{eq:SobolevM} are sharp, for all $p\in[1,\infty]$ and $s\in\R$, on each Riemannian manifold $(M,g)$ with bounded geometry. This follows from the fact that the ``loss of $2s(p)$ derivatives" in \cite{SeSoSt91} is sharp on every compact manifold.

Next, let $(N,g')$ be another complete
 Riemannian manifold with bounded geometry. In the case where $(N,g')=(M,g)$ and $m=0$, the following theorem extends the invariance of $\Hps$ under suitable Fourier integral operators to $\HpsM$.

\begin{theorem}\label{thm:FIObdd}
Let $T\in I^{m}(M,N;\Ca)$, for $m\in\R$ and $\Ca$ a local canonical graph, and suppose that $T$ is compactly supported. Then
\[
T:\HT^{s+m,p}_{FIO}(M)\to \HT^{s,p}_{FIO}(N)
\]
is bounded for all $p\in[1,\infty]$ and $s\in\R$.
\end{theorem}
\begin{proof}
Let $Q':\HT^{s,p}_{FIO}(N)\to\ell^{p}(L;\HT^{s,p}_{FIO}(\Rn))$ be the isometry from \eqref{eq:Qprime}. By Proposition \ref{prop:Qproperties} \eqref{it:Qproperties4}, one has $Q'Tu=Q'TPQu$ for all $u\in\HT^{s+m,p}_{FIO}(M)$. Given that $Q:\HT^{s+m,p}_{FIO}(M)\to \ell^{p}(K;\HT^{s+m,p}_{FIO}(\Rn))$ is an isometry as well, it thus suffices to prove that $Q'TP:\ell^{p}(K;\HT^{s+m,p}_{FIO}(\Rn))\to \ell^{p}(L;\Hps)$ is bounded.

To this end, we show that the conditions of Theorem \ref{thm:FIOseq} are satisfied. Note that, since $T$ is compactly supported, \eqref{it:FIOseq2} automatically holds, and for \eqref{it:FIOseq1} we may fix $\ka\in K$ and $\la\in L$ and ignore the dependence of the constants on $\ka$ and $\la$. Moreover, the operators $f\mapsto \ka^{*}(\psit_{\ka}f)$ and $v\mapsto \la_{*}(\theta_{\la}v)$ are compactly supported FIOs of order zero, associated with local canonical graphs, as follows e.g.~from the proof of Corollary \ref{cor:coordchange}, after representing the operators in local coordinates. Hence, by the composition theorem for FIOs, $f\mapsto \la_{*}(\theta_{\la}T\ka^{*}(\psit_{\ka}f))$ is a compactly supported element of $I^{m}(\Rn,\Rn;\Ca')$, associated with a local canonical graph $\Ca'$. Proposition \ref{prop:FIORn} concludes the proof.
\end{proof}

\subsection{An atomic decomposition}\label{subsec:atomicman}

In this subsection we provide an atomic decomposition of $\HT^{s,1}_{FIO}(M)$. This decomposition is a direct consequence of that of $\HT^{s,1}_{FIO}(\Rn)$, in Appendix \ref{sec:atomic}, by working in local coordinates. For the sake of brevity and because we will typically work in local coordinates anyway, we focus on the decomposition as such, instead of introducing terminology for the resulting atoms on $M$ and developing a full atomic theory. We use notation and terminology as in Appendix \ref{sec:atomic}. In particular, $r\in C^{\infty}_{c}(\Rn)$ is such that $r(\xi)=1$ if $|\xi|\leq 1$, and $r(\xi)=0$ if $|\xi|\geq 2$.

\begin{proposition}\label{prop:atomicHpFIOM}
Let $s\in\R$. Then the collection of all $\ka^{*}(\psit_{\ka}f)$, for $\ka\in K$ and $f$ an $\HT^{s,1}_{FIO}(\Rn)$-atom associated with a ball of radius at most 1, is a uniformly bounded subset of $\HT^{s,1}_{FIO}(M)$. Moreover, for each $t\in\R$ there exists a $C_{t}\geq0$ with the following property. For all $u\in\HT^{s,1}_{FIO}(M)$ and $\ka\in K$ there exist a $v\in\HT^{t,1}_{FIO}(M)$, a sequence $(f_{j,\ka})_{j=1}^{\infty}$ of $\HT^{s,1}_{FIO}(\Rn)$-atoms associated with balls of radius at most $1$, and an $(\alpha_{j,\ka})_{j=1}^{\infty}\in\ell^{1}$, such that
\begin{equation}\label{eq:atomdecompM}
u=v+\sum_{\ka\in K}\sum_{j=1}^{\infty}\alpha_{j,\ka}\ka^{*}(\psit_{\ka}f_{j,\ka})
\end{equation}
and
\[
\|v\|_{\HT^{t,1}_{FIO}(M)}+\sum_{\ka\in K}\sum_{j=1}^{\infty}|\alpha_{j,\ka}|\leq C_{t}\|u\|_{\HT^{s,1}_{FIO}(M)}.
\]
In fact, one may choose
\begin{equation}\label{eq:defv}
v:=\sum_{\ka\in K}\ka^{*}(\psit_{\ka}r(D)\ka_{*}(\psi_{\ka}u)).
\end{equation}
\end{proposition}
\begin{proof}
To prove the first statement, let $f$ be an $\HT^{s,1}_{FIO}(\Rn)$-atom associated with a ball of radius at most $1$, and let $\ka\in K$. Set $K(\ka):=\{\la\in K\mid U_{\ka}\cap U_{\la}\neq\emptyset\}$ and note that $|K(\ka)|$ is uniformly bounded in $\ka$, by Lemma \ref{lem:cover} \eqref{it:cover2}. Hence one can use Corollary \ref{cor:coordchange}, Remark \ref{rem:constantsFIO}, Lemma \ref{lem:cover} \eqref{it:cover4} and \eqref{eq:uniformchange} to write
\begin{align*}
&\|\ka^{*}(\psit_{\ka}f)\|_{\HT^{s,1}_{FIO}(M)}=\sum_{\la\in K(\ka)}\|\la_{*}(\psi_{\la}\ka^{*}(\psit_{\ka}f))\|_{\HT^{s,1}_{FIO}(\Rn)}\\
&=\sum_{\la\in K(\ka)}\|\psit_{\la}\mu_{\ka\la}^{*}(\psit_{\ka}f)\|_{\HT^{s,1}_{FIO}(\Rn)}\lesssim \sum_{\la\in K(\ka)}\|f\|_{\HT^{s,1}_{FIO}(\Rn)}\lesssim \|f\|_{\HT^{s,1}_{FIO}(\Rn)}\lesssim 1,
\end{align*}
where for the final inequality we used the first statement in Proposition \ref{prop:atom}.

Next, fix $u\in\HT^{s,1}_{FIO}(M)$ and let $\ka\in K$. Then $\ka_{*}(\psi_{\ka}u)\in\HT^{s,1}_{FIO}(\Rn)$, so Proposition \ref{prop:atom} yields a collection $(f_{j,\ka})_{j=1}^{\infty}$ of $\HT^{s,1}_{FIO}(\Rn)$-atoms, associated with balls of radius at most $1$, and an $(\alpha_{j,\ka})_{j=1}^{\infty}\in\ell^{1}$, such that
\[
\ka_{*}(\psi_{\ka}u)=r(D)(\ka_{*}(\psi_{\ka}u))
+\sum_{j=1}^{\infty}\alpha_{j,\ka}f_{j,\ka}.
\]
Let $v$ be as in \eqref{eq:defv}. Then we can use that $(\psi_{\ka}^{2})_{\ka\in K}$ is a partition of unity to obtain \eqref{eq:atomdecompM}:
\[
u=\sum_{\ka\in K}\ka^{*}(\psit_{\ka}\ka_{*}(\psi_{\ka}u))=v+\sum_{\ka\in K}\sum_{j=1}^{\infty}\alpha_{j,\ka}\ka^{*}(\psit_{\ka}f_{j,\ka}).
\]
Moreover, again by Proposition \ref{prop:atom},
\[
\sum_{\ka\in K}\sum_{j=1}^{\infty}|\alpha_{j,\ka}|\lesssim \sum_{\ka\in K}\|\ka_{*}(\psi_{\ka}u)\|_{\HT^{s,1}_{FIO}(\Rn)}=\|u\|_{\HT^{s,1}_{FIO}(M)}.
\]
Also, for each $t\in\R$,
\begin{align*}
&\|v\|_{\HT^{t,1}_{FIO}(\Rn)}\leq \sum_{\ka\in K}\|\ka^{*}(\psit_{\ka}r(D)\ka_{*}(\psi_{\ka}u))\|_{\HT^{t,1}_{FIO}(\Rn)}\\
&\lesssim \sum_{\ka\in K}\|r(D)\ka_{*}(\psi_{\ka}u))\|_{\HT^{t,1}_{FIO}(\Rn)}\lesssim\sum_{\ka\in K}\|\ka_{*}(\psi_{\ka}u))\|_{\HT^{s,1}_{FIO}(\Rn)}=\|u\|_{\HT^{s,1}_{FIO}(M)}.
\end{align*}
Here we used that $r\in C^{\infty}_{c}(\Rn)$, in the penultimate step.
\end{proof}

In the same manner, relying instead on the atomic decomposition of $\HT^{s,1}(\Rn)$ from Lemma \ref{lem:atomHp}, one obtains an atomic decomposition of $\HT^{s,1}(M)$. For alternative approaches to atomic decompositions of $\HT^{s,1}(M)$ we refer to \cite{Skrzypczak98,Taylor09}.

\subsection{Equivalent norms}\label{subsec:equivnorms}

In this subsection we provide equivalent norms on both $\HpsM$ and $\HT^{s,p}(M)$.

We first show that one could alternatively define $\HpsM$ by replacing $(\psi_{\ka})_{\ka\in K}$ in \eqref{eq:defHpFIOnorm} by the partition of unity $(\psi_{\ka}^{2})_{\ka\in K}$. Nonetheless, we have chosen to work with $(\psi_{\ka})_{\ka\in K}$ because it is more natural for the $L^{2}$ theory and for duality purposes, cf.~Remark \ref{rem:L2} and Proposition \ref{prop:dualM}, and because the operators $Q$ and $P$ have better mapping properties when defined in the current symmetric manner, cf.~Proposition \ref{prop:Qproperties} and the proofs in the rest of this section.

\begin{proposition}\label{prop:square}
Let $p\in[1,\infty]$ and $s\in\R$. Then there exists a $C>0$ with the following property. A $u\in \Da'(M,\Omega_{1/2})$ satisfies $u\in\HpsM$ if and only if $(\ka_{*}(\psi_{\ka}^{2}u))_{\ka\in K}\in \ell^{p}(K;\Hps)$, in which case
\begin{equation}\label{eq:square}
\frac{1}{C}\|u\|_{\HpsM}\leq \|(\ka_{*}(\psi_{\ka}^{2}u))_{\ka\in K}\|_{\ell^{p}(K;\Hps)}\leq C\|u\|_{\HpsM}.
\end{equation}
\end{proposition}
\begin{proof}
The right-hand side of \eqref{eq:square} follows from the estimate
\[
\|\ka_{*}(\psi_{\ka}^{2}u)\|_{\Hps}=\|\psit_{\ka}\ka_{*}(\psi_{\ka}u)\|_{\Hps}\lesssim \|\ka_{*}(\psi_{\ka}u)\|_{\Hps}
\]
for each $\ka\in K$, for which we used Corollary \ref{cor:coordchange}, Remark \ref{rem:constantsFIO} and Lemma \ref{lem:cover} \eqref{it:cover4}.

On the other hand, for $\ka\in K$ we have, with $K(\ka):=\{\la\in K\mid U_{\ka}\cap U_{\la}\neq\emptyset\}$,
\[
\ka_{*}(\psi_{\ka}u)=\sum_{\la\in K(\ka)}\ka_{*}(\psi_{\ka}\psi_{\la}^{2}u)=\sum_{\la\in K(\ka)}\psit_{\ka}\mu_{\la\ka}^{*}\la_{*}(\psi_{\la}^{2}u).
\]
By Lemma \ref{lem:cover}, $|K(\ka)|$ is uniformly bounded in $\ka$. Hence we can combine a minor modification of Corollary \ref{cor:coordchange} and Remark \ref{rem:constantsFIO} with H\"{o}lder's inequality:
\begin{align*}
\|\ka_{*}(\psi_{\ka}u)\|_{\Hps}&\lesssim  \sum_{\la\in K(\ka)}\|\la_{*}(\psi_{\la}^{2}u)\|_{\Hps}\\
&\lesssim \Big(\sum_{\la\in K(\ka)}\|\la_{*}(\psi_{\la}^{2}u)\|_{\Hps}^{p}\Big)^{1/p},
\end{align*}
with the obvious notational modification for $p=\infty$.
This in turn yields
\begin{align*}
\|u\|_{\HpsM}&=\Big(\sum_{\ka\in K}\|\ka_{*}(\psi_{\ka}u)\|_{\Hps}^{p}\Big)^{1/p}\\
&\lesssim \Big(\sum_{\ka\in K}\sum_{\la\in K(\ka)}\|\la_{*}(\psi_{\la}^{2}u)\|_{\Hps}^{p}\Big)^{1/p}\\
&=\Big(\sum_{\la\in K}\sum_{\ka\in K(\la)}\|\la_{*}(\psi_{\la}^{2}u)\|_{\Hps}^{p}\Big)^{1/p}\\
&\lesssim \Big(\sum_{\la\in K}\|\la_{*}(\psi_{\la}^{2}u)\|_{\Hps}^{p}\Big)^{1/p},
\end{align*}
for implicit constants independent of $u$. This concludes the proof.
\end{proof}

\begin{remark}\label{rem:squareHp}
The same proof yields an analogue of Proposition \ref{prop:square} for $\HT^{s,p}(M)$. In fact, it follows as in the proof that, for all $p\in[1,\infty]$ and $s,\gamma\in\R$, one has
\[
\|((\ka_{*})_{\gamma}(\psi_{\ka}^{2}u))_{\ka\in K}\|_{\ell^{p}(K;\HT^{s,p}(\Rn))}\eqsim \|((\ka_{*})_{\gamma}(\psi_{\ka}u))_{\ka\in K}\|_{\ell^{p}(K;\HT^{s,p}(\Rn))}
\]
for all $u\in\Da'(M,\Omega_{\gamma})$ such that either of these quantities is finite. Here $\Da'(M,\Omega_{\gamma})$ is the space of distributional densities of order $\gamma$ (see Section \ref{subsec:densities}).
\end{remark}

\begin{remark}\label{rem:Triebel}
For $p<\infty$, $\HT^{s,p}(M)$ was defined in \cite{Triebel86,Triebel87,Triebel92} to consist of all $v\in\Da'(M,\Omega_{1})$ such that $
\|((\ka_{*})_{1}(\psi_{\ka}^{2}v))_{\ka\in K}\|_{\ell^{p}(K;\HT^{s,p}(\Rn))}<\infty$.
Recall that $\rho_{g}\in C^{\infty}(M,\Omega_{1})$ is the Riemannian density on $M$. For each $\ka\in K$ one then has
\begin{align*}
\|(\ka_{*})_{1}(\psi_{\ka}^{2}\rho_{g}^{1/2}u)\|_{\HT^{s,p}(\Rn)}&=\|(\ka_{*})_{1/2}(\psi_{\ka}\rho_{g}^{1/2})\cdot (\ka_{*})_{1/2}(\psi_{\ka}u)\|_{\HT^{s,p}(\Rn)}\\
&\lesssim \|(\ka_{*})_{1/2}(\psi_{\ka}u)\|_{\HT^{s,p}(\Rn)}
\end{align*}
for all $u\in \Da'(M,\Omega_{1/2})$, and
\begin{align*}
\|(\ka_{*})_{1/2}(\psi_{\ka}^{2}\rho_{g}^{-1/2}v)\|_{\HT^{s,p}(\Rn)}&=\|(\ka_{*})_{-1/2}(\psi_{\ka}\rho_{g}^{-1/2})\cdot (\ka_{*})_{1}(\psi_{\ka}v)\|_{\HT^{s,p}(\Rn)}\\
&\lesssim \|(\ka_{*})_{1}(\psi_{\ka}v)\|_{\HT^{s,p}(\Rn)}
\end{align*}
for all $v\in\Da'(M,\Omega_{1})$. Here we used Lemma \ref{lem:cover} \eqref{it:cover4}, as well as \eqref{eq:localg} to estimate away powers of $\rho_{g}$. It now follows from Remark \ref{rem:squareHp} that $u\mapsto \rho_{g}^{1/2}u$ is an isomorphism between the $\HT^{s,p}(M)$   from Definition  \ref{def:Hpman} and the one from \cite{Triebel86,Triebel87,Triebel92}.
\end{remark}

As a corollary of the previous remark, we can now show that $\HT^{s,p}(M)$ coincides with $W^{s,p}(M)$, as introduced in Definition \ref{def:Sobolevman}, for $1<p<\infty$. We also give a more concrete definition of these spaces for $s\in\Z_{+}$, in terms of the $k$-th covariant derivative $\nabla^{k}$ for $k\in\{0,\ldots, s\}$. Recall that, for $f\in C^{\infty}(M)$,
\[
|\nabla^{k}f|^{2}=g^{\alpha_{1}\beta_{1}}\cdots g^{\alpha_{k}\beta_{k}}\nabla_{\alpha_{1}}\cdots\nabla_{\alpha_{k}}f\cdot\nabla_{\beta_{1}}\cdots\nabla_{\beta_{k}}\overline{f}
\]
using the Einstein summation convention, where each $\nabla_{j}$ is a covariant derivative with respect to a given local chart.

\begin{corollary}\label{cor:Sobolevman}
Let $p\in(1,\infty)$ and $s\in\R$. Then $\HT^{s,p}(M)=W^{s,p}(M)$, with equivalent norms. Moreover, for each $s\in\Z_{+}$ there exists a $C>0$ such that
\begin{equation}\label{eq:covariant}
\frac{1}{C}\|u\|_{\HT^{s,p}(M)}\leq \max_{k\in\{0,\ldots,s\}}\Big(\int_{M}|\nabla^{k}(\rho_{g}^{-1/2}u)(x)|^{p}\rho_{g}(x)\Big)^{1/p}\leq C\|u\|_{\HT^{s,p}(M)}
\end{equation}
for all $u\in\Da(M,\Omega_{1/2})$.
\end{corollary}
\begin{proof}
Simply combine Remark \ref{rem:Triebel} with \cite[Theorem 7.4.5]{Triebel92}.
\end{proof}

In particular, $\HT^{s,p}(M)$, for $1<p<\infty$ and $s\in\Z_{+}$, is the completion of $\Da(M,\Omega_{1/2})$ with respect to the norm in \eqref{eq:covariant}.

\section{Local smoothing}\label{sec:localsmooth}

In this section we prove our main local smoothing result for Fourier integral operators satisfying the cinematic curvature condition, Theorem \ref{thm:localsmoothFIOintro}. To do so, we first derive a version of the variable-coefficient Wolff-type inequalities from \cite{BeHiSo20}, with the Hardy spaces for FIOs as initial data.

Throughout this section, we fix a dimension $n\geq2$. The phenomenon of local smoothing does not occur for $n=1$, since in that case the relevant operators are already bounded on the classical function spaces $\HT^{s,p}(M)$ at each fixed time.

\subsection{Variable-coefficient Wolff-type inequalities}\label{subsec:decouple}

In this subsection we will use the variable-coefficient Wolff-type inequalities from \cite{BeHiSo20} to derive a statement involving the Hardy spaces for Fourier integral operators. In Appendix \ref{sec:Wolff} we show that the resulting inequalities are, in a sense, equivalent to those in \cite{BeHiSo20}, but for our main results we will only need one implication.

Throughout this subsection, we fix $m\in\R$ and symbols $a_{1}\in C^{\infty}_{c}(\R^{n+1})$ and $a_{2}\in S^{m}(\Rn)$, where $S^{m}(\Rn)$ is the standard Kohn--Nirenberg class of symbols depending only on the fiber variable.
Set $a(z,\eta):=a_{1}(z)a_{2}(\eta)$ for $(z,\eta)\in\R^{n+1}\times\Rn$. Next, let $\Phi\in C^{\infty}(\R^{n+1}\times (\Rn\setminus \{0\}))$ be real-valued and positively homogeneous of degree one in the fiber variable. Suppose that $\rank\,\partial_{x\eta}^{2}\Phi(z_{0},\eta_{0})=n$ for all $(z_{0},\eta_{0})\in\supp(a)$ with $\eta_{0}\neq 0$, where we write $z=(x,t)\in\Rn\times\R$ for each $z\in\R^{n+1}$. Moreover, we assume that for all such $(z_{0},\eta_{0})$, one has
\begin{equation}\label{eq:Gaussmap}
\rank\,\partial_{\eta\eta}^{2}(\partial_{z}\Phi(z_{0},\eta)\cdot G(z_{0},\eta_{0}))|_{\eta=\eta_{0}}=n-1,
\end{equation}
where $G(z_{0},\eta_{0}):=G_{0}(z_{0},\eta_{0})/|G_{0}(z_{0},\eta_{0})|$ is the generalized Gauss map, with
\[
G_{0}(z_{0},\eta_{0}):=\bigwedge_{i=1}^{n}\partial_{\eta_{i}}\partial_{z}\Phi(z_{0},\eta_{0}).
\]
Now set
\begin{equation}\label{eq:Tstandard}
Tf(z):=\int_{\Rn}e^{i\Phi(z,\eta)}a(z,\eta)\wh{f}(\eta)\ud\eta
\end{equation}
for $f\in\Sw(\Rn)$ and $z\in\R^{n+1}$. We note that $T$ is a Fourier integral operator of order $m-1/4$, associated with a canonical relation $\Ca$ satisfying the cinematic curvature condition (see Section \ref{subsec:FIOs}). Later on we will use that the analysis of general FIOs satisfying this condition can be reduced to ones as in \eqref{eq:Tstandard}, by Lemma \ref{lem:FIOcurv}.

Next, for each $k\in\Z_{+}$, we fix a maximal collection $\Theta_{k}\subseteq S^{n-1}$ of unit vectors such that $|\nu-\nu'|\geq 2^{-k/2}$ for all $\nu,\nu'\in \Theta_{k}$ with $\nu\neq \nu'$. Let $(\chi_{\nu})_{\nu\in\Theta_{k}}\subseteq C^{\infty}(\Rn\setminus\{0\})$ be an associated partition of unity. More precisely, each $\chi_{\nu}$ is positively homogeneous of degree $0$ and satisfies $0\leq \chi_{\nu}\leq 1$ and
\[
\supp(\chi_{\nu})\subseteq\{\eta\in\Rn\setminus\{0\}\mid |\hat{\eta}-\nu|\leq 2^{-k/2+1}\}.
\]
Moreover, $\sum_{\nu\in \Theta_{k}}\chi_{\nu}(\eta)=1$ for all $\eta\neq0$, and for all $\alpha\in\Z_{+}^{n}$ and $\beta\in\Z_{+}$ there exists a $C_{\alpha,\beta}\geq0$ independent of $k$ such that, if $2^{k-1}\leq |\eta|\leq 2^{k+1}$, then
\[
|(\hat{\eta}\cdot\partial_{\eta})^{\beta}\partial_{\eta}^{\alpha}\chi_{\nu}(\eta)|\leq C_{\alpha,\beta}2^{-k(|\alpha|/2+\beta)}
\]
for all $\nu\in \Theta_{k}$. Such a collection is straightforward to construct, in a similar manner as the wave packets in Section \ref{subsec:defRn} (see also \cite[Section IX.4]{Stein93}).

Now, for $p\in[1,\infty]$, $k\in\Z_{+}$ and $f\in\Sw(\Rn)$ with $\supp(\wh{f}\,)\subseteq\{\eta\in\Rn\mid 2^{k-1}\leq |\eta|\leq 2^{k+1}\}$, define the \emph{$k$-th decoupling norm} of $Tf$ as
\begin{equation}\label{eq:decouplenorm}
\|Tf\|_{L^{p,k}_{\dec}(\R^{n+1})}:=\Big(\sum_{\nu\in \Theta_{k}}\|T\chi_{\nu}(D)f\|_{L^{p}(\R^{n+1})}^{p}\Big)^{1/p},
\end{equation}
with the obvious modification for $p=\infty$.

The main result of this section relies crucially on the following variable-coefficient Wolff-type inequality from \cite{BeHiSo20}. Recall the definition of the exponent $d(p)$ from \eqref{eq:dp}.

\begin{theorem}\label{thm:BeHiSo}
Let $p\in(2,\infty)$ and $\veps>0$. Then, for each $L\geq0$, there exists a $C\geq0$ such that
\begin{equation}\label{eq:BeHiSo}
\|Tf\|_{L^{p}(\R^{n+1})}\leq C\big(2^{k(d(p)+\veps)}\|Tf\|_{L^{p,k}_{\dec}(\R^{n+1})}+2^{-kL}\|f\|_{L^{2}(\Rn)}\big)
\end{equation}
for all $k\in\Z_{+}$ and $f\in\Sw(\Rn)$ with $\supp(\wh{f}\,)\subseteq\{\eta\in\Rn\mid 2^{k-1}\leq |\eta|\leq 2^{k+1}\}$.
\end{theorem}

\begin{remark}\label{rem:BeHiSo}
We note that \cite[Theorem 1.4]{BeHiSo20} has a slightly different form, with a straightforward rescaling argument connecting the two versions. On the other hand, Theorem \ref{thm:BeHiSo} follows directly from \cite[Theorem 53]{BeHiSo21}, where the term $2^{-kL}\|f\|_{L^{2}(\Rn)}$ is omitted. Theorem \ref{thm:BeHiSo} is also contained in \cite[p.~424]{BeHiSo20}, although there additional support assumptions are made on the symbol $a$. Those assumptions can be removed by applying suitable partitions of unity.
\end{remark}

We will now show that each of the decoupling norms can be bounded by the $\HT^{m-s(p),p}_{FIO}(\Rn)$ norm of the initial data.
The following proposition is a variable-coefficient version of \cite[Corollary 4.2]{Rozendaal22b}, which concerns the solution operator to the Euclidean half-wave equation and related Fourier multipliers.

\begin{proposition}\label{prop:decouple}
Let $p\in[1,\infty]$. Then there exists a $C\geq0$ such that
\begin{equation}\label{eq:decouplebound}
\|Tf\|_{L^{p,k}_{\dec}(\R^{n+1})}\leq C\|f\|_{\HT^{m-s(p),p}_{FIO}(\Rn)}
\end{equation}
for all $k\in\Z_{+}$ and $f\in\Sw(\Rn)$ with $\supp(\wh{f}\,)\subseteq\{\eta\in\Rn\mid 2^{k-1}\leq |\eta|\leq 2^{k+1}\}$.
\end{proposition}
For $T$ the solution operator to the Euclidean wave equation on a compact time interval, \eqref{eq:decouplebound} is in fact an equivalence. For general FIOs the corresponding statement is more involved, and it is relegated to Appendix \ref{sec:Wolff}.
\begin{proof}
By replacing $T$ by $T\lb D\rb^{-m}$, we may suppose that $m=0$.

Set $T_{t}f(x):=Tf(x,t)$ for $t\in\R$, $f\in \Sw(\Rn)$ and $x\in \Rn$. Then each $T_{t}$ is an FIO of order $0$ as in Corollary \ref{cor:stanform}, with uniformity in $t$ on its support, phase function and symbol, as in Remark \ref{rem:constantsFIO}. Suppose that $\supp(\wh{f}\,)\subseteq\{\eta\in\Rn\mid 2^{k-1}\leq |\eta|\leq 2^{k+1}\}$ for some $k\in\Z_{+}$, and let $\nu\in\Theta_{k}$. Then, by writing $\chi_{\nu}=\wt{\chi}_{\nu}\chi_{\nu}$ for $\wt{\chi}_{\nu}\in C^{\infty}(\Rn\setminus\{0\})$ with similar support and smoothness properties as $\chi_{\nu}$, a standard argument from the $L^{p}$ theory of FIOs, contained in e.g.~\cite[Section IX.4]{Stein93} and the proof of \cite[Lemma 3.3]{BeHiSo20}, shows that
\begin{equation}\label{eq:standardest}
\|T_{t}\chi_{\nu}(D)f\|_{L^{p}(\Rn)}=\|T_{t}\wt{\chi}_{\nu}(D)\chi_{\nu}(D)f\|_{L^{p}(\Rn)}\lesssim \|\chi_{\nu}(D)f\|_{L^{p}(\Rn)}
\end{equation}
for an implicit constant independent of $t$, $f$, $k$ and $\nu$. This in turn yields
\begin{align*}
\|T\chi_{\nu}(D)f\|_{L^{p}(\R^{n+1})}&=\Big(\int_{I}\|T_{t}\chi_{\nu}(D)f\|_{L^{p}(\Rn)}^{p}\ud t\Big)^{1/p}\\
&\lesssim \Big(\int_{I}\|\chi_{\nu}(D)f\|_{L^{p}(\Rn)}^{p}\ud t\Big)^{1/p}\lesssim \|\chi_{\nu}(D)f\|_{L^{p}(\Rn)},
\end{align*}
for some bounded interval $I\subseteq \R$ such that $T_{t}=0$ whenever $t\notin I$. Hence
\begin{align*}
\|Tf\|_{L^{p,k}_{\dec}(\R^{n+1})}&=\Big(\sum_{\nu\in \Theta_{k}}\|T\chi_{\nu}(D)f\|_{L^{p}(\R^{n+1})}^{p}\Big)^{1/p}\lesssim \Big(\sum_{\nu\in \Theta_{k}}\|\chi_{\nu}(D)f\|_{L^{p}(\R^{n})}^{p}\Big)^{1/p}\\
&\eqsim \|f\|_{\HT^{-s(p),p}_{FIO}(\Rn)},
\end{align*}
where we used \cite[Proposition 4.1]{Rozendaal22b} for the final step.
\end{proof}

We now arrive at a corollary which is crucial for the proof of our main result.

\begin{corollary}\label{cor:decouple}
Let $p\in(2,\infty)$ and $\veps>0$. Then there exists a $C\geq0$ such that
\begin{equation}\label{eq:decouplefull}
\|Tf\|_{L^{p}(\R^{n+1})}\leq C\|f\|_{\HT^{m+d(p)-s(p)+\veps,p}_{FIO}(\Rn)}
\end{equation}
for all $f\in\HT^{m+d(p)-s(p)+\veps,p}_{FIO}(\Rn)$.
\end{corollary}
\begin{proof}
We may again fix a specific value of $m$, although this time we choose $m=2s(p)-d(p)-\veps$. This choice will be convenient, as it ensures that
\begin{equation}\label{eq:Sobolevproof}
\HT^{m+d(p)-s(p)+\veps,p}_{FIO}(\Rn)=\HT^{s(p),p}_{FIO}(\Rn)\subseteq L^{p}(\Rn),
\end{equation}
by \eqref{eq:Sobolev}. Moreover, we may suppose that $a(z,\eta)=0$ whenever $|\eta|<1/2$, since the low frequencies only contribute a smoothing operator $R$, and
\begin{equation}\label{eq:smoothingR}
R:\HT^{m+d(p)-s(p)+\veps,p}_{FIO}(\Rn)\subseteq \Sw'(\Rn)\to\Sw(\R^{n+1})\subseteq L^{p}(\R^{n+1}).
\end{equation}
Let $f\in\HT^{s(p),p}_{FIO}(\Rn)$. Our first goal is to preemptively deal with the $L^{2}$ term in \eqref{eq:BeHiSo}, which we will do by writing $f$ as the sum of a compactly supported function and a remainder, on the latter of which the phase function $\Phi$ of $T$ is non-singular. The argument is similar to that in the proof of Corollary \ref{cor:stanform}.

More precisely, as in \eqref{eq:cnonsingular}, set
\[
c:=\sup\{|\partial_{\eta}\Phi(z,\eta)|\mid (z,\eta)\in \supp(a), \eta\neq 0\}<\infty,
\]
and let $\chi\in C^{\infty}_{c}(\Rn)$ be such that $\chi(y)=1$ if $|y|\leq 2c$. Then the phase function of $T(1-\chi)$ is non-singular, and integration by parts shows that $T(1-\chi)$ is a smoothing operator. Hence we only have to show that
\begin{equation}\label{eq:toshowcompact}
\|T\chi f\|_{L^{p}(\R^{n+1})}\lesssim \|f\|_{\HT^{s(p),p}_{FIO}(\Rn)}
\end{equation}
for an implicit constant independent of $f$.

To this end, fix a Littlewood--Paley decomposition $(\phi_{k})_{k=0}^{\infty}\subseteq C^{\infty}_{c}(\Rn)$, with inverse Fourier transforms uniformly bounded in $L^{1}(\Rn)$, such that $\supp(\phi_{k})\subseteq\{\eta\in\Rn\mid 2^{k-1}\leq |\eta|\leq 2^{k+1}\}$ for each $k\in\N$.
Then the compact support of $\chi$, combined with \eqref{eq:Sobolevproof}, yields
\[
\|\phi_{k}(D)\chi f\|_{L^{2}(\Rn)}\lesssim \|\chi f\|_{L^{2}(\Rn)}\lesssim \|f\|_{L^{p}(\Rn)}\\
\lesssim \|f\|_{\HT^{s(p),p}_{FIO}(\Rn)}
\]
for implicit constants independent of $k$ and $f$. Also note that
\[
\|\chi f\|_{\HT^{s(p),p}_{FIO}(\Rn)}\lesssim \|f\|_{\HT^{s(p),p}_{FIO}(\Rn)},
\]
by Corollary \ref{cor:coordchange}. Hence, for each $k$, we can apply Theorem \ref{thm:BeHiSo} and Proposition \ref{prop:decouple}, with $\veps$ replaced by $\veps/2$:
\begin{align*}
\|T\phi_{k}(D)\chi f\|_{L^{p}(\R^{n+1})}&\lesssim 2^{k(d(p)+\veps/2)}\|T\phi_{k}(D)\chi f\|_{L^{p,k}_{\dec}(\R^{n+1})}+2^{-kL}\|\phi_{k}(D)\chi f\|_{L^{2}(\Rn)}\\
&\lesssim 2^{k(d(p)+\veps/2)}\|\phi_{k}(D)\chi f\|_{\HT^{m-s(p),p}_{FIO}(\Rn)}+2^{-kL}\|f\|_{\HT^{s(p),p}_{FIO}(\Rn)}\\
&\eqsim 2^{-k\veps/2}\|\phi_{k}(D)\chi f\|_{\HT^{m+d(p)+\veps-s(p),p}_{FIO}(\Rn)}+2^{-kL}\|f\|_{\HT^{s(p),p}_{FIO}(\Rn)}\\
&\lesssim 2^{-k\veps/2}\|\chi f\|_{\HT^{s(p),p}_{FIO}(\Rn)}+2^{-kL}\|f\|_{\HT^{s(p),p}_{FIO}(\Rn)}\\
&\lesssim 2^{-k\veps/2}\|f\|_{\HT^{s(p),p}_{FIO}(\Rn)},
\end{align*}
if we choose $L\geq \veps/2$. Since $T\chi f=\sum_{k=0}^{\infty}T\phi_{k}(D)\chi f$, we thus obtain \eqref{eq:toshowcompact}:
\begin{align*}
\|T\chi f\|_{L^{p}(\R^{n+1})}&\leq \sum_{k=0}^{\infty}\|T\phi_{k}(D)\chi f\|_{L^{p}(\R^{n+1})}\lesssim \sum_{k=0}^{\infty}2^{-k\veps/2}\|f\|_{\HT^{s(p),p}_{FIO}(\Rn)}\\
&\lesssim \|f\|_{\HT^{s(p),p}_{FIO}(\Rn)}.\qedhere
\end{align*}
\end{proof}

\begin{remark}\label{rem:symboldependence}
We will use that, for a fixed phase function $\Phi$, the constant $C$ in \eqref{eq:decouplefull} only depends on $T$ through the size of the spatial support of the symbol $a$, and on finitely many of its $S^{m}(\R^{n+1}\times \Rn)$ seminorms. This follows either by keeping track of the constants in the proof of \eqref{eq:BeHiSo} and \eqref{eq:decouplebound}, or from an application of the closed graph theorem.
\end{remark}

\subsection{Local smoothing}\label{subsec:localsmooth}

In this subsection we will prove our main theorem on local smoothing for Fourier integral operators.

Let $(M,g)$ be a complete Riemannian manifold of dimension $n\geq2$ with bounded geometry, with smooth structure $\A$. Fix $K\subseteq \A$ and a uniformly locally finite family $(\psi_{\ka})_{\ka\in K}\subseteq C^{\infty}_{c}(M)$ as in Lemma \ref{lem:cover}.
Let $Q$ and $P$ be the associated maps from \eqref{eq:Q} and \eqref{eq:P}. Next, let $(N,g')$ be another complete Riemannian manifold with bounded geometry, this time of dimension $n+1$. Fix $L\subseteq \B$ and a uniformly locally finite family $(\theta_{\la})_{\la\in L}\subseteq C^{\infty}_{c}(N)$ as in Lemma \ref{lem:cover}, where $\B$ is the smooth structure on $N$. Let $Q'$ and $P'$ be as in \eqref{eq:Qprime} and \eqref{eq:Pprime}.

To prove our main result, it will be convenient to formulate a version of Theorem \ref{thm:FIOseq} adapted to the present setting.

\begin{lemma}\label{lem:FIOseq2}
Let $p\in(2,\infty)$, $r,s,m\in\R$, and let $T\in I^{m-1/4}(M,N;\Ca)$ be compactly supported, for $\Ca$ a canonical relation from $T^{*}M$ to $T^{*}N$ satisfying the cinematic curvature condition. Suppose that, for all $\la\in L$ and $\ka\in K$, there exists a $C_{\la\ka}\geq 0$ such that, for all $f\in\HT^{s,p}_{FIO}(\Rn)$, one has $\la_{*}(\theta_{\la}T\ka^{*}(\psit_{\ka}f))\in W^{s,p}(\R^{n+1})$ and
\begin{equation}\label{eq:FIOseq2}
\|\la_{*}(\theta_{\la}T\ka^{*}(\psit_{\ka}f))\|_{W^{s,p}(\R^{n+1})}\leq C_{\la\ka}\|f\|_{\HT^{r,p}_{FIO}(\Rn)}.
\end{equation}
Then $Q'TP:\ell^{p}(K;\HT^{r,p}_{FIO}(\Rn))\to\ell^{p}(L;W^{s,p}(\R^{n+1}))$ is bounded.
\end{lemma}
\begin{proof}
Just as was the case for Theorem \ref{thm:FIOseq}, the proof consists of an application of Lemma \ref{lem:matrix}, with $A_{\la\ka}f:=\la_{*}(\theta_{\la}T\ka^{*}(\psit_{\ka}f))$. Note that condition \eqref{it:matrix2} in Lemma \ref{lem:matrix} is automatically satisfied, since $T$ is compactly supported. In particular, one may choose $C_{\la\ka}=0$ for all but finitely many $\la$ and $\ka$, so condition \eqref{it:matrix1} is also satisfied.
\end{proof}

The main theorem of this section is as follows.

\begin{theorem}\label{thm:localsmoothFIO}
Let $T\in I^{m-1/4}(M,N;\Ca)$, for $m\in\R$ and $\Ca$ a canonical relation from $T^{*}M$ to $T^{*}N$ satisfying the cinematic curvature condition, and suppose that $T$ is compactly supported. Then
\begin{equation}\label{eq:localsmoothFIO}
T:\HT^{s+m,p}_{FIO}(M)\to W^{s-d(p)+s(p)-\veps,p}(N)
\end{equation}
is bounded for all $p\in(2,\infty)$, $s\in\R$ and $\veps>0$.
\end{theorem}
\begin{proof}
We first reduce to working in local coordinates. By Proposition \ref{prop:Banach}, Remark \ref{rem:propclassical} and Corollary \ref{cor:Sobolevman},
\[
Q':W^{s-d(p)+s(p)-\veps,p}(N)\to \ell^{p}(L;W^{s-d(p)+s(p)-\veps,p}(\R^{n+1}))
\]
is an isomorphic embedding. By combining this with Propositions \ref{prop:Banach} and \ref{prop:Qproperties}, it suffices to show that
\[
Q'TP:\ell^{p}(K;\HT^{s+m,p}_{FIO}(\Rn))\to \ell^{p}(L;W^{s-d(p)+s(p)-\veps,p}(\R^{n+1}))
\]
is bounded. To this end, by Lemma \ref{lem:FIOseq2}, we may fix $\la\in L$ and $\ka\in K$, and prove \eqref{eq:FIOseq2} with $r=s+m$ and with $s$ replaced by $s-d(p)+s(p)-\veps$. Moreover,  by the composition theorem for FIOs, using also that the cinematic curvature condition is invariant under changes of coordinates, $f\mapsto \la_{*}(\theta_{\la}T\ka^{*}(\psit_{\ka}f))$ is a compactly supported element of $I^{m-1/4}(\Rn,\R^{n+1},\Ca')$, associated with a canonical relation $\Ca'$ satisfying the cinematic curvature condition. Hence it suffices to prove \eqref{eq:localsmoothFIO} with $M=\Rn$ and $N=\R^{n+1}$.

In fact, we claim that it suffices to prove \eqref{eq:localsmoothFIO} for $s=d(p)-s(p)+\veps$. To prove this claim, set $r:=d(p)-s(p)+\veps$ and note that, for a general $s\in\R$, \eqref{eq:localsmoothFIO} is equivalent to
\begin{equation}\label{eq:reduction}
\lb D_{z}\rb^{s-r}T\lb D_{y}\rb^{-(s-r)}:\HT^{m+d(p)-s(p)+\veps,p}_{FIO}(\Rn)\to L^{p}(\R^{n+1}),
\end{equation}
where $D_{z}=-i\partial_{z}$ with respect to the $z$ variable in $N=\R^{n+1}$, and $D_{y}=-i\partial_{y}$ with respect to the $y$ variable in $M=\Rn$. Moreover, since $T$ is compactly supported, we can multiply $T$ by smooth cutoffs $\rho_{1}$ and $\rho_{2}$ and apply Lemma \ref{lem:propersupp} twice, to obtain
\[
\lb D_{z}\rb^{s-r}T\lb D_{y}\rb^{-(s-r)}=\lb D_{z}\rb^{s-r}\rho_{1}T\rho_{2}\lb D_{y}\rb^{-(s-r)}=S_{1}TS_{2}+R
\]
for compactly supported pseudodifferential operators $S_{1}$ and $S_{2}$ of orders $s-r$ and $-(s-r)$, respectively, and a smoothing operator $R$. As in \eqref{eq:smoothingR}, $R$ has the required mapping property. Moreover, the composition theorem for FIOs implies that $S_{1}TS_{2}$ is a compactly supported FIO of order $m$, associated with the same canonical relation as $T$. This proves the claim, and we may thus suppose that $s=d(p)-s(p)+\veps$ from here on.

Next, we reduce to working with operators in standard form. As in Lemma \ref{lem:FIOcurv}, write $T=\sum_{j=1}^{l}S_{j,1}T_{j}S_{j,2}+R$ for some $l\in\N$, FIOs $(T_{j})_{j=1}^{l}\subseteq I^{m-1/4}(\Rn,\R^{n+1};\Ca_{j})$ in standard form associated with canonical relations $(\Ca_{j})_{j=1}^{l}$, changes of coordinates $(S_{j,1})_{j=1}^{l}$ and $(S_{j,2})_{j=1}^{l}$, and a harmless smoothing operator $R$. Each change of coordinates $S_{j,1}$ is bounded on $L^{p}(\R^{n+1})$, and each $S_{j,2}$ is bounded on $\HT^{m+d(p)-s(p)+\veps,p}_{FIO}(\Rn)$, by Corollary \ref{cor:coordchange}. Hence it suffices to prove the required statement for $T$ in standard form, as in \eqref{eq:Tstandard}. Moreover, by the second part of Lemma \ref{lem:FIOcurv}, we may suppose that the phase function $\Phi$ of $T$ satisfies the conditions in Section \ref{subsec:decouple}. Finally, for simplicity of notation, we may assume without loss of generality that $m=0$, so that $a\in S^{0}(\R^{n+1}\times \Rn)$.

We are now almost in the setting of Corollary \ref{cor:decouple}, but we still have to reduce to the case where the symbol $a\in S^{0}(\R^{n+1}\times\Rn)$ of $T$ is of the form $a(z,\eta)=a_{1}(z)a_{2}(\eta)$ for $a_{1}\in C^{\infty}_{c}(\R^{n+1})$ and $a_{2}\in S^{0}(\Rn)$. This reduction to symbols of product type is contained in e.g.~\cite[Section VI.2]{Stein93} and \cite[pp.~424-425]{BeHiSo20}. Let $\chi\in C^{\infty}_{c}(\R^{n+1})$ be such that $\chi(z)a(z,\eta)=a(z,\eta)$ for all $z\in\R^{n+1}$ and $\eta\in\Rn$. Write
\[
Tf(z)=\frac{1}{(2\pi)^{n+1}}\int_{\R^{n+1}}e^{iz\cdot\zeta}\lb \zeta\rb^{-(n+2)}T_{\zeta}f(z)\ud \zeta
\]
for $f\in\Sw(\Rn)$ and $z\in\R^{n+1}$, where
\[
T_{\zeta}f(z):=\int_{\Rn}e^{i\Phi(z,\eta)}a_{\zeta}(z,\eta)\wh{f}(\eta)\ud\eta
\]
and $a_{\zeta}(z,\eta):=\chi(z)\lb \zeta\rb^{n+2}\F a(\cdot,\eta)(\zeta)$, for $\zeta\in\R^{n+1}$ and $\eta\in\Rn$. Here $\F a(\cdot,\eta)$ denotes the Fourier transform of $a(z,\eta)$ with respect to the $z$ variable.

 It then suffices to show that $\|T_{\zeta}f\|_{L^{p}(\R^{n+1})}\lesssim \|f\|_{\HT^{d(p)-s(p)+\veps,p}_{FIO}(\Rn)}$ for an implicit constant independent of $\zeta$ and $f$. To this end it is relevant to note that, as follows from integration by parts, $(a_{\zeta})_{\zeta\in\R^{n+1}}\subseteq S^{0}(\R^{n+1}\times \Rn)$ is a collection of symbols of product type, with uniform bounds on each of the $S^{0}(\R^{n+1}\times\Rn)$ seminorms and with spatial support contained in $\supp(\chi)$. Hence, to conclude the proof, one can simply apply Corollary \ref{cor:decouple} and Remark \ref{rem:symboldependence} to each $T_{\zeta}$.
\end{proof}

\section{Regularity of wave equations}\label{sec:wave}

In this section we apply our theorems for Fourier integral operators from Sections \ref{subsec:FIOembed} and \ref{subsec:localsmooth} to obtain new regularity results for linear wave equations, and new well-posedness results for nonlinear wave equations. In particular, we will prove Theorems \ref{thm:waveintro} and \ref{thm:nonlinear}.

Throughout, let $(M,g)$ be a compact Riemannian manifold of dimension $n\in\N$, with (non-positive) Laplace--Beltrami operator $\Delta_{g}$.

\subsection{Linear wave equations}\label{subsec:wavelinear}

In this subsection we consider the Cauchy problem \eqref{eq:Cauchyintro} on $M\times\R$:
\[
\begin{cases}(\partial_{t}^{2}-\Delta_{g})u(x,t)=0,
\\u(x,0)=u_{0}(x), \ \partial_{t}u(x,0)=u_{1}(x).
\end{cases}
\]
Recall that $\Delta_{g}$ acts on functions in local coordinates as in \eqref{eq:Laplace}. Then $\Delta_{g}$ also acts on half densities, using the identification $u\mapsto \rho_{g}^{-1/2}u$ between half densities and functions on $M$, where $\rho_{g}$ is the Riemannian density on $M$. Moreover, $\Delta_{g}$ is a non-positive operator on $L^{2}(M,\Omega_{1/2})$. For $t\in\R$, the operators $\cos(t\sqrt{-\Delta_{g}})$ and $\sin(t\sqrt{-\Delta_{g}})/\sqrt{-\Delta_{g}}$, initially defined through the spectral calculus for $\Delta_{g}$ on $L^{2}(M,\Omega_{1/2})$, extend to $\Da'(M,\Omega_{1/2})$, and the function $u:\R\to\Da'(M,\Omega_{1/2})$,
\begin{equation}\label{eq:solution}
u(t):=\cos(t\sqrt{-\Delta_{g}})u_{0}+\frac{\sin(t\sqrt{-\Delta_{g}})}{\sqrt{-\Delta_{g}}}u_{1}
\end{equation}
for $t\in\R$, solves \eqref{eq:Cauchyintro} for all $u_{0},u_{1}\in \Da'(M,\Omega_{1/2})$.

Here we will be interested in the regularity of solutions for initial data in the Hardy spaces for FIOs. To this end, we recall that (see e.g.~\cite[Sections 4.1 and 8.1]{Sogge17}) $\cos(t\sqrt{-\Delta_{g}})$ and $\sin(t\sqrt{-\Delta_{g}})/\sqrt{-\Delta_{g}}$ are compactly supported Fourier integral operators of order $0$ and $-1$, respectively, associated with a local canonical graph, for each $t\in\R$. Moreover, as space-time solution operators, they are elements of $I^{-1/4}(M,M\times\R,\Ca)$ and $I^{-5/4}(M,M\times\R,\Ca)$ for $\Ca$ satisfying the cinematic curvature condition. Hence Theorems \ref{thm:FIObdd} and \ref{thm:localsmoothFIO}, combined with the uniform boundedness principle, immediately yield regularity results for \eqref{eq:Cauchyintro}. However, in this case we have a distinguished time variable, and it is natural to consider a version of these theorems which takes this into account.

\begin{theorem}\label{thm:wave}
Let $p\in[1,\infty]$, $s\in\R$ and $t_{0}>0$. Then there exists a $C\geq0$ such that, for all $u_{0}\in \HpsM$ and $u_{1}\in \HT^{s-1,p}_{FIO}(M)$, the function $u$ in \eqref{eq:solution} satisfies
\begin{equation}\label{eq:wave1}
\|u(t)\|_{\HpsM}\leq C\big(\|u_{0}\|_{\HpsM}+\|u_{1}\|_{\HT^{s-1,p}_{FIO}(M)}\big)
\end{equation}
for all $t\in[-t_{0},t_{0}]$. Moreover, if $n\geq 2$ and $p\in(2,\infty)$, then for each $\veps>0$ there exists a $C_{\veps}\geq 0$ such that
\begin{equation}\label{eq:wave2}
\Big(\int_{-t_{0}}^{t_{0}}\|u(t)\|_{W^{s-d(p)+s(p)-\veps,p}(M)}^{p}\ud t\Big)^{1/p}\leq C_{\veps}\big(\|u_{0}\|_{\HpsM}+\|u_{1}\|_{\HT^{s-1,p}_{FIO}(M)}\big).
\end{equation}
\end{theorem}
\begin{proof}
For \eqref{eq:wave1}, Theorem \ref{thm:FIObdd} immediately yields norm bounds for each fixed $t\in\R$, but we also need the bounds to be locally uniform in $t$. Moreover, Theorem \ref{thm:localsmoothFIO} yields \eqref{eq:wave2} for $s-d(p)+s(p)-\veps\geq 0$, but for $s<d(p)-s(p)+\veps$ we want to take advantage of the specific structure of the underlying canonical relation, which yields integrability in $t$ and \eqref{eq:wave2}. To do so, we will slightly modify the proof of Theorem \ref{thm:localsmoothFIO}.

Let $\phi\in C^{\infty}_{c}(\R)$ be such that $\phi(t)=1$ for all $t\in[-t_{0},t_{0}]$. Then it suffices to prove \eqref{eq:wave1} and \eqref{eq:wave2} with $u$ replaced by $\phi u$. Now, $\phi u=T_{0}+T_{1}$ for compactly supported $T_{0}\in I^{-1/4}(M,M\times\R,\Ca)$ and $T_{1}\in I^{-5/4}(M,M\times\R,\Ca)$, where $\Ca$ satisfies the cinematic curvature condition.

Next, as in the proofs of Theorems \ref{thm:FIObdd} and \ref{thm:localsmoothFIO}, one can reduce to working in local coordinates, and (due to the product structure of $M\times \R$) one can choose geodesic normal coordinates independently for $M$ and $\R$. This expresses each $T_{k}$ as a finite sum of compactly supported elements of $I^{-k-1/4}(\Rn,\R^{n+1},\Ca')$, where
\[
\Ca'=\{(x,t,\xi,\tau,y,\eta)\mid (x,\xi)=\chi_{t}(y,\eta), \tau=p(x,\xi)\}
\]
for $\chi_{t}$ a canonical transformation on $T^{*}(\R^{n})\setminus o$ for each $t$ (see e.g.~\cite[pp. 129-130]{Sogge17}). It then suffices to prove mapping properties for these FIOs in local coordinates.

At this point one can reduce to the case where $s=d(p)-s(p)+\veps$, using a similar reduction as in the proof of Theorem \ref{thm:localsmoothFIO}. However, here one has to replace the derivative $D_{z}=-i\partial_{z}$ with respect to $z=(x,t)$ in \eqref{eq:reduction} by the derivative $D_{x}=-i\partial_{x}$ with respect to $x$. The resulting compactly supported FIOs are still associated with $\Ca'$, since $\lb D_{x}\rb^{s-r}$ and $\lb D_{y}\rb^{-(s-r)}$ are pseudodifferential operators.

Finally, an application of Lemma \ref{lem:FIOcurv} and Remark \ref{rem:timeparameter} expresses each of the resulting FIOs as a sum of operators in standard form. One can then apply Corollary \ref{cor:stanform} at each time $t$ to obtain \eqref{eq:wave1}, with the desired uniformity in $t$ coming from Remark \ref{rem:constantsFIO}. Moreover, Theorem \ref{thm:localsmoothFIO} with $s=d(p)-s(p)+\veps$ yields \eqref{eq:wave2}.
\end{proof}

\begin{remark}\label{rem:semigroup}
One can reformulate \eqref{eq:Cauchyintro} as an abstract Cauchy problem on $X:=\HT^{s,p}_{FIO}(M)\times \HT^{s-1,p}_{FIO}(M)$. That is, one wants to find a $U:\R\to X$ such that
\begin{equation}\label{eq:abstractCauchy}
\frac{\ud}{\ud t}U(t)=\left(\!\begin{array}{cc}
\!0&I\!\\
\!\Delta_{g}&0\!
\end{array}\!\right)U(t)\text{ and }U(0)=\left(\!\begin{array}{c}
\!u_{0}\!\\
\!u_{1}\!
\end{array}\!\right).
\end{equation}
Now, the function $u$ in \eqref{eq:solution} is strongly continuous on $\Da(M,\Omega_{1/2})$, as follows from its strong continuity on $L^{2}(M,\Omega_{1/2})$ combined with Sobolev embeddings. Then Theorem \ref{thm:wave}, combined with the density of $\Da(M,\Omega_{1/2})$ from Proposition \ref{prop:densityM}, implies that \eqref{eq:abstractCauchy} is well posed for $p<\infty$. In particular, $u$ is the unique strongly continuous solution to \eqref{eq:Cauchyintro} as in Theorem \ref{thm:wave}. For $p=\infty$, $u$ is weak-star continuous, cf.~Proposition \ref{prop:dualM}.
\end{remark}

\subsection{Nonlinear wave equations}\label{subsec:wavenonlinear}

We consider the cubic nonlinear wave equation from \eqref{eq:nonlinear} on $M\times \R$:
\begin{equation}\label{eq:nonlinearmain}
\begin{cases}(\partial_{t}^{2}-\Delta_{g})u(x,t)=\pm |u(x,t)|^{2}u(x,t),
\\u(x,0)=f_{1}(x), \ \partial_{t}u(x,0)=f_{2}(x).
\end{cases}
\end{equation}
We will assume throughout that $M$ has dimension $n=2$, although our techniques extend to higher dimensions and to other nonlinearities (see e.g.~Proposition \ref{prop:quintic} below).

Our notion of well-posedness of \eqref{eq:nonlinearmain} is taken from \cite{Bejenaru-Tao06}. Consider the abstract evolution equation
\begin{equation}
\label{eq:AbstractEvolution}
u = L(f_{1},f_{2}) + N(u,u,u),
\end{equation}
where $L: X\to S$ is a densely defined linear operator, and $N:S \times S \times S \to S$ is a densely defined operator which is either linear or antilinear in each of its variables.  We say that \eqref{eq:AbstractEvolution} is \emph{quantitatively well posed} (with initial data space $X$ and solution space $S$) if there exists a $C\geq0$ such that
\begin{align}
\label{eq:linearabstract}
\| L(f_{1},f_{2}) \|_{S} &\leq C \| (f_{1},f_{2}) \|_{X }, \\
\label{eq:nonlinearabstract}
\| N(u_1,u_2,u_3) \|_S &\leq C \prod_{j=1}^3 \| u_j \|_S,
\end{align}
for all $(f_{1},f_{2}) \in X$ and $u_1,u_{2},u_{3} \in S$. If \eqref{eq:AbstractEvolution} is quantitatively well posed, then a fixed-point argument yields $C_{0},\delta_{0}>0$ such that, for all
\[
(f_{1},f_{2})\in B_{X}(0,\delta_{0}):=\{(g_{1},g_{2})\in X\mid \|(g_{1},g_{2})\|_{X}<\delta_{0}\},
\]
there exists a unique solution $u[f_{1},f_{2}]\in B_{S}(0,C_{0}\delta_{0})$ to \eqref{eq:AbstractEvolution}. Moreover, the map $(f_{1},f_{2})\mapsto u[f_{1},f_{2}]$ is Lipschitz continuous from $B_{X}(0,\delta_{0})$ to $B_{S}(0,C_{0}\delta_{0})$, and one can expand $u[f_{1},f_{2}]$ in terms of its Picard iterates. For more on this see \cite[Theorem 3]{Bejenaru-Tao06}.

In our setting, $S=S_{t_{0}}$ is a space-time function space on the time interval $[0,t_{0}]$, for $t_{0}>0$. Then \eqref{eq:AbstractEvolution} arises by rewriting \eqref{eq:nonlinearmain} using Duhamel's formula, with
\[
\begin{split}
L(f_{1},f_{2})(t)&:=\cos(t\sqrt{-\Delta_{g}})f_{1}+\frac{\sin(t\sqrt{-\Delta_{g}})}{\sqrt{-\Delta_{g}}}f_{2}, \\
N(u_{1},u_{2},u_{3})(t)&:=\pm \int_{0}^{t}\frac{\sin((t-s)\sqrt{-\Delta_{g}})}{\sqrt{-\Delta_{g}}}u_{1}(s)\overline{u_{2}(s)}u_{3}(s)\ud s,
\end{split}
\]
for $t\in[0,t_{0}]$. In fact, we will improve \eqref{eq:nonlinearabstract} by showing that
\begin{equation}\label{eq:nonlinearabstract2}
\| N(u_1,u_2,u_3) \|_S \leq C t_{0}^{\gamma}\prod_{i=1}^3 \| u_i \|_S
\end{equation}
for all $t_{0}\in(0,1]$ and $u_{1},u_{2},u_{3}\in S=S_{t_{0}}$, where $C$ and $\gamma>0$ are independent of $t_{0}\in(0,1]$. Then the fixed-point argument also yields well-posedness for large data: for all $(f_{1},f_{2})\in X$ there exists a $t_{0}>0$ such that the power series for $u[f_{1},f_{2}]$ converges in $S_{t_{0}}$.

The main result of this section deals with two critical values of $p$, namely $p=6=2(n+1)/(n-1)$ and $p=4=2n/(n-1)$. It implies in particular Theorem \ref{thm:nonlinear}.

\begin{theorem}\label{thm:nonlinearmain}
Let $\veps,t_{0}>0$. Then \eqref{eq:nonlinearmain} is quantitatively well posed with initial data space
\[
X:=(\HT^{\veps,6}_{FIO}(M)+W^{1/2,2}(M))\times(\HT^{\veps-1,6}_{FIO}(M)+W^{-1/2,2}(M))
\]
and solution space
\[
S_{t_{0}}:=L^{4}\big([0,t_{0}];L^{6}(M)\big) \cap C\big([0,t_{0}];\HT^{\varepsilon,6}_{FIO}(M) + W^{1/2,2}(M)\big).
\]
Moreover, \eqref{eq:nonlinearmain} is quantitatively well posed with initial data space
\begin{equation}\label{eq:Ydata}
Y:=(\HT^{\veps,4}_{FIO}(M)+W^{3/8,2}(M))\times(\HT^{\veps-1,4}_{FIO}(M)+W^{-5/8,2}(M))
\end{equation}
and solution space
\[
T_{t_{0}}:=L^{24/7}\big([0,t_{0}];L^{4}(M)\big) \cap C\big([0,t_{0}];\HT^{\varepsilon,4}_{FIO}(M) + W^{3/8,2}(M)\big).
\]
\end{theorem}
\begin{proof}
The proof follows the same template as that of \cite[Theorem 1.2]{Rozendaal-Schippa23}. We will only consider $p=6$, with the other case being analogous. The spaces $W^{s,2}(M)$ arise due to the use of Strichartz estimates, to deal with the nonlinear term. 

We have to prove \eqref{eq:linearabstract} and \eqref{eq:nonlinearabstract}. In the process we will also prove \eqref{eq:nonlinearabstract2} for $t_{0}\in(0,1]$. In fact, we will obtain estimates for the solution spaces $L^{4}([0,t_{0}];L^{6}(M))$ and $C([0,t_{0}];\HT^{\varepsilon,6}_{FIO}(M) + W^{1/2,2}(M))$ separately.

Firstly, for a given $(f_{1},f_{2})\in X$, write $(f_{1},f_{2})=(f_{11},f_{21})+(f_{12},f_{22})$ with
\[
(f_{11},f_{21})\in \HT^{\veps,6}_{FIO}(M)\times \HT^{\veps-1,6}_{FIO}(M)\text{ and }(f_{12},f_{22})\in W^{1/2,2}(M)\times W^{-1/2,2}(M).
\]
Then H\"{o}lder's inequality, the second part of Theorem \ref{thm:wave}, and Strichartz estimates on compact manifolds (cf.~\cite{Kapitanski89,Kapitanski89b}) yield
\begin{align*}
&\|L(f_{1},f_{2})\|_{L^{4}([0,t_{0}];L^{6}(M))}\leq t_{0}^{1/12}\|L(f_{1},f_{2})\|_{L^{6}([0,t_{0}];L^{6}(M))}\\
&\lesssim \|L(f_{11},f_{21})\|_{L^{6}([0,t_{0}];L^{6}(M))}+\|L(f_{12},f_{22})\|_{L^{6}([0,t_{0}];L^{6}(M))}\\
&\lesssim \|(f_{11},f_{21})\|_{\HT^{\veps,6}_{FIO}(M)\times \HT^{\veps-1,6}_{FIO}(M)}+\|(f_{12},f_{22})\|_{W^{1/2,2}(M)\times W^{-1/2,2}(M)}. 
\end{align*}
Taking the infimum over all such decompositions, one obtains
\[
\|L(f_{1},f_{2})\|_{L^{4}([0,t_{0}];L^{6}(M))}\lesssim \|(f_{1},f_{2})\|_{X}.
\]
On the other hand, since $W^{1/2,2}(M)=\HT^{1/2,2}_{FIO}(M)$ (see e.g.~Theorem \ref{thm:SobolevM} and Corollary \ref{cor:Sobolevman}), the first part of Theorem \ref{thm:wave}, with $p=6$ and $p=2$, yields
\[
\|L(f_{1},f_{2})\|_{L^{\infty}([0,t_{0}];\HT^{\veps,6}_{FIO}(M)+W^{1/2,2}(M))}\lesssim \|(f_{1},f_{2})\|_{X}.
\]
Moreover, applying Remark \ref{rem:semigroup} with $p=6$ and $p=2$, we obtain
\[
L(f_{1},f_{2})\in C\big([0,t_{0}];\HT^{\varepsilon,6}_{FIO}(M) + W^{1/2,2}(M)\big),
\]
which completes the proof of \eqref{eq:linearabstract}.

Next, we prove \eqref{eq:nonlinearabstract} and \eqref{eq:nonlinearabstract2} for $u_{1},u_{2},u_{3}\in S$. Firstly, H\"{o}lder's inequality and Strichartz estimates yield
\begin{equation}\label{eq:slack}
\begin{aligned}
&\|N(u_{1},u_{2},u_{3})\|_{L^{4}([0,t_{0}];L^{6}(M))}\leq t_{0}^{1/12}\|N(u_{1},u_{2},u_{3})\|_{L^{6}([0,t_{0}];L^{6}(M))}\\
&\lesssim t_{0}^{1/12}\|u_{1}u_{2}u_{3}\|_{L^{1}([0,t_{0}];L^{2}(M))}\leq t_{0}^{1/3}\|u_{1}u_{2}u_{3}\|_{L^{4/3}([0,t_{0}];L^{2}(M))}\\
&\leq t_{0}^{1/3}\prod_{j=1}^{3}\|u_{j}\|_{L^{4}([0,t_{0}];L^{6}(M))},
\end{aligned}
\end{equation}
where the implicit constant is independent of $t_{0}$ if $t_{0}\in(0,1]$. On the other hand, Minkowski's inequality, the spectral theorem and H\"{o}lder's inequality yield
\begin{align*}
&\|N(u_{1},u_{2},u_{3})\|_{L^{\infty}([0,t_{0}];\HT^{\veps,6}_{FIO}(M)+W^{1/2,2}(M))}\leq \|N(u_{1},u_{2},u_{3})\|_{L^{\infty}([0,t_{0}];W^{1/2,2}(M))}\\
&\leq  \int_{0}^{t_{0}}\sup_{t\in[0,t_{0}]}\Big\|\frac{\sin((t-s)\sqrt{-\Delta_{g}})}{\sqrt{-\Delta_{g}}}u_{1}(s)\overline{u_{2}(s)}u_{3}(s)\Big\|_{W^{1/2,2}(M)}\ud s\\
&\lesssim \|u_{1}\overline{u_{2}}u_{3}\|_{L^{1}([0,t_{0}];L^{2}(M))}\leq t_{0}^{1/4}\|u_{1}u_{2}u_{3}\|_{L^{4/3}([0,t_{0}];L^{2}(M))}\\
&\leq t_{0}^{1/4}\prod_{j=1}^{3}\|u_{j}\|_{L^{4}([0,t_{0}];L^{6}(M))}
\end{align*}
for an implicit constant independent of $t_{0}$. Using also the strong continuity of $t\mapsto \sin(t\sqrt{-\Delta_{g}})$, cf.~Remark \ref{rem:semigroup}, this concludes the proof.
\end{proof}

\begin{remark}\label{rem:nonlinSob}
One can combine the same proof strategy with local smoothing estimates that use classical Sobolev spaces as spaces of initial data. For $p=6$, this would yield a weaker statement than that in Theorem \ref{thm:nonlinearmain}, given that Theorem \ref{thm:wave} improves upon the local smoothing conjecture for $p\geq 2(n+1)/(n-1)$. On the other hand, for $p=4$, the main result of \cite{GaLiMiXi23} allows one to show that \eqref{eq:nonlinearmain} is quantitatively well posed with initial data space
\[
Y:=(W^{\veps,4}(M)+W^{3/8,2}(M))\times(W^{\veps-1,4}(M)+W^{-5/8,2}(M))
\]
and solution space
\[
T_{t_{0}}:=L^{24/7}\big([0,t_{0}];L^{4}(M)\big).
\]
Due to the sharpness of the embeddings in \eqref{eq:Sobolevintro}, this initial data space neither contains the one in \eqref{eq:Ydata}, nor vice versa. Note, moreover, that in this case there is no pointwise regularity statement as in Theorem \ref{thm:nonlinearmain}.
\end{remark}

In \eqref{eq:slack} we did not make full use of the regularizing effect of the Duhamel term. The additional regularization can be used to solve the quintic nonlinear wave equation
\begin{equation}
\label{eq:quintic}
\begin{cases}(\partial_{t}^{2}-\Delta_{g})u(x,t)=\pm |u(x,t)|^{4}u(x,t),
\\u(x,0)=f_{1}(x), \ \partial_{t}u(x,0)=f_{2}(x),
\end{cases}
\end{equation}
for small initial data.

\begin{proposition}\label{prop:quintic}
Let $\veps,t_{0}>0$.   Then \eqref{eq:quintic} is quantitatively well posed with initial data space
\[
X:=(\HT^{\veps,6}_{FIO}(M)+W^{1/2,2}(M))\times(\HT^{\veps-1,6}_{FIO}(M)+ W^{-1/2,2}(M))
\]
and solution space
\[
S_{t_{0}} = L^{6}\big([0,t_{0}]; L^6(M)\big) \cap C\big([0,t_{0}]; \HT^{\varepsilon,6}_{FIO}(M) + W^{1/2,2}(M)\big)
\]
\end{proposition}

The notion of quantitative well-posedness of \eqref{eq:quintic} is the natural analog of that of \eqref{eq:nonlinearmain}.
\begin{proof}
The linear estimate \eqref{eq:linearabstract} again follows from Strichartz estimates and from Theorem \ref{thm:wave}. For the natural analog of \eqref{eq:nonlinearabstract} we also use Strichartz estimates:
\begin{align*}
& \Big\| \int_0^t \frac{\sin((t-s) \sqrt{- \Delta})}{\sqrt{-\Delta}} \prod_{j=1}^5 u_j(s) ds \Big\|_{L^6([0,T],L^6(M))\cap L^{\infty}([0,T];W^{1/2,2}(M))}\\
&\lesssim \Big\| \prod_{j=1}^5 u_j \Big\|_{L_t^{6/5}([0,T], L^{6/5}(M))} \lesssim \prod_{j=1}^5 \| u_j \|_{L_t^6([0,T], L^6(M))}.\qedhere
\end{align*}
\end{proof}

\begin{remark}\label{rem:smalldata}
Note that we did not show that \eqref{eq:nonlinearabstract2} holds, so we do not obtain well-posedness for large initial data.
\end{remark}

Finally, global existence holds for the defocusing cubic nonlinear wave equation:
\begin{equation}\label{eq:defocus}
\begin{cases}(\partial_{t}^{2}-\Delta_{g})u(x,t)=-|u(x,t)|^{2}u(x,t),
\\u(x,0)=f_{1}(x), \ \partial_{t}u(x,0)=f_{2}(x).
\end{cases}
\end{equation}
The following proposition is a version on compact surfaces of \cite[Theorem 5.8]{Rozendaal-Schippa23}. The proof, which involves a blow-up alternative, Sobolev embeddings and an application of Gr\"{o}nwall's inequality, is exactly as in the Euclidean case (see also \cite{DoSoSp21,Schippa22}).

\begin{proposition}\label{prop:global}
Let $s>1/2$ and
\[
(f_{1},f_{2})\in (\HT^{s,6}_{FIO}(M)+W^{1/2,2}(M))\times(\HT^{s-1,6}_{FIO}(M)+W^{-1/2,2}(M)).
\]
Let $u$ be the solution to \eqref{eq:defocus} given by Theorem \ref{thm:nonlinearmain}. Then $u\in L^{4}([0,t_{0}];L^{6}(M))$ for all $t_{0}>0$.
\end{proposition}

\section{Sharpness of the results}\label{sec:sharpness}

In this section we show that Theorems \ref{thm:waveintro} and \ref{thm:wave} are sharp, and Theorems \ref{thm:localsmoothFIOintro} and \ref{thm:localsmoothFIO} are sharp under mild assumptions. To this end, we first prove that the index $d(p)-s(p)$ in Corollary \ref{cor:decouple} cannot be improved, from which we then derive the same conclusion in the setting of Theorem \ref{thm:localsmoothFIO}.

Throughout, fix a dimension $n\geq 2$.

\subsection{Sharpness for operators in standard form}\label{subsec:sharpstand}

Throughout this subsection, fix $m\in\R$, $a\in S^{m}(\R^{n+1}\times\Rn)$ and a real-valued $\Phi\in C^{\infty}(\R^{n+1}\times (\Rn\setminus\{0\}))$, positively homogeneous of degree one in the fiber variable, such that $\rank\,\partial_{z\eta}^{2}\Phi(z_{0},\eta_{0})=n$ for all $(z_{0},\eta_{0})\in\supp(a)$ with $\eta_{0}\neq 0$. Set
\[
Tf(z):=\int_{\Rn}e^{i\Phi(z,\eta)}a(z,\eta)\wh{f}(\eta)\ud\eta
\]
for $f\in\Sw(\Rn)$ and $z\in\R^{n+1}$. Using terminology from Section \ref{subsec:FIOs}, we say that $T$ is a Fourier integral operator in standard form.

To prove sharpness, we will proceed in a similar way as in \cite[Section 5]{Rozendaal22b}, which dealt with the case where $Tf(x,t)=e^{it\phi(D)}f(x)$ for $\phi\in C^{\infty}(\Rn\setminus\{0\})$ positively homogeneous of degree one. A key step in that proof was to approximate $T$ by translations of the form $x\mapsto x+t\partial_{\eta}\phi(\nu)$, which arise through propagation of singularities from the canonical relation
\[
\{(x,t,\eta,\phi(\eta),x+t\partial_{\eta}\phi(\eta),\eta)\}.
\]
In the present variable-coefficient case, the relevant canonical relation is
\[
\{(z,\partial_{z}\Phi(z,\eta),\partial_{\eta}\Phi(z,\eta),\eta)\},
\]
and we will adjust the flow. However, we do not need to take the momentum variable into account. This leads to the following version of \cite[Lemma 5.1]{Rozendaal22b}.

\begin{lemma}\label{lem:flow}
Suppose that $m=0$ and that there exists a non-empty conic $V\subseteq\Rn\setminus\{0\}$ such that $a(0,\eta)=1$ and $\Phi(0,\eta)=0$
for all $\eta\in V$ with $|\eta|\geq 1/2$. Suppose also that $|(1-r)\nu+r\w|\geq 1/2$ for all $\nu,\w\in V\cap S^{n-1}$ and $r\in[0,1]$. Then there exists a $C\geq0$ such that
\[
|Th(z)-(2\pi)^{n}h(\partial_{\eta}\Phi(z,\nu))|\leq C\|\wh{h}\|_{L^{1}(\Rn)}|z|(1+\gamma(h,\nu))
\]
for all $h\in \Sw(\Rn)$ with $\supp(\wh{h})\subseteq V\cap \{\eta\in\Rn\mid |\eta|\geq 1/2\}$ compact, all $\nu\in S^{n-1}\cap V$ and all $z\in\R^{n+1}$ with $|z|\leq 1$. Here
\[
\gamma(h,\nu):=\sup\{|\hat{\eta}-\nu||\eta|\mid \eta\in\supp(\wh{h})\}.
\]
\end{lemma}
\begin{proof}
Fix $\nu\in S^{n-1}\cap V$ and $z\in\R^{n+1}$, and write
\begin{align*}
&Th(z)-(2\pi)^{n}h(\partial_{\eta}\Phi(z,\nu))=\int_{\Rn}\big(e^{i\Phi(z,\eta)}a(z,\eta)-e^{i\partial_{\eta}\Phi(z,\nu)\cdot\eta}\big)\wh{h}(\eta)\ud\eta\\
&=\int_{\Rn}e^{i\Phi(z,\eta)}(a(z,\eta)-1)\wh{h}(\eta)\ud\eta+\int_{\Rn}\big(e^{i\Phi(z,\eta)}-e^{i\partial_{\eta}\Phi(z,\nu)\cdot\eta}\big)\wh{h}(\eta)\ud\eta.
\end{align*}
We will bound each of the terms on the second line separately. The first term is easy:
\[
\Big|\int_{\Rn}e^{i\Phi(z,\eta)}(a(z,\eta)-1)\wh{h}(\eta)\ud\eta\Big|\lesssim \|\wh{h}\|_{L^{1}(\Rn)}|z|,
\]
where we used Taylor approximation and the assumptions that $m=0$ and that $a(0,\eta)=1$ for $\eta\in V$.

To deal with the second term, we first use Taylor approximation for the exponential:
\begin{align*}
\Big|\int_{\Rn}\big(e^{i\Phi(z,\eta)}-e^{i\partial_{\eta}\Phi(z,\nu)\cdot\eta}\big)\wh{h}(\eta)\ud\eta\Big|&\leq \int_{\Rn}\big|e^{i(\Phi(z,\eta)-\partial_{\eta}\Phi(z,\nu)\cdot\eta)}-1\big|\,|\wh{h}(\eta)|\ud\eta\\
&\leq \|\wh{h}\|_{L^{1}(\Rn)}\delta(h,\nu,z),
\end{align*}
where
\[
\delta(h,\nu,z):=\sup\{|(\partial_{\eta}\Phi(z,\eta)-\partial_{\eta}\Phi(z,\nu))\cdot\eta|\mid \eta\in\supp(\wh{h})\}.
\]
It thus suffices to show that $\delta(h,\nu,z)\lesssim |z|\gamma(h,\nu)$. To do so, for $\eta\in V$, we can use the homogeneity of $\Phi$ and another Taylor expansion to write
\begin{align*}
(\partial_{\eta}\Phi(z,\eta)-\partial_{\eta}\Phi(z,\nu))\cdot\eta&=(\partial_{\eta}\Phi(z,\hat{\eta})-\partial_{\eta}\Phi(z,\nu))\cdot\eta\\
&=\Big(\int_{0}^{1}\partial_{\eta\eta}^{2}\Phi(z,(1-r)\nu+r\hat{\eta})(\hat{\eta}-\nu)\ud r\Big)\cdot \eta.
\end{align*}
Finally, note that $\partial_{\eta\eta}^{2}\Phi(0,(1-r)\nu+r\hat{\eta})=0$ for all $r\in[0,1]$, since $\Phi(0,\eta')= 0$ for $\eta'\in V$ with $|\eta'|\geq 1/2$, and because $|(1-r)\nu+r\hat{\eta}|\geq 1/2$, by assumption. Hence, for each $r$ we can use another Taylor approximation, this time with respect to $z$, to obtain
\[
|\Phi(z,\eta)-\partial_{\eta}\Phi(z,\nu)\cdot\eta|\lesssim |z|\,|\hat{\eta}-\nu|\,|\eta|,
\]
as required.
\end{proof}

To prove sharpness for $2<p<2(n+1)/(n-1)$, we will use a version of \cite[Lemma 5.2]{Rozendaal22b} which relates the Hardy spaces for FIOs to square functions as in \cite{GuWaZh20,GaLiMiXi23}. The proof in \cite{Rozendaal22b}, which relies on Khintchine's inequality, extends directly to this setting. We use notation as in the decoupling norm in \eqref{eq:decouplenorm}.

\begin{lemma}\label{lem:squarefun}
Let $p\in[1,\infty)$ and $s\in\R$ be such that $T:\Hps\to L^{p}(\R^{n+1})$ is bounded. Then there exists a $C\geq0$ such that
\[
\Big(\int_{\R^{n+1}}\Big(\sum_{\nu\in\Theta_{k}}|T\chi_{\nu}(D)f(z)|^{2}\Big)^{p/2}\ud z\Big)^{1/p}\leq C\|f\|_{\Hps}
\]
for all $f\in\Hps$ such that $\supp(\wh{f}\,)\subseteq \{\xi\in\Rn\mid 2^{k-1}\leq |\xi|\leq 2^{k+1}\}$ for some $k\in\Z_{+}$.
\end{lemma}

We are now ready to prove that the index $d(p)-s(p)$ in Corollary \ref{cor:decouple} is sharp. Note, however, that our assumptions on $T$ are much weaker than those in Section \ref{subsec:decouple}. In particular, we do not assume that the canonical relation $\Ca$ associated with $T$ satisfies the cinematic curvature condition, nor even that the projection condition holds which guarantees sharpness of the fixed-time $L^{p}$ estimates (see \cite{SeSoSt91}).

\begin{theorem}\label{thm:sharpstand}
Let $p\in(2,\infty)$ and $s\in\R$ be such that $T:\HT^{s+m}_{FIO}(\Rn)\to L^{p}(\R^{n+1})$ is bounded. Suppose that there exist a $z_{0}\in\R^{n+1}$, a non-empty open conic set $V\subseteq \Rn\setminus\{0\}$ and $c,C>0$ such that $|a(z_{0},\eta)|\geq c|\eta|^{m}$ for all $\eta\in V$ with $|\eta|\geq C$. Then $s\geq d(p)-s(p)$.
\end{theorem}
\begin{proof}
The approach is similar to that in \cite[Theorem 5.3]{Rozendaal22b}, as are the examples that show sharpness. However, we need additional work to reduce to a setting where we can apply Lemma \ref{lem:flow}.

By translating and precomposing with $\lb D\rb^{-m}$, we may assume that $z_{0}=0$ and $m=0$, and by adding a smoothing operator which modifies the low frequencies near $z_{0}$, we may assume that $C=1/2$. We may also suppose, without loss of generality, that $|(1-r)\nu+r\w|\geq 1/2$ for all $\nu,\w\in V\cap S^{n-1}$ and $r\in[0,1]$. Finally, we will assume for the moment that $a(0,\eta)=1$ and $\Phi(0,\eta)=0$ for all $\eta\in V$ with $|\eta|\geq 1/2$. This assumption will be removed at the end of the proof.

Now, we want to show that $s\geq d(p)-s(p)=\max(s(p)-1/p,0)$. Fix a standard Littlewood--Paley decomposition $(\phi_{k})_{k=0}^{\infty}\subseteq C^{\infty}_{c}(\Rn)$, as in the proof of Corollary \ref{cor:decouple}, and let $k\in\Z_{+}$. Our construction will only rely on estimates for large $k$. Given that $V$ is non-empty and open, we may thus choose $k$ large enough such that the collection $V_{k}$ of $\nu\in \Theta_{k}$ with $\supp(\chi_{\nu})\subseteq V$ is non-empty. Note that then $V_{k}$ has approximately $2^{k(n-1)/2}$ elements, because $V$ has a fixed nonzero aperture.

We first prove that $s\geq s(p)-1/p$.
In this case we will apply Lemma \ref{lem:flow} to functions whose Fourier support fills up a dyadic-parabolic piece.

We will construct a collection $(f_{\nu})_{\nu\in V_{k}}\subseteq \Sw(\Rn)$ with the following properties. For each $\nu\in V_{k}$ one has $\phi_{k}(D)f_{\nu}=\chi_{\nu}(D)f_{\nu}=f_{\nu}$, and $\Real(f_{\nu}(x))\gtrsim 1$ for all $x\in\Rn$ with $|x|\leq 2^{-k}$. Moreover, $\|\F f_{\nu}\|_{L^{1}(\Rn)}\eqsim 1$ and $\|f_{\nu}\|_{L^{p}(\Rn)}\eqsim 2^{-k\frac{n+1}{2p}}$. Here all the implicit constants are independent of $\nu$ and $k$. To construct such a collection, one may fix a $\psi\in \Sw(\Rn)$ such that $\psi(x)\geq1$ whenever $|x|\leq 1$, and $\wh{\psi}(\xi)=0$ if $|\xi|>c'$, for some small $c'>0$. Then set
\[
f_{\nu}(x):=e^{i2^{k}\nu\cdot x}\psi(2^{k}(\nu\cdot x)\nu+2^{k/2}\Pi_{\nu}^{\perp}x),
\]
where $\Pi^{\perp}_{\nu}$ is the orthogonal projection onto the complement of the span of $\nu$.

Now set $f:=\sum_{\nu\in V_{k}}f_{\nu}$. Then $\supp(\wh{f}\,)\subseteq V$, $f\in\Hps$ and
\begin{equation}\label{eq:fnorm}
\|f\|_{\Hps}\eqsim 2^{k(s+\frac{n-1}{2}(\frac{1}{2}-\frac{1}{p}))}\Big(\sum_{\nu\in V_{k}}\|f_{\nu}\|_{L^{p}(\Rn)}^{p}\Big)^{1/p}\eqsim 2^{k(s+\frac{n-1}{2}(\frac{1}{2}-\frac{1}{p})-\frac{1}{p})},
\end{equation}
by \cite[Proposition 4.1]{Rozendaal22b} (see also \cite[Equation (5.3)]{Rozendaal22b}).

Now, $\Phi$ and all of its derivatives are bounded on compact subsets of $\Rn\times(\Rn\setminus\{0\})$. Hence $z\mapsto \partial_{\eta}\Phi(z,\nu)$ is a locally Lipschitz flow for each $\nu\in V_{k}$, with Lipschitz constants independent of $\nu$. Moreover, by assumption, $\partial_{\eta}\Phi(0,\nu)=0$. It thus follows that $\Real(f_{\nu}(\partial_{\eta}\Phi(z,\nu)))\gtrsim 1$ whenever $|z|\lesssim 2^{-k}$ for some implicit constant independent of $k$ and $\nu$. On the other hand, because $\supp(f_{\nu})\subseteq \supp(\phi_{k}\chi_{\nu})$, one has $|\hat{\eta}-\nu|\,|\eta|\lesssim 2^{k/2}$ for all $\eta\in\supp(f_{\nu})$. Hence Lemma \ref{lem:flow} yields
\[
|Tf_{\nu}(z)-(2\pi)^{n}f_{\nu}(\partial_{\eta}\Phi(z,\nu))|\lesssim |z|2^{k/2}
\]
whenever $|z|\lesssim 2^{-k}$. By combining all this, we find
\[
\Real(Tf_{\nu}(z))\geq (2\pi)^{n}\Real(f_{\nu}(\partial_{\eta}\Phi(z,\nu)))-|Tf_{\nu}(z)-(2\pi)^{n}f_{\nu}(\partial_{\eta}\Phi(z,\nu))|\gtrsim 1
\]
if $|z|\leq c_{1}2^{-k}$ for some small $c_{1}>0$ independent of $k$ and $\nu$. Now we can use the assumption on $T$ and \eqref{eq:fnorm}:
\begin{align*}
2^{-k\frac{n+1}{p}}2^{k\frac{n-1}{2}}&\lesssim \Big(\int_{\R^{n+1}}\Big|\sum_{\nu\in V_{k}}Tf_{\nu}(z)\Big|^{p}\ud z\Big)^{1/p}=\|Tf\|_{L^{p}(\R^{n+1})}\lesssim \|f\|_{\HT^{s,p}_{FIO}(\Rn)}\\
&\eqsim 2^{k(s+\frac{n-1}{2}(\frac{1}{2}-\frac{1}{p})-\frac{1}{p})}.
\end{align*}
This implies that $s\geq s(p)-1/p$.

Next, we will show that $s\geq0$, which is the required statement for $2<p<2(n+1)/(n-1)$. Here we will apply Lemma \ref{lem:flow} to functions with frequency support of unit size in a dyadic-parabolic piece, and we will also use Lemma \ref{lem:squarefun}.

Let $\psi\in\Sw(\Rn)$ be as before, and set $g_{\nu}(x):=e^{i2^{k}\nu\cdot x}\psi(x)$ for $\nu\in V_{k}$ and $x\in\Rn$. Since $\wh{\psi}(\xi)=0$ if $|\xi|>c'$, for some small $c'>0$, we may again suppose that $\phi_{k}(D)g_{\nu}=\chi_{\nu}(D)g_{\nu}=g_{\nu}$. Also, $\|g_{\nu}\|_{L^{p}(\Rn)}\eqsim 1\eqsim \|\F(g_{\nu})\|_{L^{1}(\Rn)}$. Set $g:=\sum_{\nu\in V_{k}}g_{\nu}$. Then $\supp(\wh{g})\subseteq V$, $g\in\Hps$ and
\begin{equation}\label{eq:gnorm}
\|g\|_{\Hps}\eqsim 2^{k(s+\frac{n-1}{2}(\frac{1}{2}-\frac{1}{p}))}\Big(\sum_{\nu\in V_{k}}\|g_{\nu}\|_{L^{p}(\Rn)}^{p}\Big)^{1/p}\eqsim 2^{k(s+\frac{n-1}{4})},
\end{equation}
by \cite[Proposition 4.1]{Rozendaal22b} (see also \cite[Equation (5.5)]{Rozendaal22b}). 

Let $\nu\in V_{k}$. Since $\psi(x)\geq 1$ for $|x|\leq 1$, we find in a similar way as before that $|g_{\nu}(\partial_{\eta}\Phi(z,\nu))|\gtrsim 1$ whenever $|z|\lesssim 1$, for small implicit constants independent of $k$ and $\nu$. Moreover, since $\supp(g_{\nu})\subseteq \supp(\phi_{k}\chi_{\nu})$ is of unit size, Lemma \ref{lem:flow} yields
\[
|Tg_{\nu}(z)-(2\pi)^{n}g_{\nu}(\partial_{\eta}\Phi(z,\nu))|\lesssim |z|
\]
if $|z|\lesssim 1$. Combined, this implies that
\[
|Tg_{\nu}(z)|\geq (2\pi)^{n}|g_{\nu}(\partial_{\eta}\Phi(z,\nu))|-|Tg_{\nu}(z)-(2\pi)^{n}g_{\nu}(\partial_{\eta}\Phi(z,\nu))|\gtrsim 1
\]
whenever $|z|\leq c_{2}$ for some small $c_{2}>0$ independent of $k$ and $\nu$. Now we can use the assumption on $T$, Lemma \ref{lem:squarefun} and \eqref{eq:gnorm}:
\begin{align*}
2^{k\frac{n-1}{4}}&\lesssim \Big(\int_{\R^{n+1}}\Big(\sum_{\nu\in V_{k}}|Tg_{\nu}(z)|^{2}\Big)^{p/2}\ud z\Big)^{1/p}\\
&=\Big(\int_{\R^{n+1}}\Big(\sum_{\nu\in V_{k}}|T\chi_{\nu}(D)g(z)|^{2}\Big)^{p/2}\ud z\Big)^{1/p}\lesssim \|g\|_{\Hps}\eqsim 2^{k(s+\frac{n-1}{4})},
\end{align*}
which shows that $s\geq0$.

We have now proved that $s\geq d(p)-s(p)$, and it only remains to remove the assumption that $a(0,\eta)=1$ and $\Phi(0,\eta)=0$ for all $\eta\in V$ with $|\eta|\geq 1/2$. To do so, and in particular for the second part of this assumption, we will use the invariance of $\Hps$ under FIOs.

First note that, so far, we have not used any information about $V$ and $c$, beyond what is listed in the statement of the theorem, as well as the assumption that $|(1-r)\nu+r\w|\geq 1/2$ for all $\nu,\w\in V\cap S^{n-1}$ and $r\in[0,1]$, which can be made without loss of generality by choosing $V$ small enough. Hence, by choosing slightly smaller $V$ and $c$, we may suppose that the closure of $V$ in $\R^{n}\setminus\{0\}$ is a strict subset of a closed cone $V_{0}\subseteq\Rn\setminus\{0\}$ such that $|a(z_{0},\eta)|\geq c$ for all $\eta\in V_{0}$ with $|\eta|\geq c'$, for some $c'<1/2$.

Let $\rho\in C^{\infty}(\R^{n})$ be such that $\supp(\rho)\subseteq V_{0}\cap\{\eta\in\Rn\mid |\eta|\geq c'\}$, and such that $\rho(\eta)=1$ for all $\eta\in V\cap \{\eta\in\Rn\mid |\eta|\geq1/2\}$. Set
\[
\theta(\eta):=e^{-i\Phi(0,\eta)}\frac{\rho(\eta)}{a(0,\eta)}
\]
for $\eta\in\Rn$. Then $\theta(D)$ is an FIO in standard form, with phase function $x\cdot\eta-\Phi(0,\eta)$ and symbol $\rho(\eta)/a(0,\eta)$. In fact, $\theta(D)\in I^{0}(\Rn,\Rn;\Ca')$ for
\[
\Ca':=\{(x,\eta,x-\partial_{\eta}\Phi(0,\eta),\eta)\mid (x,\eta)\in\Rn\times(\Rn\setminus\{0\})\},
\]
which is the graph of $(y,\eta)\mapsto (y+\partial_{\eta}\Phi(0,\eta),\eta)$ on $T^{*}\Rn\setminus o$. It thus follows from \cite[Theorem 6.10]{HaPoRo20} (see also \cite[Remark 3.7]{Rozendaal21}) that $\theta(D):\HT^{s,p}_{FIO}(\Rn)\to\HT^{s,p}_{FIO}(\Rn)$ is bounded.

By assumption on $T$, this in turn implies that $T\theta(D):\Hps\to L^{p}(\R^{n+1})$ is bounded. Moreover, $T\theta(D)$ is an FIO in standard form, with phase function given by $\Phi(z,\eta)-\Phi(0,\eta)$ for $\eta\in V\cap \{\eta\in\Rn\mid|\eta|\geq1/2\}$, and symbol given by $a(z,\eta)/a(0,\eta)$ for such $\eta$. Now the first part of the proof shows that $s\geq d(p)-s(p)$, thereby concluding the proof for general $T$ as in the statement of the theorem.
\end{proof}

\begin{remark}\label{rem:sharpexamples}
Note that the sharpness examples are universal, in the sense that
they can be used for any FIO satisfying the conditions of Theorem \ref{thm:sharpstand}. Also note that, when combined with Proposition \ref{prop:conversedecouple}, Theorem \ref{thm:sharpstand} shows that the exponent $d(p)$ in Theorem \ref{thm:BeHiSo} is sharp whenever $T$ is non-characteristic somewhere.
\end{remark}

\subsection{Sharpness on manifolds}\label{subsec:sharpman}

Throughout, let $(M,g)$ and $(N,g')$ be complete Riemannian manifolds of dimensions $n$ and $n+1$, respectively, with bounded geometry.

We will prove that the index $d(p)-s(p)$ in Theorem \ref{thm:localsmoothFIO} is sharp. However, the assumptions under which we show sharpness are, for the most part, much weaker than those in Theorem \ref{thm:localsmoothFIO}, and in particular we do not assume that $\Ca$ satisfies the cinematic curvature condition.

\begin{theorem}\label{thm:sharpgen}
Let $T\in I^{m-1/4}(M,N;\Ca)$, for $m\in\R$ and $\Ca\subseteq (T^{*}N\setminus o)\times(T^{*}M\setminus o)$ a homogeneous canonical relation such that the projections $\Pi_{N}:\Ca\to N$ and $\Pi_{M}:\Ca\to M$ are submersions. Let $p\in(2,\infty)$ and $s,r\in\R$. Suppose that
\[
T:\HT^{s+m,p}_{FIO}(M)\to W^{r,p}(N)
\]
is bounded, and that $T$ is non-characteristic at some point $c\in \Ca$. Then $r\leq s-d(p)+s(p)$.
\end{theorem}
Note that, if $\Ca$ satisfies the cinematic curvature condition, then it also satisfies the projection conditions in the theorem, given that the natural projection $\Pi_{M}:T^{*}M\setminus o\to M$ is a submersion. Moreover, the space-time solution operator to the Cauchy problem \eqref{eq:Cauchyintro} satisfies the conditions of the theorem, given that the wave propagators are invertible and therefore elliptic. Hence the index $d(p)-s(p)$ in Theorems \ref{thm:waveintro} and \ref{thm:wave} is sharp on every compact manifold.
\begin{proof}
Firstly, we may work in a sufficiently small conic neighborhood of $c$, to be chosen later. More precisely, let $S_{1}$ and $S_{2}$ be compactly supported pseudodifferential operators of order zero on $N$ and $M$, whose kernels are smooth away from small conic neighborhoods of $(z_{0},\zeta_{0},z_{0},\zeta_{0})$ and $(y_{0},\eta_{0},y_{0},\eta_{0})$, and whose principal symbols are equal to one near $(z_{0},\zeta_{0})$ and $(y_{0},\eta_{0})$, respectively. Then, by the composition theorem for FIOs and because the canonical relation of a pseudodifferential operator is the diagonal, $S_{1}TS_{2}\in I^{m-1/4}(M,N;\Ca')$ is non-characteristic at $c$, and $\Ca'$ is a small neighborhood of $c$ in $\Ca$. Moreover, $S_{2}$ is bounded on $\HT^{s+m,p}_{FIO}(M)$, by Theorem \ref{thm:FIObdd}, and $S_{1}:W^{r,p}(N)\to W^{r,p}(N)$ is bounded as well, by Corollary \ref{cor:Sobolevman} and because pseudodifferential operators of order zero are bounded on $W^{r,p}(\R^{n+1})$. Hence $S_{1}TS_{2}:\HT^{s+m,p}_{FIO}(M)\to W^{r,p}(N)$ is bounded, and we may replace $T$ by $S_{1}TS_{2}$ in the remainder. For simplicity of notation, we will continue working with $T$ and simply assume that $T\in I^{m-1/4}(M,N;\Ca)$ for $\Ca$ a small neighborhood of $c$.

Next, we reduce to working in local coordinates. Let uniformly locally finite families $K\subseteq \A$, $L\subseteq\B$, $(\psi_{\ka})_{\ka\in K}\subseteq C^{\infty}_{c}(M)$ and $(\theta_{\la})_{\la\in L}\subseteq C^{\infty}_{c}(N)$ as in Lemma \ref{lem:cover} be given, where $\A$ and $\B$ are the smooth structures on $M$ and $N$, respectively. Write $c=(z_{0},\zeta_{0},y_{0},\eta_{0})$, and fix $\ka\in K$ and $\la\in L$ such that $\psi_{\ka}(y_{0})\neq 0$ and $\theta_{\la}(z_{0})\neq 0$.
Set $\tilde{T}f:=\la_{*}(\theta_{\la}T\ka^{*}(\tilde{\psi}_{\ka} f))$ for $f\in\Da(\Rn)$. Then, by the composition theorem for FIOs, $\tilde{T}\in I^{m}(\Rn,\R^{n+1};\tilde{\Ca})$ for
\[
\tilde{\Ca}:=\{(\la(z),(\partial \la(z))_{*}\zeta,\ka(y),(\partial\ka(y))_{*}\eta)\mid (z,\zeta,y,\eta)\in\Ca, z\in U_{\la}, y\in U_{\ka}\}.
\]
Note that the natural projections $\Pi_{\R^{n+1}}:\tilde{\Ca}\to \R^{n+1}$ and $\Pi_{\Rn}:\tilde{\Ca}\to\Rn$ are submersions, by the assumption on $\Ca$. Moreover, the principal symbol of $\tilde{T}$ is just the principal symbol of $T$, expressed in local coordinates near $z_{0}$ and $y_{0}$. In particular,  $\tilde{T}$ is non-characteristic at
\[
\tilde{c}:=(\la(z_{0}),(\partial \la(z_{0}))_{*}\zeta_{0},\ka(y_{0}),(\partial\ka(y_{0}))_{*}\eta_{0})\in\tilde{\Ca},
\]
given that $\theta_{\la}(z_{0})\neq0\neq\psi_{\ka}(y_{0})$. Also, by Definition \ref{def:Hpman} and Corollary \ref{cor:Sobolevman}, one has $\la_{*}(\theta_{\la}T):\HT^{s+m,p}_{FIO}(M)\to W^{r,p}(\R^{n+1})$, and $f\mapsto \ka^{*}(\tilde{\psi}_{\ka}f)$ is bounded from $\HT^{s+m,p}_{FIO}(\Rn)$ to $\HT^{s+m,p}_{FIO}(M)$ by Theorem \ref{thm:FIObdd}. Hence $\tilde{T}:\HT^{s+m,p}_{FIO}(\Rn)\to W^{r,p}(\R^{n+1})$ is bounded. We may thus work with $\tilde{T}$ in the remainder. Note also that $\tilde{T}$ is compactly supported.

Next, we want to reduce to the case where $r=0$. To this end, note that $\lb D\rb^{r}\tilde{T}\lb D\rb^{-r}:\HT^{s-r+m,p}_{FIO}(\Rn)\to L^{p}(\R^{n+1})$ is bounded. Since $\tilde{T}$ is compactly supported, we can multiply $\tilde{T}$ by smooth cutoffs $\rho_{1}$ and $\rho_{2}$, and apply Lemma \ref{lem:propersupp} twice, to write
\[
\lb D\rb^{r}\tilde{T}\lb D\rb^{-r}=\lb D\rb^{r}\rho_{1}\tilde{T}\rho_{2}\lb D\rb^{-r}=\tilde{S}_{1}\tilde{T}\tilde{S}_{2}+R,
\]
for compactly supported pseudodifferential operators $\tilde{S}_{1}$ and $\tilde{S}_{2}$ of order $r$ and $-r$, respectively, and a smoothing operator $R:\Sw'(\Rn)\to \Sw(\R^{n+1})$. Then $\tilde{S}_{1}\tilde{T}\tilde{S}_{2}:\HT^{s-r+m,p}_{FIO}(\Rn)\to L^{p}(\R^{n+1})$ is bounded, and the composition theorem for FIOs implies that $\tilde{S}_{1}\tilde{T}\tilde{S}_{2}\in I^{m-1/4}(\Rn,\R^{n+1};\tilde{C})$. Moreover, since $\lb D\rb^{r}$ and $\lb D\rb^{-r}$ are elliptic and because the smoothing term $R$ does not change the principal symbol, $\tilde{S}_{1}\tilde{T}\tilde{S}_{2}$ is a compactly supported element of $I^{m-1/4}(\Rn,\R^{n+1};\tilde{C})$ that is non-characteristic at $\tilde{c}$.

It only remains to express $\tilde{S}_{1}\tilde{T}\tilde{S}_{2}$ in standard form. This is where the projection conditions on $\tilde{C}$ come in. Namely, these allow us to apply Lemma \ref{lem:FIOcurv} to write $\tilde{S}_{1}\tilde{T}\tilde{S}_{2}$ as a sum of FIOs of order $m-1/4$ as in Section \ref{subsec:sharpstand}, composed with changes of coordinates, plus a smoothing term. Moreover, at the start of the proof we observed that we are allowed to work in a sufficiently small neighborhood of $c$ in $\Ca$ (and therefore also of $\tilde{c}$ in $\tilde{C}$). Hence Remark \ref{rem:smallcan} in fact yields a decomposition $\tilde{S}_{1}\tilde{T}\tilde{S}_{2}=T'S'+R'$, for $T'$ an FIO of order $m-1/4$ as in Section \ref{subsec:sharpstand}, $S'$ a change of coordinates on $\Rn$, and $R':\Sw'(\Rn)\to\Sw(\R^{n+1})$ a smoothing operator. Then $T'S'$ is non-characteristic at $\tilde{c}$ and bounded from $\HT^{s-r+m,p}_{FIO}(\Rn)$ to $L^{p}(\R^{n+1})$, given that $R'$ does not affect these properties. And, since $\tilde{S}_{1}\tilde{T}\tilde{S}_{2}$ is compactly supported and we are allowed to work in a sufficiently small neighborhood of $\tilde{c}$, the change of coordinates $S'$ does not affect the mapping properties of $\tilde{S}_{1}\tilde{T}\tilde{S}_{2}$ either.

To conclude, the operator $T'$ is as in Section \ref{subsec:sharpstand}, is non-characteristic at some point $c'$ (obtained from $\tilde{c}$ in terms of the change of coordinates $S'$), and is bounded from $\HT^{s-r+m,p}_{FIO}(\Rn)$ to $L^{p}(\R^{n+1})$. Thus Theorem \ref{thm:sharpstand} yields $s-r\geq d(p)-s(p)$, as required.
\end{proof}

\addtocontents{toc}{\protect\setcounter{tocdepth}{0}}

\section*{Acknowledgments}

The authors would like to thank the referee for their careful reading of the manuscript and for many helpful suggestions. The second author would also like to thank Po--Lam Yung for many inspiring conversations about the phenomenon of local smoothing. 

\addtocontents{toc}{\protect\setcounter{tocdepth}{2}}

\appendix

\section{Geometry and Fourier integral operators}\label{sec:background}

In this appendix we collect some basics on manifolds with bounded geometry and on the theory of Fourier integral operators on manifolds. We first discuss the main geometric assumptions on the underlying manifolds. We then recall some basics on distributional densities on manifolds, which will allow us to introduce Fourier integral operators on manifolds.

\subsection{Manifolds with bounded geometry}\label{subsec:geom}

In this subsection we recall some basics from Riemannian geometry. The underlying geometric assumptions correspond to those used in \cite{Triebel86,Triebel87,Taylor09} to define the classical local Hardy spaces on manifolds.

Throughout this subsection, we let $(M,g)$ be a complete Riemannian manifold of dimension $n\in\N$. Denote by $\A$ the smooth structure on $M$. That is, $\A$ is a maximal smooth atlas of homeomorphisms $\ka:U_{\ka}\to \Util_{\ka}\subseteq\Rn$, where the $U_{\ka}\subseteq M$ are open sets that cover $M$, and
\begin{equation}\label{eq:change}
\mu_{\la\ka}:=\la\circ\ka^{-1}:\ka(U_{\ka}\cap U_{\la})\to \la(U_{\ka}\cap U_{\la})
\end{equation}
is a diffeomorphism for all $\kappa,\la\in\A$ such that $U_{\ka}\cap U_{\la}\neq \emptyset$.

We assume throughout that $(M,g)$ has \emph{bounded geometry}, by which we mean that the injectivity radius of $(M,g)$ is positive, and that all covariant derivatives $\nabla^{k}R$, $k\in\Z_{+}$, of the Riemannian curvature tensor $R$ are bounded. Recall that the injectivity radius $\inj(M)\in[0,\infty]$ is the infimumum over $p\in M$ of the supremum over all $r>0$ such that the exponential map
\[
\exp_{p}:B_{r}(0)\to U_{r}(p):=\exp_{p}(B_{r}(0))\subseteq M
\]
at $p$ is a diffeomorphism. Here $B_{r}(x):=\{y\in\Rn\mid |x-y|<r\}\subseteq \Rn$ for all $x\in\Rn$, where we identify $T_{p}M$ with $\Rn$. Then $\ka=\exp^{-1}_{p}:U_{r}(p)\to B_{r}(0)$ is an element of $\A$, and the corresponding local coordinates are \emph{geodesic normal coordinates}. We will often alternate between denoting the domain of $\ka$ by $U_{\ka}$ and by $U_{r}(p)$.

One can reformulate the assumption that all covariant derivatives of the curvature tensor are bounded using the matrix expression $(g_{ij})_{i,j=1}^{n}$ of the metric tensor in local coordinate charts (see \cite{Eichhorn91}). That is, we assume that for all $0<r<\inj(M)$ there exists a $c>0$ and, for every $\alpha\in\Z_{+}^{n}$, a $C_{\alpha}\geq0$ such that
\begin{equation}\label{eq:localg}
|\det (g_{ij}(y))_{i,j=1}^{n}|\geq c\quad\text{and}\quad|\partial^{\alpha}_{y}g_{ij}(y)|\leq C_{\alpha}\text{ for }1\leq i,j\leq n,
\end{equation}
in every normal geodesic coordinate chart $\exp^{-1}_{p}:U_{r}(p)\to B_{r}(0)$, $p\in M$, and for all $y\in B_{r}(0)$. This implies similar uniform bounds for the inverse matrix $(g^{ij})_{i,j=1}^{n}$.

There is yet another equivalent formulation of the assumption of bounded geometry; this one, in terms of the changes of coordinates in \eqref{eq:change}, will be used frequently by us. Namely, we assume that, for all $0<r<\inj(M)$ and $\alpha\in\Z_{+}^{n}$, there exists a $C_{\alpha}\geq0$ such that
\begin{equation}\label{eq:uniformchange}
|\partial^{\alpha}_{y}\mu_{\la\ka}(y)|\leq C_{\alpha}
\end{equation}
for all normal geodesic charts $\ka=\exp_{p}^{-1}:U_{r}(p)\to B_{r}(0)$ and $\la=\exp_{q}^{-1}:U_{r}(q)\to B_{r}(0)$ with $U_{r}(p)\cap U_{r}(q)\neq \emptyset$, $p,q\in M$, and all $y\in \ka(U_{r}(p)\cap U_{r}(q))$.

 To avoid confusion, we note that some authors use the terminology ``bounded geometry" for the weaker assumption that $M$ has positive injectivity radius and Ricci tensor bounded from below (see e.g.~\cite{Meda-Veronelli22,MaMeVa22}).

For us, the main reason to make the assumption of bounded geometry is that it implies the existence of a \emph{uniformly locally finite} cover by geodesic balls, in a sense which is specified by the following lemma.

\begin{lemma}\label{lem:cover}
For every $\delta\in (0,\inj(M)/3]$, there exist countable collections $K\subseteq \A$ and $(\psi_{\ka})_{\ka\in K}\subseteq C^{\infty}_{c}(M)$ with the following properties.
\begin{enumerate}
\item\label{it:cover1} Each $\kappa=\exp_{p_{\ka}}^{-1}:U_{\delta}(p_{\ka})\to B_{\delta}(0)$, for $\ka\in K$, is a geodesic normal coordinate chart for some $p_{\ka}\in M$, and $\cup_{\ka\in K}U_{\delta}(p_{\ka})=M$.
\item\label{it:cover2} There exists an $N\in\N$ such that, for every $\ka\in K$, the collection of $\la\in K$ with $U_{\inj(M)/3}(p_{\ka})\cap U_{\inj(M)/3}(p_{\la})\neq \emptyset$ contains at most $N$ elements.
\item\label{it:cover3} For every $\ka\in K$ one has $\supp(\psi_{\ka})\subseteq U_{\delta}(p_{\ka})$, and
\[
\sum_{\ka\in K}\psi_{\ka}(x)^{2}=1\quad(x\in M).
\]
\item\label{it:cover4} Writing $\psit_{\ka}:=\psi_{\ka}\circ\ka^{-1}$ for $\ka\in K$, for each $\alpha\in \Z_{+}^{n}$ one has
\[
\sup\{|\partial_{y}^{\alpha}\psit_{\ka}(y)|\mid \ka\in K,y\in B_{\delta}(0)\}<\infty.
\]
\end{enumerate}
\end{lemma}

Note that the family $(\psi_{\ka})_{\ka\in K}$ is also uniformly locally finite, in the sense that, for each $x\in M$, one has $\psi_{\ka}(x)\neq 0$ for at most $N$ different $\ka\in K$. In fact, we typically choose $\delta\leq \inj(M)/9$, which allows us to treat intersecting geodesic balls within the same coordinate chart. This is in turn useful for the proof of the crucial Corollary \ref{cor:boundedPQ}.

The statement and the proof of Lemma \ref{lem:cover} are essentially contained in \cite[Lemmas A1.1.2 and A1.1.3]{Shubin92}. See also \cite[Section 7.1]{Triebel92}) and \cite[Lemma 2.1]{LiRoSoYa22preprint}.

\subsection{Distributional densities}\label{subsec:densities}

In this subsection we collect some basics on distributional densities on manifolds, cf.~\cite{Duistermaat11,Hormander03}. We apply this theory to a Riemannian manifold $(M,g)$ with bounded geometry, but for the material presented here we may assume that $M$ is a general $n$-dimensional smooth manifold.

Let $\gamma\in\R$. We write $\Da(M,\Omega_{\gamma})$ for the space of smooth compactly supported sections of the complex line bundle $\Omega_{\gamma}(M)$. The fibers of this bundle are the one-dimensional vector spaces $\Omega_{\gamma}(T_{x}M)$, $x\in M$, consisting of all functions $\sigma:\Lambda^{n}T_{x}M\to\C$, on the exterior product $\Lambda^{n}T_{x}M$, with the property that $\sigma(\lambda v)=|\lambda|^{\gamma}\sigma(v)$ for all $v\in\Lambda^{n}T_{x}M$ and $\lambda\in\R\setminus\{0\}$. If $\rho_{0}\in C^{\infty}(M,\Omega_{1})$ is a fixed strictly positive density on $M$ (such as the Riemannian density $\rho_{g}$), then
\begin{equation}\label{eq:identifydens}
f\mapsto f\rho_{0}^{\gamma}\quad (f\in C^{\infty}_{c}(M))
\end{equation}
is a bijection between $C^{\infty}_{c}(M)$ and $\Da(M,\Omega_{\gamma})$.

The topology on $\Da(M,\Omega_{\gamma})$ is defined in the same manner as that of the test functions $C^{\infty}_{c}(M)$. In fact, $\Da(M,\Omega_{\gamma})$ inherits the topology of $C^{\infty}_{c}(M)$ via \eqref{eq:identifydens}. Moreover, the resulting topology does not depend on the choice of $\rho_{0}$, by definition of the inductive limit topology on $C^{\infty}_{c}(M)$ and because the quotient of two strictly positive densities is bounded from above and below on compact subsets of $M$.

A distributional density of order $\gamma$ on $M$ is a continuous linear functional on $\Da(M,\Omega_{1-\gamma})$. We denote by $\Da'(M,\Omega_{\gamma})$ the space of distributional densities of order $\gamma$, endowed with the weak topology, and $\lb u,\rho\rb_{M}$ denotes the duality between $u\in\Da'(M,\Omega_{\gamma})$ and $\rho\in \Da(M,\Omega_{1-\gamma})$.  The space $C^{\infty}(M,\Omega_{\gamma})$ of smooth sections of $\Omega_{\gamma}$ is embedded in $\Da'(M,\Omega_{\gamma})$, via $\lb u,\rho\rb_{M}=\int_{M}u(x)\rho(x)$ for $u\in C^{\infty}(M,\Omega_{\gamma})$ and $\rho\in \Da(M,\Omega_{1-\gamma})$. This is well defined because $u\rho\in \Da(M,\Omega_{1})$ is an integrable density in the classical sense. Moreover, $\Da(M,\Omega_{\gamma})$ is dense in $\Da'(M,\Omega_{\gamma})$, as follows by approximating a $u\in \Da'(M,\Omega_{\gamma})$ in local coordinate charts by test functions, and then lifting these back to $M$ (see e.g.~\eqref{eq:Pinv}).

As is typical for the coordinate-invariant theory of Fourier integral operators (see e.g.~\cite{Duistermaat11,Hormander09}), we will mostly consider $\gamma=1/2$ in this definition, in which case we refer to \emph{half densities}. 
However, we note that $u\mapsto \rho_{0}^{\gamma-1}u$ identifies functionals on $C^{\infty}_{c}(M)$ with distributional densities of a general\footnote{In particular, after making this identification, our approach coincides with the notion of distributions on manifolds from \cite[Chapter 7]{Triebel92} (where $\gamma=1$) and \cite{BeHiSo20,BeHiSo21} (where $\gamma=0$).} order $\gamma\in\R$, as in \eqref{eq:identifydens}.


Let $\gamma\in\R$. Recall that if $A:E\to F$ is a linear map between $n$-dimensional vector spaces, then the pullback $A^{*}:\Omega_{\gamma}(F)\to\Omega_{\gamma}(E)$ is given by
\begin{equation}\label{eq:pullbackexterior}
(A^{*}\sigma)(v)=\sigma(Av_{1}\wedge\ldots\wedge Av_{n})
\end{equation}
for $\sigma\in \Omega_{\gamma}(F)$ and $v=v_{1}\wedge \ldots\wedge v_{n}\in \Lambda^{n}E$.
If $N$ is another $n$-dimensional manifold and $\chi:M\to N$ is a diffeomorphism, then the pullback $\chi^{*}_{\gamma}:C^{\infty}(N,\Omega_{\gamma})\to C^{\infty}(M,\Omega_{\gamma})$ is
\begin{equation}\label{eq:pullbacktest}
\chi^{*}_{\gamma}\rho(x):=(\partial_{x}\chi(x))^{*}\big(\rho(\chi(x))\big)
\end{equation}
for $\rho\in C^{\infty}(N,\Omega_{\gamma})$ and $x\in M$. Finally, the pullback $\chi_{\gamma}^{*}u\in \Da'(M,\Omega_{\gamma})$ of $u\in \Da'(N,\Omega_{\gamma})$ is given by
\begin{equation}\label{eq:pullbackdensity}
\lb \chi_{\gamma}^{*}u,\rho\rb_{M}:=\lb u,(\chi^{-1})_{1-\gamma}^{*}\rho\rb_{N}
\end{equation}
for $\rho\in \Da(M,\Omega_{1-\gamma})$. We also set $(\chi_{*})_{\gamma}:=(\chi^{-1})^{*}_{\gamma}$, acting on smooth or distributional densities. We will almost exclusively consider $\gamma=1/2$, in which case we write $\chi^{*}=\chi^{*}_{1/2}$ and $\chi_{*}=(\chi^{-1})^{*}_{1/2}$, for simplicity of notation.

Let $(U_{\ka})_{\ka\in K}$ be a cover of $M$ by coordinate charts $\ka:U_{\ka}\to\Util_{\ka}\subseteq\Rn$, and set
\begin{equation}\label{eq:Qinv}
u_{\ka}:=(\ka_{*})_{\gamma}u
\end{equation}
for $u\in\Da'(M,\Omega_{\gamma})$ and $\ka\in\A$. This gives rise to a collection $(u_{\ka})_{\ka\in K}$ of distributions $u_{\ka}\in\Da'(\Util_{\ka})=\Da'(\Util_{\ka},\Omega_{\gamma})$ such that
\begin{equation}\label{eq:densityequiv}
u_{\ka}=(\la\circ\ka^{-1})^{*}_{\gamma}u_{\la}
\end{equation}
in $D'(\ka(U_{\la}\cap U_{\ka}))$, for all $\ka,\la\in K$. Here and throughout, we identify distributions on $\Rn$ with distributional densities of order $\gamma$ using the standard basis. Note that
\[
\lb u_{\ka},\rho_{\ka}\rb_{\Rn}=\lb (\ka_{*})_{\gamma}u,(\ka_{*})_{1-\gamma}\rho\rb_{\Rn}=\lb u,\rho\rb_{M}
\]
for all $u\in \Da'(M,\Omega_{\gamma})$, $\rho\in\Da(U_{\ka},\Omega_{1-\gamma})$ and $\ka\in K$.

Conversely, for any collection $(u_{\ka})_{\ka\in K}$ of $u_{\ka}\in\Da'(\Util_{\ka})$,
\begin{equation}\label{eq:Pinv}
u:=\sum_{\ka\in K}\psi_{\ka}^{2}\ka_{\gamma}^{*}(u_{\ka})\in\Da'(M,\Omega_{\gamma})
\end{equation}
is a well-defined distributional density. Here $(\psi_{\ka}^{2})_{\ka\in K}$ is some partition of unity. If we additionally restrict this correspondence to those $(u_{\ka})_{\ka\in K}$ satisfying \eqref{eq:densityequiv}, then \eqref{eq:Qinv} and \eqref{eq:Pinv} are each other's inverses.

One can also express \eqref{eq:densityequiv} without any reference to distributional densities. Indeed, the pullback $\chi_{\gamma}^{*}u\in \Da'(V)$ of a distribution $u\in\Da'(W)$ is given by
\begin{equation}\label{eq:pullbackRn}
\lb \chi^{*}_{\gamma}u,\phi\rb_{\Rn}=\lb u,|\det\partial(\chi^{-1})|^{1-\gamma}\phi\circ\chi^{-1}\rb_{\Rn}
\end{equation}
for $\phi\in\Da(V)$. Here we used that $A^{*}\sigma=|\det(A)|^{1-\gamma}\sigma$ in \eqref{eq:pullbackexterior} when $E=F$. We note that \eqref{eq:pullbackRn}, with $\gamma=0$, corresponds to the pullback of a distribution on $\Rn$ as it is typically defined. Now \eqref{eq:pullbackRn} shows that \eqref{eq:densityequiv} is equivalent to the identity
\begin{equation}\label{eq:density}
u_{\ka}=|\det\partial(\la\circ \ka^{-1})|^{\gamma}(\la \circ\ka^{-1})^{*}_{0}u_{\la}
\end{equation}
in $\Da'(\ka(U_{\la}\cap U_{\ka}))$, for all $\ka,\la\in K$.

The remarks above show that one could alternatively define distributional densities of order $\gamma$ to be collections $(u_{\ka})_{\ka\in K}$ of distributions $u_{\ka}\in\Da'(\Util_{\ka})$ satisfying \eqref{eq:density}. However, doing so would complicate notation and computations in Sections \ref{sec:sequences} and \ref{sec:spacesmanifold}. Instead, in those sections we use an association of sequences of distributions to distributional densities which is more symmetric with respect to the partition of unity $(\psi_{\ka}^{2})_{\ka\in K}$ in \eqref{eq:Pinv}.

\subsection{Fourier integral operators}\label{subsec:FIOs}

This subsection contains some background on the Fourier integral operators under consideration in this article. We first collect the relevant definitions, and then some basic results about these operators.

\subsubsection{Definitions}

We refer to \cite{Hormander07,Hormander09,Duistermaat11} for the general theory of Fourier integral operators on manifolds and the associated symplectic geometry. In particular, as in \cite[Section 25.1]{Hormander09}, Fourier integral operators are defined using the class $I^{m}(X,\Lambda,\Omega_{1/2})\subseteq\Da'(X,\Omega_{1/2})$ of Lagrangian distributions of order $m\in\R$, for $X$ a smooth manifold and $\Lambda\subseteq T^{*}X\setminus o$ a closed conic Lagrangian submanifold. Recall that $o$ is the zero section. On the other hand, our main results rely on estimates in local coordinates, and these distributions have well-known concrete oscillatory integral representations in local coordinates (see e.g.~Lemma \ref{lem:FIOcurv} below).

For $m\in\R$ and $n_{1},n_{2}\in\N$, we will work with the standard Kohn--Nirenberg symbol class $S^{m}(\R^{n_{1}}\times \R^{n_{2}})$, consisting of those $a\in C^{\infty}(\R^{n_{1}}\times\R^{n_{2}})$ such that
\begin{equation}\label{eq:Kohn-Nirenberg}
\sup\{\lb\eta\rb^{|\beta|-m}|\partial_{z}^{\alpha}\partial_{\eta}^{\beta}a(z,\eta)|\mid z\in\R^{n_{1}},\eta\in \R^{n_{2}}\}<\infty
\end{equation}
for all $\alpha\in\Z_{+}^{n_{1}}$ and $\beta\in \Z_{+}^{n_{2}}$. We write $S^{m}(\R^{n_{2}})$ for the subspace of symbols which do not depend on the $z$ variable. As in \cite{HaPoRo20}, one could consider symbols from more general symbol classes, but for simplicity we will not do so here.

Let $M$ and $N$ be smooth manifolds. The Schwartz kernel theorem (see \cite[Theorem 1.1.1]{Duistermaat11}) yields a bijection between continuous linear maps $T:\Da(M,\Omega_{1/2})\to \Da'(N,\Omega_{1/2})$ and kernels $K\in \Da'(N\times M,\Omega_{1/2})$. A \emph{homogeneous canonical relation} from $T^{*}M$ to $T^{*}N$ is a submanifold $\Ca\subseteq (T^{*}N\setminus o)\times (T^{*}M\setminus o)$ such that
\[
\Ca':=\{(x,y,\xi,-\eta)\mid (x,\xi,y,\eta)\in\Ca\}
\]
is a closed conic Lagrangian submanifold of $T^{*}(N\times M)\setminus o$. Then a \emph{Fourier integral operator of order $m\in\R$ associated with $\Ca$} is an operator $T:\Da(M,\Omega_{1/2})\to \Da'(N,\Omega_{1/2})$ such that the Schwartz kernel of $T$ is an element of $I^{m}(N\times M,\Ca',\Omega_{1/2})$.  The class of all such Fourier integral operators will be denoted by $I^{m}(M,N;\Ca)$.

An FIO $T\in I^{m}(M,N;\Ca)$ is \emph{compactly supported} if its Schwartz kernel $K$ is compactly supported in $N\times M$, and \emph{properly supported} if both projections from $\supp(K)\subseteq N\times M$, to $N$ and to $M$, are proper maps. We say that $T$ is \emph{non-characteristic} at a point $c\in\Ca$ if the principal symbol $a\in S^{m}(\Ca,L_{\Ca})$ of $T$ is of the form $a=s_{0}a_{0}$ for some $a_{0}\in S^{m}(\Ca,L_{\Ca})$ homogeneous of degree $m$ and nowhere vanishing, and $|s_{0}|$ bounded away from zero in a conic neighborhood of $c$. We refer to \cite{Duistermaat11,Hormander09} for more on the principal symbol of an FIO, the Maslov line bundle $L_{\Ca}$ and the symbol class $S^{m}(\Ca,L_{\Ca})$, but we note that these notions reduce to more concrete concepts when working in local coordinates (see Section \ref{sec:sharpness} and \eqref{eq:elliptic}). 
One says that $T$ \emph{elliptic} if it is non-characteristic at every point $c\in\Ca$.  

The transpose $T^{t}:\Da(N,\Omega_{1/2})\to\Da'(M,\Omega_{1/2})$ of a $T$, given by $\lb T^{t}v,u\rb_{M}=\lb Tu,v\rb_{N}$ for $v\in\Da(N,\Omega_{1/2})$ and $u\in\Da(M,\Omega_{1/2})$, is an element of $I^{m}(N,M;\Ca^{-1})$, where
\begin{equation}\label{eq:inversecan}
\Ca^{-1}:=\{(y,\eta,x,\xi)\mid (x,\xi,y,\eta)\in\Ca)\subseteq (T^{*}M\setminus o)\times (T^{*}N\setminus o)
\end{equation}
is the inverse canonical relation.

We consider two types of FIOs, by making geometric assumptions on their canonical relations. Firstly, we deal with $T\in I^{m}(M,N;\Ca)$ for $\Ca$ a canonical relation which is a \emph{local canonical graph}. This means that the natural projections
\[
\Pi_{T^{*}N}:\Ca\to T^{*}N\setminus o\ \ \text{and}\ \ \Pi_{T^{*}M}:C\to T^{*}M\setminus o
\]
are local diffeomorphisms, or equivalently that $\Ca$ is locally the graph of a symplectomorphism $\chi$ between open subsets of $T^{*}M\setminus o$ and $T^{*}N\setminus o$. This condition implies in particular that $M$ and $N$ have the same dimension, and $\Ca^{-1}$ is then a local canonical graph as well. If $T$ is properly supported, then $T:\Da(M,\Omega_{1/2})\to \Da(N,\Omega_{1/2})$ continuously. Then $T^{t}$ has the same properties with $M$ and $N$ reversed, so $T$ extends uniquely to a continuous map $T:\Da'(M,\Omega_{1/2})\to \Da'(N,\Omega_{1/2})$ by setting $\lb Tu,v\rb_{N}:=\lb u,T^{t}v\rb_{M}$ for $u\in\Da'(M,\Omega_{1/2})$ and $v\in\Da(N,\Omega_{1/2})$.

Secondly, we consider the case where $M$ has dimension $n\geq 2$ and $N$ has dimension $n+1$. Then, cf.~\cite{MoSeSo93}, we say that $\Ca$ satisfies the \emph{cinematic curvature condition} if the natural projections
\[
\Pi_{T^{*}M}:\Ca\to T^{*}M\setminus o\ \ \text{and}\ \ \Pi_{N}:\Ca\to N
\]
are submersions, and if for each $z\in \Pi_{N}\Ca$ the conic hypersurface $\Pi_{T^{*}_{z}N}\Ca\subseteq T^{*}_{z}N\setminus \{0\}$ has everywhere $n-1$ non-vanishing principal curvatures. Here $\Pi_{T_{z}^{*}N}:\Ca\to T^{*}_{z}N\setminus \{0\}$ is the natural projection onto the fiber $T^{*}_{z}N$ of $z$. This condition is invariant under changes of coordinates on $M$ and $N$.

\subsubsection{Results}

We now collect some basic results about Fourier integral operators.

Let $\Ca$ be a homogeneous canonical relation from $T^{*}\R^{n}$ to $T^{*}\R^{n+k}$, for $n\in\N$ and $k\in\Z_{+}$, and let $m\in\R$ and $a\in S^{m}(\R^{n+k}\times \R^{n})$. Also let $V\subseteq \R^{n+k}\times(\R^{n}\setminus\{0\})$ be open, conic and such that $(z,\eta)\in V$ whenever $(z,\eta)\in\supp(a)$ and $\eta\neq0$. Let $\Phi\in C^{\infty}(V)$ be real-valued and positively homogenous of degree $1$ in the fiber variable $\eta$, and let $T\in I^{m-k/4}(\Rn,\R^{n+k};\Ca)$ be such that
\[
Tf(z):=\int_{\R^{n}}e^{i\Phi(z,\eta)}a(z,\eta)\wh{f}(\eta)\ud \eta
\]
for all $f\in\Da(\R^{n})$ and $z\in\R^{n+k}$. We then say that $T$ is a Fourier integral operator \emph{in standard form}, with phase function $\Phi$ and symbol $a$. Moreover, we call an operator $S:\Da'(\Rn)\to\Da'(\Rn)$ a \emph{change of coordinates} if there exists a diffeomorphism $\chi:U\to U'$ between open subsets $U,U'\subseteq\Rn$, and a $\psi\in C_{c}(U')$, such that $Sf=\chi^{*}_{0}(\psi f)$ for all $f\in\Da'(\Rn)$. Finally, an operator $R:\Sw'(\R^{n_{1}})\to \Sw(\R^{n_{2}})$ is a \emph{smoothing operator} if the Schwartz kernel of $R$ is an element of $\Sw(\R^{n_{2}}\times \R^{n_{1}})$.

The following lemma reduces the analysis of relatively general Fourier integral operators to operators in standard form.

\begin{lemma}\label{lem:FIOcurv}
Let $n\in\N$ and $k\in\Z_{+}$, and let $\Ca$ be a homogeneous canonical relation from $\R^{n}$ to $\R^{n+k}$ such that the projections
\[
\Pi_{\R^{n+k}}:\Ca\to \R^{n+k}\ \ \text{and}\ \ \Pi_{T^{*}\R^{n}}:\Ca\to T^{*}\Rn\setminus o
\]
are submersions. Let $T\in I^{m-k/4}(\R^{n},\R^{n+k};\Ca)$, for $m\in\R$, be compactly supported. Then there exist an $l\in\N$, homogeneous canonical relations $(\Ca_{j})_{j=1}^{l}$ from $\Rn$ to $\R^{n+k}$,  Fourier integral operators $(T_{j})_{j=1}^{l}$ in standard form satisfying $T_{j}\in I^{m-k/4}(\Rn,\R^{n+k};\Ca_{j})$ for each $j\in\{1,\ldots, l\}$, changes of coordinates $(S_{j,1})_{j=1}^{l}$ and $(S_{j,2})_{j=1}^{l}$, and a smoothing operator $R:\Sw'(\R^{n})\to \Sw(\R^{n+k})$, with the following properties:
\begin{enumerate}
\item\label{it:FIOcurv1} $T=\sum_{j=1}^{l}S_{j,1}T_{j}S_{j,2}+R$;
\item\label{it:FIOcurv2} For each $j\in\{1,\ldots, l\}$, the phase function $\Phi_{j}\in C^{\infty}(V_{j})$ of $T_{j}$ satisfies $\rank\,\partial^{2}_{x\eta}\Phi_{j}(z,\eta)=n$ for all $(z,\eta)\in V_{j}$, where $x=(z_{1},\ldots, z_{n})$, and the symbol $a_{j}\in S^{m}(\R^{n+k}\times \Rn)$ has compact support in $\R^{n+k}$.
\end{enumerate}
Moreover, suppose that $k=1$ and that $\Ca$ satisfies the cinematic curvature condition. Then one may additionally assume that the following holds for all $j\in\{1,\ldots, l\}$ and  $(z_{0},\eta_{0})\in V_{j}$. One has $\Pi_{T^{*}_{z_{0}}\R^{n+1}}\Ca_{j}=\{\partial_{z}\Phi_{j}(z_{0},\eta)\mid (z_{0},\eta)\in V_{j}\}$, and the unit normal vectors $\pm\theta$ to $\Pi_{T^{*}_{z_{0}}\R^{n+1}}\Ca_{j}$ at $\partial_{z}\Phi_{j}(z_{0},\eta_{0})$ satisfy
\begin{equation}\label{eq:cincurv}
\rank\,\partial^{2}_{\eta\eta}(\partial_{z}\Phi_{j}(z_{0},\eta)\cdot\theta)
|_{\eta=\eta_{0}}=n-1.
\end{equation}
\end{lemma}

One may express \eqref{eq:cincurv} as in \eqref{eq:Gaussmap}:
\[
\rank\,\partial_{\eta\eta}^{2}(\partial_{z}\Phi_{j}(z_{0},\eta)\cdot G_{j}(z_{0},\eta_{0}))|_{\eta=\eta_{0}}=n-1,
\]
where $G_{j}(z_{0},\eta_{0}):=G_{j,0}(z_{0},\eta_{0})/|G_{j,0}(z_{0},\eta_{0})|$ is the generalized Gauss map.

A statement similar to the first one in the lemma can be found in \cite[pp.~29-30]{Hormander09}, under more general assumptions on the canonical relation. Note that the assumptions of the lemma hold with $k=0$ if and only if $\Ca$ is a local canonical graph, since $\Ca$ then has dimension $2n$. In this case, the statement is essentially contained in \cite[pp.~26-27]{Hormander09}, \cite[Proposition 6.2.4]{Sogge17}, \cite[Proposition 2.14]{HaPoRo20} and \cite[p.~108]{Duistermaat11}, and the changes of coordinates $(S_{j,1})_{j=1}^{l}$ can be omitted. The proof for general $k$ is similar to that for $k=0$. The second statement can in turn be found in \cite[Section 8.1]{Sogge17}. One can also find a more detailed proof of both statements in \cite[Lemma 2.2]{LiRoSoYa22preprint}.

\begin{remark}\label{rem:timeparameter}
In Section \ref{sec:wave}, we apply this lemma to canonical relations $\Ca$ from $\Rn$ to $\R^{n+1}$ of the form
\[
\Ca=\{(x,t,\xi,\tau,y,\eta)\mid (x,\xi)=\chi_{t}(y,\eta), \tau=q(x,t,\xi)\},
\]
where each $\chi_{t}$ is a canonical transformation on $T^{*}(\R^{n})\setminus o$. 
In this case the changes $S_{j,1}$ of $z$ coordinates are not needed to guarantee that $\partial_{x\eta}^{2}\Phi_{j}(x,t,\eta)=n$. It follows from the inverse function theorem that, in suitable local coordinates, any canonical relation satisfying the cinematic curvature condition has this form (see e.g.~\cite[Section 8.1]{Sogge17}).
\end{remark}

\begin{remark}\label{rem:smallcan}
It is relevant for the proof of Theorem \ref{thm:sharpgen} to note that, if $\Ca$ is contained in a sufficiently small neighborhood of a given point $\tilde{c}\in \Ca$, then the decomposition in \eqref{it:FIOcurv1} simplifies and the summation and one of the changes of coordinates are not necessary. More precisely, one may write $T=T'S'+R$ for $T'\in I^{m-k/4}(\Rn,\R^{n+k};\Ca')$ a Fourier integral operator in standard form associated with a homogeneous canonical relation $\Ca'$, a change of coordinates $S'$, and a smoothing operator $R$.
\end{remark}

Finally, for our main results we use the following lemma, a slight variation of \cite[Proposition 18.1.22]{Hormander07}. It allows us to apply the composition theorem for Fourier integral operators (see e.g.~\cite[Theorem 25.2.3]{Hormander09} or \cite[Theorem 6.2.2]{Sogge17}), which deals with properly supported operators. The proof of the lemma is similar to that of \cite[Proposition 18.1.22]{Hormander07} and can also be found in \cite[Lemma 2.4]{LiRoSoYa22preprint}.

\begin{lemma}\label{lem:propersupp}
Let $T_{1}$ and $T_{2}$ be pseudodifferential operators on $\Rn$, of orders $m_{1}\in\R$ and $m_{2}\in\R$, respectively. Suppose that either $T_{1}$ or $T_{2}$ is compactly supported. Then there exist a compactly supported pseudodifferential $S$ of order $m_{1}+m_{2}$, and a smoothing operator $R$, such that $T_{1}T_{2}=S+R$.
\end{lemma}

\section{Atomic decomposition}\label{sec:atomic}

This appendix contains the atomic decomposition of $\HT^{s,1}_{FIO}(\Rn)$. This atomic decomposition was proved for $s=(n-1)/4$ in \cite{Smith98a}. In the proof of Theorem \ref{thm:localization}, we use that every $\HT^{s,1}_{FIO}(\Rn)$-atom has compact support, and this support condition is not preserved when directly applying the fractional derivative $\lb D\rb^{-s-(n-1)/4}$ to an $\HT^{(n-1)/4,1}_{FIO}(\Rn)$-atom. Moreover, we believe that it is of independent interest to include an atomic decomposition of $\HT^{s,1}_{FIO}(\Rn)$ for general $s\in\R$.

Throughout this section, fix $s\in\R$. An $f\in W^{s,2}(\Rn)$ is an $\HT^{s,1}_{FIO}(\Rn)$-\emph{atom}, associated with a ball $B_{\sqrt{\tau}}(y,\nu)\subseteq\Sp$, for $(y,\nu)\in\Sp$ and $\tau>0$, if
\begin{equation}\label{eq:supportatom}
\supp(f)\subseteq \{x\in\Rn\mid (|\nu\cdot(x-y)|+|x-y|^{2})^{1/2}\leq \sqrt{\tau}\}
\end{equation}
and $\|M_{\tau}^{\nu}(D)f\|_{W^{s,2}(\Rn)}\leq \tau^{-n/2}$. Here, for $\xi\in\Rn\setminus\{0\}$,
\[
M^{\nu}_{\tau}(\xi):=\big(1+\tau^{-2}\big(\lb\xi\rb^{-1}+\min(|\hat{\xi}-\nu|^{2},|\hat{\xi}+\nu|^{2})\big)^{2}\big)^{n/2}.
\]
Then $M^{\nu}_{\tau}$ is a weight which is concentrated in part of the union of two cones, in the direction of $\nu$ and $-\nu$. In this terminology, $B_{\sqrt{\tau}}(y,\nu)$ is a ball of radius $\sqrt{\tau}$ around $(y,\nu)$ with respect to a sub-Riemannian metric on $\Sp$ which arises from contact geometry, and $f$ is supported in approximately the projection of such a ball onto $\Rn$. However, for the present article these subtleties play no role whatsoever, given that we will only use the compact support condition \eqref{eq:supportatom} and Proposition \ref{prop:atom}. Instead, we refer the interested reader to \cite{Smith98a,HaPoRo20} for more on the nomenclature.

The following proposition contains the atomic decomposition of $\HT^{s,1}_{FIO}(\Rn)$. Fix $r\in C^{\infty}_{c}(\Rn)$ such that $r(\xi)=1$ if $|\xi|\leq 1$, and $r(\xi)=0$ if $|\xi|\geq 2$.

\begin{proposition}\label{prop:atom}
The collection of $\HT^{s,1}_{FIO}(\Rn)$-atoms of radius at most $1$ is a uniformly bounded subset of $\HT^{s,1}_{FIO}(\Rn)$. Moreover, there exists a $C\geq0$ with the following property. For each $f\in\HT^{s,1}_{FIO}(\Rn)$ there exists a sequence $(f_{k})_{k=1}^{\infty}$ of $\HT^{s,1}_{FIO}(\Rn)$-atoms, associated with balls of radius at most $1$, and an $(\alpha_{k})_{k=1}^{\infty}\in\ell^{1}$, such that
\begin{equation}\label{eq:atomdecomp}
f=\sum_{k=1}^{\infty}\alpha_{k}f_{k}+r(D)f\quad\text{and}\quad\sum_{k=1}^{\infty}|\alpha_{k}|\leq C\|f\|_{\HT^{s,1}_{FIO}(\Rn)}.
\end{equation}
\end{proposition}
\begin{proof}
The proof in \cite{Smith98a} for $s=(n-1)/4$, which in turn is a modification of the proof of the atomic decomposition of $\HT^{1}(\Rn)$, extends almost directly to general $s\in\R$. Most notably, to obtain the decomposition in \eqref{eq:atomdecomp}, one may replace the factor $\lb D\rb^{(n-1)/2}$ in \cite[Lemma 3.10]{Smith98a} by $\lb D\rb^{s+(n-1)/4}$, after which the remainder of the proof is identical to that of \cite[Theorem 3.5]{Smith98a}.

To show that the collection of $\HT^{s,1}_{FIO}(\Rn)$-atoms of radius at most $1$ is uniformly bounded, one can modify the notion of a $\phi$-molecule in \cite[Definition 3.3]{Smith98a}, again by changing the factor $\lb D\rb^{(n-1)/2}$ to $\lb D\rb^{s+(n-1)/4}$. Then one can show as in \cite[Lemma 3.4(b)]{Smith98a} that an $\HT^{s,1}_{FIO}(\Rn)$-atom is the sum of two such $\phi$-molecules, using that \cite[Lemma 2.13(b)]{Smith98a} is already adapted to general $s$. In turn, that the collection of these $\phi$-molecules form a uniformly bounded subset of $\HT^{s,1}_{FIO}(\Rn)$ follows immediately from the case where $s=(n-1)/4$, given that $\phi$-molecules need not be compactly supported and are well behaved under application of $\lb D\rb^{s-(n-1)/4}$.
\end{proof}

We note in passing, and to prove the localization principle for $\HT^{s,p}(\Rn)$ in \eqref{eq:localclassical} in the same way as that for $\Hps$ in Theorem \ref{thm:localization}, that the atomic decomposition of $\HT^{s,1}(\Rn)$ also extends from $s=0$ to general $s\in\R$, by a similar modification of the proof. An $f\in W^{s,2}(\Rn)$ is an $\HT^{s,1}(\Rn)$-\emph{atom}, associated with a ball $B_{\tau}(y)\subseteq\Rn$, for $y\in\Rn$ and $\tau>0$, if $\supp(f)\subseteq B_{\tau}(y)$ and $\|f\|_{W^{s,2}(\Rn)}\leq \tau^{-n/2}$.

\begin{lemma}\label{lem:atomHp}
The collection of $\HT^{s,1}(\Rn)$-atoms is a uniformly bounded subset of $\HT^{s,1}(\Rn)$. Moreover, there exists a $C\geq0$ with the following property. For each $f\in\HT^{s,1}(\Rn)$ there exists a sequence $(f_{k})_{k=1}^{\infty}$ of $\HT^{s,1}(\Rn)$-atoms, associated with balls of radius at most $1$, and an $(\alpha_{k})_{k=1}^{\infty}\in\ell^{1}$, such that
\[
f=\sum_{k=1}^{\infty}\alpha_{k}f_{k}+r(D)f\quad\text{and}\quad\sum_{k=1}^{\infty}|\alpha_{k}|\leq C\|f\|_{\HT^{s,1}(\Rn)}.
\]
\end{lemma}

\section{More on variable-coefficient Wolff-type inequalities}\label{sec:Wolff}

In this appendix we will prove a converse statement to Proposition \ref{prop:decouple}. Such a converse statement is contained in \cite[Corollary 4.2]{Rozendaal22b} for the Euclidean half-wave equation, with a relatively simple proof that uses the invertibility of the Euclidean half-wave propagators $e^{it\sqrt{-\Delta}}$, for $t\in\R$. By contrast, a general FIO need not be invertible, and in fact the operators in Proposition \ref{prop:decouple} are never invertible at fixed times, due to the assumption that their symbol has compact spatial support. Instead, we may merely hope for microlocal invertibility wherever the operator is non-characteristic. This makes the converse statement to Proposition \ref{prop:decouple} less elegant than that in \cite[Corollary 4.2]{Rozendaal22b}. Nonetheless, the main result of this appendix shows that, microlocally near any point where the operator $T$ is non-characteristic, one may equivalently replace the decoupling norm in \eqref{eq:BeHiSo} by the $\HT^{m-s(p),p}_{FIO}(\Rn)$ norm of the initial data.

Throughout, we consider an operator $T\in I^{m-1/4}(\Rn,\R^{n+1};\Ca)$ as in Section \ref{subsec:decouple}. Additionally, we will assume that $\Ca$ is as in Remark \ref{rem:timeparameter}:
\[
\Ca=\{(x,t,\xi,\tau,y,\eta)\mid (x,t,\eta)\in \supp(a),\eta\neq0, (x,\xi)=\chi_{t}(y,\eta), \tau=q(x,t,\xi)\},
\]
where each $\chi_{t}$ is a homogeneous canonical transformation between open subsets of $T^{*}(\R^{n})\setminus o$. This assumption always holds locally in appropriate coordinate systems and, in this appendix, it only serves to simplify the statement of the main result. In particular, it implies the following. Set $T_{t}f(x):=Tf(x,t)$ for $f\in\Sw(\Rn)$, $x\in\Rn$ and $t\in\R$. Then $T_{t}\in I^{m}(\Rn,\Rn;\Ca_{t})$, where
\[
\Ca_{t}:=\{(x,\xi,y,\eta)\mid \exists\,\tau\text{ such that }(x,t,\xi,\tau,y,\eta)\in\Ca\}
\]
is the graph of $\chi_{t}$ whenever this set is non-empty. The inverse canonical relation $\Ca_{t}^{-1}$, from \eqref{eq:inversecan}, is then the graph of $\chi_{t}^{-1}$.

We note that $T$ is non-characteristic at a point $c=(z_{0},\zeta_{0},y_{0},\eta_{0})\in \Ca$ if and only if there exists a conic neighborhood $V_{0}\subseteq \R^{n+1}\times(\Rn\setminus\{0\})$ of $(z_{0},\eta_{0})$ and a $c>0$ such that
\begin{equation}\label{eq:elliptic}
|a(z,\eta)|\geq c\lb\eta\rb^{m}\text{ for all }(z,\eta)\in V_{0}.
\end{equation}
Indeed, the principal symbol of each $T_{t}$ is (see e.g.~\cite[p.~27]{Hormander09})
\begin{equation}\label{eq:principalsym}
a(x,t,\eta)|\!\det\partial_{x\eta}^{2}\Phi(x,t,\eta)|^{-1/2},
\end{equation}
and $T$ is non-characteristic at $c$ if and only if this quantity is bounded in absolute value, both from above and below, by multiples of $\lb\eta\rb^{m}$, uniformly near $z_{0}=(x_{0},t_{0})$. Since $\det\partial_{x\eta}^{2}\Phi(x,t,\eta)\neq 0$, this reduces to \eqref{eq:elliptic}.

Also recall that the \emph{microsupport} $\WF'(R)\subseteq\Rn\times (\Rn\setminus\{0\})$ of a pseudodifferential operator $R:\Sw'(\Rn)\to\Sw'(\Rn)$ with symbol $r$ is the complement of the set of $(x_{0},\eta_{0})\in \Rn\times(\Rn\setminus\{0\})$ such that, in some conic neighborhood of $(x_{0},\eta_{0})$, one has $r(x,\eta)=O(\lb\eta\rb^{-l})$ for each $l\in\R$.

For our main result we need a lemma concerning a microlocal parametrix for $T$. Let $(\chi_{\nu})_{\nu\in\Theta_{k}}$ be as in the decoupling norm in \eqref{eq:decouplenorm}, and let $(\phi_{k})_{k=0}^{\infty}$  be a standard
Littlewood--Paley decomposition, as in the proof of Corollary \ref{cor:decouple}. Fix real-valued $(\tilde{\chi}_{\nu})_{\nu\in\Theta_{k}}$ and $(\tilde{\phi}_{k})_{k=0}^{\infty}\subseteq C^{\infty}_{c}(\Rn)$ with similar properties. More precisely, each $\tilde{\chi}_{\nu}$ is positively homogeneous of degree $0$, satisfies
\[
\supp(\tilde{\chi}_{\nu})\subseteq\{\eta\in\Rn\setminus\{0\}\mid |\hat{\eta}-\nu|\leq 2^{-k/2+2}\}
\]
and $\tilde{\chi}_{\nu}\chi_{\nu}=\chi_{\nu}$, and for all $\alpha\in\Z_{+}^{n}$ and $\beta\in\Z_{+}$ there exists a $\tilde{C}_{\alpha,\beta}\geq0$ independent of $k$ such that, if $2^{k-2}\leq |\eta|\leq 2^{k+2}$, then
\[
|(\hat{\eta}\cdot\partial_{\eta})^{\beta}\partial_{\eta}^{\alpha}\tilde{\chi}_{\nu}(\eta)|\leq \tilde{C}_{\alpha,\beta}2^{-k(|\alpha|/2+\beta)}
\]
for all $\nu\in \Theta_{k}$. The $(\tilde{\phi}_{k})_{k=0}^{\infty}$ have inverse Fourier transforms uniformly in $L^{1}(\Rn)$ and satisfy $\supp(\tilde{\phi}_{k})\subseteq\{\eta\in\Rn\mid 2^{k-2}\leq |\eta|\leq 2^{k+2}\}$ and $\tilde{\phi}_{k}\phi_{k}=\phi_{k}$ for each $k$.

\begin{lemma}\label{lem:parametrix}
Suppose that $T$ is non-characteristic at a point $c=(x_{0},t_{0},\zeta_{0},y_{0},\eta_{0})\in \Ca$. Then there exists a conic neighborhood $V\subseteq\Rn\times(\Rn\setminus\{0\})$ of $(y_{0},\eta_{0})$, an interval $J\subseteq\R$ containing $t_{0}$, and, for each $t\in J$, a compactly supported $S_{t}\in I^{-m}(\Rn,\Rn;\Ca_{t}^{-1})$ with the following properties.
\begin{enumerate}
\item\label{it:parametrix1} For all $p\in[1,\infty]$ and $s\in\R$ one has
\[
\sup\{\|\tilde{\phi}_{k}(D)\tilde{\chi}_{\nu}(D)S_{t}\|_{\La(W^{s,p}(\Rn),W^{s+m,p}(\Rn))}\mid t\in J, k\in\Z_{+},\nu\in\Theta_{k}\}<\infty;
\]
\item\label{it:parametrix2}
The pseudodifferential operator $R:\Sw'(\Rn)\to\Sw'(\Rn)$, given by
\[
Rf(x):=f(x)-\int_{J}S_{t}T_{t}f(x)\ud t
\]
for $f\in\Sw(\Rn)$ and $x\in\Rn$, satisfies $\WF'(R)\cap V=\emptyset$.
\end{enumerate}
\end{lemma}
\begin{proof}
The construction is analogous to that of a parametrix for a single FIO associated with a canonical graph, as in \cite[Theorem 4.2.5]{Duistermaat11}, but we will sketch the details since our setting is slightly different.

Let $V_{0}$ be as in \eqref{eq:elliptic}, and let $J\subseteq\R$ and $V'\subseteq\Rn\times(\R^{n}\setminus \{0\})$ be the projections of $V_{0}$ with respect to the $t$ variable and the $(y,\eta)$ variable, respectively. We may assume that the reciprocal of the symbol in \eqref{eq:principalsym} is bounded, uniformly in $t\in J$, by a multiple of $\lb\eta\rb^{-m}$.
Then, as in \cite[Theorem 4.2.5]{Duistermaat11}, one can cut off this reciprocal away from a small neighborhood of $c$ to construct, for each $t\in J$ and each conic subset $V\subseteq V'$ which is closed in $\Rn\times(\Rn\setminus\{0\})$, a properly supported $\tilde{S}_{t}\in I^{-m}(\Rn,\Rn;\Ca_{t}^{-1})$ such that $\WF'(I-\tilde{S}_{t}T_{t})\cap V=\emptyset$.
The relevant bounds for $\tilde{S}_{t}$ and the pseudodifferential operator $R_{t}:=I-\tilde{S}_{t}T_{t}$ are uniform in $t$. Moreover, since $a$ has compact spatial support and because each $\tilde{S}_{t}$ is properly supported, we may assume in addition that each $\tilde{S}_{t}$ is compactly supported. By assumption, $\Ca_{t}$ is a canonical graph, and therefore so is $\Ca_{t}^{-1}$.

Now set $S_{t}:=|J|^{-1}\tilde{S}_{t}$ for each $t\in J$. Then \eqref{it:parametrix2} holds by construction.
Moreover, \eqref{it:parametrix1} is the dual of a standard estimate from the $L^{p}$ theory of FIOs (see \eqref{eq:standardest}):
\[
\|S_{t}^{*}\tilde{\chi}_{\nu}(D)\tilde{\phi}_{k}(D)\|_{\La(W^{s-m,p}(\Rn),W^{s,p}(\Rn)}\lesssim 1,
\]
with an implicit constant which is uniform in $k$, $\nu$ and $t\in J$, due to Remark \ref{rem:constantsFIO}.
\end{proof}

\begin{remark}\label{rem:nocurv}
Note that we did not use anywhere in the proof that $\Ca$ satisfies the cinematic curvature condition. Indeed, we only used the specific form of the operator $T$ in \eqref{eq:Tstandard}, and the assumption that the canonical relations $\Ca_{t}$ are canonical graphs. In fact, the results in this appendix hold under these more general assumptions. The product structure of the symbol in Section \ref{subsec:decouple} also plays no role here, nor do the low frequencies. In fact, by the very statement of Proposition \ref{prop:conversedecouple}, it suffices to only require \eqref{eq:elliptic} when $|\eta|\geq 1/2$.
\end{remark}

We can now prove a converse to Proposition \ref{prop:decouple}, a version for general FIOs of \cite[Corollary 4.2]{Rozendaal22b}.
Recall that the wavefront set $\WF(f)\subseteq\Rn\times (\Rn\setminus\{0\})$ of an $f\in\Sw'(\Rn)$ is the complement of the set of $(y_{0},\eta_{0})$ such that there exists a pseudodifferential operator $S$, elliptic at $(y_{0},\eta_{0})$, with $Sf\in\Sw(\Rn)$.

\begin{proposition}\label{prop:conversedecouple}
Suppose that $T$ is non-characteristic at a point $(z_{0},\zeta_{0},y_{0},\eta_{0})\in \Ca$. Then there exists a conic neighborhood $V$ of $(y_{0},\eta_{0})$ and, for all $p\in(2,\infty)$ and $L\geq0$, a $C\geq0$ such that
\begin{equation}\label{eq:conversedecouple}
\|\phi_{k}(D)f\|_{\HT^{m-s(p),p}_{FIO}(\Rn)}\leq C\big(\|T\phi_{k}(D)f\|_{L^{p,k}_{\dec}(\R^{n+1})}+2^{-kL}\|\phi_{k}(D)f\|_{L^{2}(\Rn)}\big)
\end{equation}
for all $k\in\Z_{+}$ and $f\in L^{2}(\Rn)$ with $\WF(f)\subseteq V$.
\end{proposition}
\begin{proof}
By replacing $T$ by $T\lb D\rb^{-m}$, we may suppose that $m=0$.

Let $V$, $J$ and $R$ be as in Lemma \ref{lem:parametrix}, and write $f_{k}:=\phi_{k}(D)f=\tilde{\phi}_{k}(D)\phi_{k}(D)f$. Then
\begin{equation}\label{eq:tobound}
f_{k}=\sum_{\nu\in\Theta_{k}}\tilde{\phi}_{k}(D)\tilde{\chi}_{\nu}(D)R\chi_{\nu}(D)f_{k}+\sum_{\nu\in\Theta_{k}}\tilde{\phi}_{k}(D)\tilde{\chi}_{\nu}(D)(1-R)\chi_{\nu}(D)f_{k},
\end{equation}
and we will bound the $\HT^{-s(p),p}_{FIO}(\Rn)$ norm of each of the terms in \eqref{eq:tobound} separately.

To deal with the first term on the right-hand side of \eqref{eq:tobound}, let $S$ be a pseudodifferential operator of order zero, the symbol of which is supported away from $\WF'(R)$ and is equal to $1$ on $V$. Then $\WF'(R)\cap \WF'(S)=\emptyset$ and, for each $r\in\R$,
\[
\|(1-S)\chi_{\nu}(D)f_{k}\|_{W^{r,2}(\Rn)}\lesssim \|\chi_{\nu}(D)f_{k}\|_{W^{-L-(n-1)/2,2}(\Rn)}
\]
uniformly in $k$ and $\nu$. Hence \eqref{eq:Sobolev} and a standard Sobolev embedding yield, for some $r\in\R$,
\begin{align*}
&\Big\|\sum_{\nu\in\Theta_{k}}\tilde{\phi}_{k}(D)\tilde{\chi}_{\nu}(D)R\chi_{\nu}(D)f_{k}\Big\|_{\HT^{-s(p),p}_{FIO}(\Rn)}\\
&\lesssim \sum_{\nu\in\Theta_{k}}\|\tilde{\phi}_{k}(D)\tilde{\chi}_{\nu}(D)R\chi_{\nu}(D)f_{k}\|_{W^{r,2}(\Rn)}\\
&\lesssim \sum_{\nu\in\Theta_{k}}\|RS\chi_{\nu}(D)f_{k}\|_{W^{r,2}(\Rn)}+\sum_{\nu\in\Theta_{k}}\|R(1-S)\chi_{\nu}(D)f_{k}\|_{W^{r,2}(\Rn)}\\
&\lesssim \sum_{\nu\in\Theta_{k}}\|\chi_{\nu}(D)f_{k}\|_{W^{-L-(n-1)/2,2}(\Rn)}\lesssim 2^{k(n-1)/2}\|f_{k}\|_{W^{-L-(n-1)/2,2}(\Rn)}\\
&\eqsim 2^{-kL}\|f_{k}\|_{L^{2}(\Rn)},
\end{align*}
where we also used that the $\tilde{\phi}_{k}(D)$ and $\tilde{\chi}_{\nu}(D)$ have kernels uniformly in $L^{1}(\Rn)$.

On the other hand, there exists a constant $N\geq0$, independent of $k$ and $\mu\in\Theta_{k}$, such that $\chi_{\mu}\tilde{\chi}_{\nu}\neq 0$ for at most $N$ elements $\nu\in \Theta_{k}$. Hence one can use \cite[Proposition 4.1]{Rozendaal22b} and H\"{o}lder's inequality to write
\begin{align*}
&\Big\|\sum_{\nu\in\Theta_{k}}\tilde{\phi}_{k}(D)\tilde{\chi}_{\nu}(D)(1-R)\chi_{\nu}(D)f_{k}\Big\|_{\HT^{-s(p),p}_{FIO}(\Rn)}\\
&\eqsim \Big(\sum_{\mu\in\Theta_{k}}\Big\|\sum_{\nu\in\Theta_{k}}\chi_{\mu}(D)\tilde{\phi}_{k}(D)\tilde{\chi}_{\nu}(D)(1-R)\chi_{\nu}(D)f_{k}\Big\|_{L^{p}(\Rn)}^{p}\Big)^{1/p}\\
&\leq \Big(\sum_{\mu\in\Theta_{k}}\Big(\sum_{\nu\in\Theta_{k}}\|\chi_{\mu}(D)\tilde{\phi}_{k}(D)\tilde{\chi}_{\nu}(D)(1-R)\chi_{\nu}(D)f_{k}\|_{L^{p}(\Rn)}\Big)^{p}\Big)^{1/p}\\
&\lesssim \Big(\sum_{\mu\in\Theta_{k}}\sum_{\nu\in\Theta_{k}}\|\chi_{\mu}(D)\tilde{\phi}_{k}(D)\tilde{\chi}_{\nu}(D)(1-R)\chi_{\nu}(D)f_{k}\|_{L^{p}(\Rn)}^{p}\Big)^{1/p}\\
&\lesssim \Big(\sum_{\nu\in\Theta_{k}}\|\tilde{\phi}_{k}(D)\tilde{\chi}_{\nu}(D)(1-R)\chi_{\nu}(D)f_{k}\|_{L^{p}(\Rn)}^{p}\Big)^{1/p}\\
&=\Big(\sum_{\nu\in\Theta_{k}}\Big\|\int_{J}\tilde{\phi}_{k}(D)\tilde{\chi}_{\nu}(D)S_{t}T_{t}\chi_{\nu}(D)f_{k}\ud t\Big\|_{L^{p}(\Rn)}^{p}\Big)^{1/p}.
\end{align*}
Now Lemma \ref{lem:parametrix} \eqref{it:parametrix1}, together with H\"{o}lder's inequality and the fact that $J$ is a bounded interval, allows us to bound the quantity on the final line by
\begin{align*}
&\Big(\sum_{\nu\in\Theta_{k}}\Big(\int_{J}\|\tilde{\phi}_{k}(D)\tilde{\chi}_{\nu}(D)S_{t}T_{t}\chi_{\nu}(D)f_{k}\|_{L^{p}(\Rn)}\ud t\Big)^{p}\Big)^{1/p}\\
&\lesssim \Big(\sum_{\nu\in\Theta_{k}}\Big(\int_{J}\|T_{t}\chi_{\nu}(D)f_{k}\|_{L^{p}(\Rn)}\ud t\Big)^{p}\Big)^{1/p}\lesssim \Big(\sum_{\nu\in\Theta_{k}}\int_{J}\|T_{t}\chi_{\nu}(D)f_{k}\|_{L^{p}(\Rn)}^{p}\ud t\Big)^{1/p}\\
&\leq \|Tf_{k}\|_{L^{p,k}_{\dec}(\R^{n+1})},
\end{align*}
as required.
\end{proof}

\begin{remark}\label{rem:regularityf}
The reason to formulate \eqref{eq:conversedecouple} using the $\phi_{k}(D)f$, instead of a general $f$ with $\supp(\wh{f}\,)\subseteq\{\eta\in\Rn\mid 2^{k-1}\leq |\eta|\leq 2^{k+1}\}$, is that a function with the latter property is necessarily smooth. Hence, in order for the condition $\WF(f)\subseteq V$ to be useful, we  dyadically localize a single function $f$ in frequency, and it is the uniformity in $k$ which is relevant. Note also that, when deriving local smoothing estimates, Theorem \ref{thm:BeHiSo} is applied to functions of the form $\phi_{k}(D)f$.

For the same reason, one could replace the a priori assumption that $f\in L^{2}(\Rn)$ by $f\in W^{s,2}(\Rn)$ for any $s\in\R$. On the other hand, we work with square integrable functions to match the setting of Theorem \ref{thm:BeHiSo}. Doing so limits us to $p>2$, because of the use of a standard Sobolev embedding. The same proof as above would work for any $p\in[1,\infty]$ upon replacing $L^{2}(\Rn)$ by $L^{p}(\Rn)$.
\end{remark}

\begin{remark}\label{rem:wavefront}
One could attempt to choose the set $V$ as large as possible, which would lead to a stronger statement in Proposition \ref{prop:conversedecouple}. Moreover, one could work with a more refined notion of wavefront set, taking into account only microlocalized $\HT^{s,p}_{FIO}(\Rn)$ regularity. For the sake of simplicity, we will not do so here. Note also that, since the symbol $a$ has compact spatial support, the operator $T$ always has characteristic points.
\end{remark}

\bibliographystyle{plain}
\bibliography{Bibliography}

\end{document}